\documentclass{amsart}
\usepackage{a4}

\usepackage{hyperref}
\usepackage[italian, english]{babel}

\usepackage{mathtools} % per la freccia di inclusione \xhookrightarrow{}

\theoremstyle{definition}
\newtheorem{definition}{Definition}[section]
\newtheorem{theorem}[definition]{Theorem}
\newtheorem{lemma}[definition]{Lemma}
\newtheorem{proposition}[definition]{Proposition}

\newtheorem{remark}[definition]{Remark}

\usepackage{xcolor} % per scrivere colorato con {\color{nome colore}testo}
\usepackage[font={small,it}]{caption}

\usepackage{tikz-cd} % per i diagrammi commutativi
\usepackage{enumerate} % per numerare coi numeri romani con \begin{enumerate}[(I)]

\begin{document}

\title[On the rational EF 3-folds discovered by Fano]{On the rational Enriques-Fano threefolds discovered by Fano}
\author[V. Martello]{Vincenzo Martello}
\address{Dipartimento di Matematica, Universit\`{a} della Calabria, Arcavacata di Rende (CS)}
\email{vincenzomartello93@gmail.com}
%\date{\today}

%\thanks{}
%\subjclass{}
%\keywords{enriques fano threefolds, castelnuovo, conjecture, elliptic ruled surface.}
%\dedicatory{}
%\commby{}

\begin{abstract}
It was Fano who first classified Enriques-Fano threefolds. However his arguments appear to contain several gaps. In this paper, we will verify some of his assertions through the use of modern techniques.
\end{abstract}

\maketitle
%\tableofcontents

\section{Introduction}

An \textit{Enriques-Fano threefold} is a normal threefold $W$ endowed with a complete linear system $\mathcal{L}$ of ample Cartier divisors such that the general element $S\in \mathcal{L}$ is an Enriques surface and such that $W$ is not a \textit{generalized cone} over $S$, i.e., $W$ is not obtained by contraction of the negative section on the $\mathbb{P}^1$-bundle $\mathbb{P}(\mathcal{O}_{S}\oplus \mathcal{O}_{S}(S))$ over $S$. The linear system $\mathcal{L}$ defines a rational map $\phi_{\mathcal{L}} : W \dashrightarrow \mathbb{P}^{p}$, where $p:=\frac{S^3}{2}+1$ is called the \textit{genus} of $W$ and $2\le p\le 17$ (see \cite{KLM11} and \cite{Pro07}). 
A non-degenerate threefold $W'\subset \mathbb{P}^{N}$, whose general hyperplane section $S'$ is an Enriques surface and such that $W'$ is not a cone over $S'$, is called Enriques-Fano threefold, too: indeed, in this case, it is enough to take its normalization $\nu : W \to W'$ to obtain an Enriques-Fano threefold in the general sense, that is $(W, \mathcal{L}:=|\mathcal{O}_{W}(\nu^* S')|)$.
%If the elements of $\mathcal{L}$ are very ample divisors, then $W$ is embedded in $\mathbb{P}^p$ via $\phi_{\mathcal{L}}$ as a 
%%(possibly non-normal) 
%non-degenerate threefold whose general hyperplane section is an Enriques surface.
%We will refer to the Enriques-Fano threefolds of the above authors, respectively, as \textit{F-EF 3-folds}, \textit{BS-EF 3-folds}, \textit{P-EF 3-folds} and \textit{KLM-EF 3-fold}.
It is known that any Enriques-Fano threefold is singular with isolated canonical singularities (see \cite[Lemma 3.2]{CoMu85} and \cite{Ch96}). 
Though the classification of Enriques-Fano threefolds still accounts for an open question, some examples have been found by several authors. 
%Fano found examples of genus $p=4,6,7,9,13$ (see \cite{Fa38}); Bayle (and, in a similar and independent way, Sano) found fourteen examples of genus $2\le p\le 10$ and $p=13$ (see \cite{Ba94} and \cite{Sa95}); Prokhorov found two of genus $p=13,17$ (see \cite[\S 3]{Pro07}) and finally Knutsen-Lopez-Mu\~{n}oz found one of genus $p=9$ (see \cite[\S 13]{KLM11}). 
Under the assumption that the singularities are terminal cyclic quotients, Enriques-Fano threefolds were classified by Bayle and Sano (see \cite{Ba94} and \cite{Sa95}): they are fourteen and they have genus $2\le p \le 10$ or $p=13$. More generally, an Enriques-Fano threefold with only terminal singularities is a limit of some found by Bayle and Sano (see \cite[Main Theorem 2]{Mi99}). Instead, only a few examples of Enriques-Fano threefolds with non-terminal canonical singularities are known: two of genus $p=13, 17$ found by Prokhorov (see \cite[Proposition 3.2, Remark 3.3]{Pro07}) and one of genus $p=9$ found by Knutsen, Lopez and Mu\~{n}oz (see \cite[\S 13]{KLM11}).

The first to deal with the classification problem was Fano (see \cite{Fa38}).
He found five Enriques-Fano threefolds: one of genus $4$ (see \cite[\S 10]{Fa38}), which is non-rational (see \cite{P-BV83}), and four of genus $p=6,7,9,13$ (see \cite[\S 3,4,7,8]{Fa38}), which are rational. However, in his paper there are many gaps, as Conte and Murre showed in \cite{CoMu85}. Indeed, Fano often stated some of his results while failing to provide a real proof. 

By using blow-ups techniques, we will verify that the images of the rational map defined by the following linear systems on $\mathbb{P}^3$ actually are rational Enriques-Fano threefolds with eight quadruple points, as Fano said:
%in his paper \cite{Fa38}: 
\begin{itemize}
\item[(i)] the linear system $\mathcal{S}$ of the sextic surfaces having double points along the six edges of a tetrahedron; 
\item[(ii)] the linear system $\mathcal{K}$ of the septic surfaces having double points along the six edges of two trihedra; 
\item[(iii)] the linear system $\mathcal{X}$ of the sextic surfaces having double points along the six edges of a tetrahedron and containing a plane cubic curve intersecting each edge at one point; 
\item[(iv)] the linear system $\mathcal{P}$ of the septic surfaces with double points along three twisted cubics having five points in common.
\end{itemize}
%We will explicitly show this for the linear system $\mathcal{P}$, since it is the most complex one to deal with.
We will start with the classical case, i.e. the one of $\mathcal{S}$, in order to have a model to refer to, and then we will continue with the lesser known ones. 
We will also verified that the singular points of these threefolds are associated in the way imposed by Fano: two distinct singular points of an Enriques-Fano threefold $W$ are said to be \textit{associated} if the line joining them is contained in $W$.
We will work over the field $\mathbb{C}$ of the complex numbers. 
For some results we will use the software Macaulay2: in these cases we will work over a finite field (we will choose $\mathbb{F}_n := \mathbb{Z}/n\mathbb{Z}$ with $n=10000019$). In Appendix~\ref{app:code} we will collect the input codes used in Macaulay2.

\subsection*{Acknowledgment}
The results of this paper are contained in my PhD-thesis. I would like to thank my main advisors C. Ciliberto and C. Galati and my co-advisor A.L. Knutsen for our stimulating conversations and for providing me very useful suggestions.

\section{Terminology}

In the next sections we will use some known facts about the blow-ups of threefolds.
We refer to \cite[Chap 4, \S 6]{GH} and \cite[Lemma 2.2.14]{IsPro99} for more details. 
Furthermore, let us give some notation. 
If $p$ is a smooth point of a projective variety $X$, we will denote \textit{the tangent space to $X$ at $p$} by the symbol $T_p X$; if $p$ is a singular point of a projective variety $X$, we will denote \textit{the tangent cone to $X$ at $p$} by the symbol $TC_p X$. Finally, we recall the following definition.

\begin{definition}\label{def:ordinarysingularities}
A surface in $\mathbb{P}^3$ has \textit{ordinary singularities} if it has at most the following singularities: a curve $\gamma$ of double points (that are generically the transverse intersection of two branches), with at most finitely many pinch points and with $\gamma$ having at most finitely many triple points as singularities, with three independent tangent lines, which are triple points also for the surface. 
\end{definition}

\section{The Enriques-Fano threefold of genus 13}\label{subsec:Fano13}

%
%\subsection{Construction}

Let us take a tetrahedron $T\subset \mathbb{P}^3$ with vertices $v_0$, $v_1$, $v_2$, $v_3$. Let $f_i$ be the face of $T$ opposite to the vertex $v_i$ and let us denote the edges of $T$ by $l_{ij}:=f_i\cap f_j$, for $0\le i<j\le 3$.
Let $\mathcal{S}$ be the linear system of the sextic surfaces of $\mathbb{P}^{3}$ double along the six edges of $T$.
Up to a change of coordinates, we can consider in $\mathbb{P}^3_{\left[ s_0:s_1:s_2:s_3 \right]}$ the tetrahedron $T=\{s_0s_1s_2s_3=0\}$ with faces $f_i = \{s_i=0\}$, for $0\le i\le 3$. The linear system $\mathcal{S}$ is defined by the zero locus of the following homogeneous polynomial
$$\lambda_0s_1^2s_2^2s_3^2+ \lambda_1s_0^2s_2^2s_3^2+\lambda_2s_0^2s_1^2s_3^2+\lambda_3s_0^2s_1^2s_2^2+ s_0s_1s_2s_3Q(s_0,s_1,s_2,s_3),$$
where $\lambda_0,\lambda_1,\lambda_2,\lambda_3\in \mathbb{C}$ and $Q(s_0,s_1,s_2,s_3)=\sum_{i\le j} q_{ij}s_is_j$ is a quadratic form (see \cite[p.635]{GH}).
Since $\dim H^0(\mathbb{P}^3, \mathcal{O}_{\mathbb{P}^3}(2)) = \binom{3+2}{2}$, then $\dim \mathcal{S} = 13$. 

\begin{remark}\label{rem:variabeleTC13}
Let $\Sigma$ be a general element of $\mathcal{S}$. By looking locally at the equation of $\mathcal{S}$, then we obtain the following two assertions, for distinct indices $i,j,k,h\in\{0,1,2,3\}$:
\begin{itemize}
\item[(i)] $\Sigma$ has triple points at the vertices of $T$ and $TC_{v_i}\Sigma=f_j\cup f_k \cup f_h$;
\item[(ii)] if $p\in l_{ij}$ with $p\ne v_k$ and $p\ne v_h$, then $TC_{p}\Sigma$ is the union of two variable planes containing $l_{ij}$, depending on the choice of the point $p$ and of the surface $\Sigma$, and coinciding for finitely many points $p$.
\end{itemize}
\end{remark}

\begin{lemma}\label{lem:birationalitanu13}
The rational map $\nu_{\mathcal{S}} : \mathbb{P}^{3} \dashrightarrow \mathbb{P}^{13}$ defined by $\mathcal{S}$ is birational onto the image.
\end{lemma}
\begin{proof}
It is sufficient to verify that the map defined by $\mathcal{S}$ on a general $\Sigma \in \mathcal{S}$ is birational onto the image, and this actually happens because $\mathcal{S}|_{\Sigma}$ contains a sublinear system which defines a birational map. Indeed, $\mathcal{S}$ contains a sublinear system $\overline{\mathcal{S}} \subset \mathcal{S}$ 
%which has a movable part, given by the quadric surfaces of $\mathbb{P}^3$, and 
whose fixed part is given by the tetrahedron $T$ and such that $\overline{\mathcal{S}}|_{\Sigma}$ coincides with the linear system on $\Sigma$ cut out by the quadric surfaces of $\mathbb{P}^3$.
\end{proof} 

\begin{remark}\label{rem:veryample13}
The proof of Lemma~\ref{lem:birationalitanu13} tells us that the linear system $\mathcal{S}$ is very ample outside the tetrahedron $T$.
%dato che contiene un sottosistema T piu quadriche e le quadriche sono molto ampie... 
So $\nu_{\mathcal{S}} : \mathbb{P}^{3} \dashrightarrow  \nu_{\mathcal{S}}(\mathbb{P}^3)\subset \mathbb{P}^{13}$ is an isomorphism outside $T$.
\end{remark}

\begin{theorem}\cite[\S 8]{Fa38}\label{thm:WF13 isEF}
Let $W_{F}^{13}$ be the image of the map $\nu_{\mathcal{S}} : \mathbb{P}^{3} \dashrightarrow \mathbb{P}^{13}$. Then $W_F^{13}$ is an Enriques-Fano threefold of genus $p=13$.
\end{theorem}

\begin{proof}
The idea of the proof is to blow-up $\mathbb{P}^3$ along the base locus of $\mathcal{S}$, until we obtain a smooth rational threefold $Y$ and a base point free linear system $\widetilde{\mathcal{S}}$ on $Y$. By Lemma~\ref{lem:birationalitanu13}, the new linear system $\widetilde{\mathcal{S}}$ will define a birational morphism $\nu_{\widetilde{\mathcal{S}}}: Y \to W_F^{13}\subset \mathbb{P}^{13}$. So as to obtain that $W_F^{13}$ is an Enriques-Fano threefold, it will be sufficient to verify that the general hyperplane section $S$ is an Enriques surface and that $W_F^{13}$ is not a cone on $S$. Furthermore, to obtain the genus $p=13$ of $W_F^{13}$ we will compute the degree of the threefold, which is $24 = \widetilde{\Sigma}^3 = \deg W_F^{13} = 2p-2$ for 
$\widetilde{\Sigma}\in \widetilde{\mathcal{S}}$.
The proof is divided into several steps, given by the Remarks~\ref{rem:alphap13},$\dots$,~\ref{rem:noContractionSpigoli13} and the Theorem~\ref{thm:WF13moderno} below.

We blow-up first $\mathbb{P}^3$ at the vertices of $T$, obtaining a smooth threefold $Y'$ and a birational morphism $bl' : Y'\to \mathbb{P}^{3}$ with exceptional divisors
$E_i := (bl')^{-1}(v_i)$, for $0\le i \le 3$. 
Let $\mathcal{S}'$ be the strict transform of $\mathcal{S}$ and let us denote by $H$ the pullback on $Y'$ of the hyperplane class on $\mathbb{P}^{3}$. Then an 
element of $\mathcal{S}'$ is linearly equivalent to $6H-3\sum_{i=0}^{3}E_i$.
Let $\widetilde{f}_i$ be the strict transform of the face $f_i$, for $0\le i \le 3$. We denote by $\gamma_{ij}:=E_i\cap \widetilde{f}_j$ the line cut out by $\widetilde{f}_j$ on $E_i$, for $0\le i<j \le 3$. We have that $\gamma_{ij}$ is a $(-1)$-curve on $\widetilde{f}_j$. If $\Sigma'$ is the strict transform of a general $\Sigma\in\mathcal{S}$, then $\Sigma'\cap E_i = \bigcup_{\substack{j=0 \\ j\ne i}}^3 \gamma_{ij}$, for all $0\le i \le 3$, and $\Sigma'$ is smooth at a general point of $\gamma_{ij}$ (see Remark~\ref{rem:variabeleTC13}). 
The base locus of $\mathcal{S}'$ is now given by the union of the strict transforms $\widetilde{l}_{ij}$ of the six edges of $T$ (along which a general $\Sigma'\in \mathcal{S}'$ has double points) and the $12$ lines $\gamma_{ij}$ (see Remark~\ref{rem:variabeleTC13}).
Let us blow-up the strict transforms of the edges of $T$: we obtain a smooth threefold $Y''$ and a birational morphism $bl'' : Y'' \to Y'$ with exceptional divisors 
$$(bl'')^{-1}(\widetilde{l}_{ij})=: F_{ij}\cong \mathbb{P}(\mathcal{N}_{\widetilde{l}_{ij}|Y'}) \cong \mathbb{P}(\mathcal{O}_{\mathbb{P}^1}(-1)\oplus \mathcal{O}_{\mathbb{P}^1}(-1))\cong \mathbb{F}_0,$$ for $0\le i<j \le 3$.
This blow-up has no effect on $\widetilde{f}_i$, for $0\le i \le 3$, so, by abuse of notation, we will use the same symbol to indicate its strict transform on $Y''$. 

\begin{remark}\label{rem:alphap13}
Let $\widetilde{E}_i$ be the strict transform of $E_i$ and let us consider the curve $\alpha_{kij}:=\widetilde{E}_k\cap F_{ij}$, where $i,j,k$ are distinct indices in $\{0,1,2,3\}$ and $i<j$. Since $\alpha_{kij}$ is a $(-1)$-curve on $\widetilde{E}_k$ and it is a fibre on $F_{ij}$, then we have that
$F_{ij}^2\cdot \widetilde{E}_k=\alpha_{kij}^2|_{\widetilde{E}_k} = -1$ and $\widetilde{E}_k^2\cdot F_{ij}=\alpha_{kij}^2|_{F_{ij}} = 0$.
\end{remark}

Let $\mathcal{S}''$ be the strict transform of $\mathcal{S}'$: an 
%general 
element of $\mathcal{S}''$ is linearly equivalent to $6H-3\sum_{i=0}^3 \widetilde{E}_i-2\sum_{0\le i < j \le 3} F_{ij}$, where $H$ denotes the pullback $bl''^* H$, by abuse of notation. The base locus of $\mathcal{S}''$ is given by the disjoint union of the strict transforms $\widetilde{\gamma}_{ij}$ of the 12 lines $\gamma_{ij}$, for $i,j\in\{0,1,2,3\}$ and $i\ne j$ (see Remark~\ref{rem:variabeleTC13}).

\begin{remark}\label{rem:primeE^3p13}
Let $l_k$ be the linear equivalence class of the lines of $E_k\cong \mathbb{P}^2$: then $E_k|_{E_k}\sim -l_k$ (see \cite[Chap 4, \S 6]{GH} and \cite[Lemma 2.2.14]{IsPro99}). Let $L_k$ be the strict transform of $l_k$ via $bl'' |_{\widetilde{E}_k} : \widetilde{E}_k \to E_k$. 
Since $bl''^* (E_k)=\widetilde{E}_k$, then $\widetilde{E}_k|_{\widetilde{E}_k}\sim -L_k$ and $\widetilde{E}_k^3 = 1$.
\end{remark}

\begin{remark}\label{rem:-1curvesSigma}
By construction we have that $\widetilde{\gamma}_{ij}^2|_{\widetilde{E}_{i}}=-1$ and $\widetilde{\gamma}_{ij}^2|_{\widetilde{f}_{j}}=-1$, for $i,j\in\{0,1,2,3\}$ and $i\ne j$. We also have that $\widetilde{\gamma}_{ij}|_{\Sigma''}=-1$, where $\Sigma''$ is the strict transform on $Y''$ of a general element $\Sigma \in \mathcal{S}$. Indeed,
%since these twelve curves are disjoint, then 
considering that these twelve curves are disjoint, it follows that
$(\Sigma''\cap \widetilde{E}_i)^2|_{\Sigma''}= \sum_{\substack{j=0 \\ j\ne i}}^3 \widetilde{\gamma}_{ij}^2|_{\Sigma''}$, for all $0\le i \le 3$. On the other hand we have that $(\Sigma''\cap \widetilde{E}_i)^2|_{\Sigma''} = \widetilde{E}_i^2\cdot \Sigma'' = -3$ (see Remarks~\ref{rem:alphap13},~\ref{rem:primeE^3p13}). Thus, $(\widetilde{\gamma}_{ij})^2|_{\Sigma''}=-1$, since the curves $\widetilde{\gamma}_{ij}$ behave in the same way.
\end{remark}

Finally let us consider $bl''' : Y \to Y''$ the blow-up of $Y''$ along the twelve curves $\widetilde{\gamma}_{ij}$, for $i,j\in\{0,1,2,3\}$ and $i\ne j$, with exceptional divisors $\Gamma_{ij}:=bl'''^{-1}(\widetilde{\gamma}_{ij})$. 
We denote by $\mathcal{E}_i$ the strict transform of $\widetilde{E}_i$, by $\mathcal{F}_{ij}$ the strict transform of $F_{ij}$ and by $\mathcal{H}$ the pullback of $H$, for $0\le i<j \le 3$. 

\begin{remark}\label{rem:Gamma^3p13}
We have that
$$\Gamma_{ij}=\mathbb{P}(\mathcal{N}_{\widetilde{\gamma}_{ij}|Y''})\cong \mathbb{P}(\mathcal{O}_{\widetilde{\gamma}_{ij}}(E_{i})\oplus \mathcal{O}_{\widetilde{\gamma}_{ij}}(\widetilde{f}_{j})) \cong \mathbb{P}(\mathcal{O}_{\mathbb{P}^1}(-1)\oplus \mathcal{O}_{\mathbb{P}^1}(-1))\cong \mathbb{F}_0$$
and that $\Gamma_{ij}^{3}=-\deg (\mathcal{N}_{\widetilde{\gamma}_{ij}|Y''})=2$ (see \cite[Chap 4, \S 6]{GH} and \cite[Lemma 2.2.14]{IsPro99}).
\end{remark}

\begin{remark}\label{rem:intersectionYp13}
Let us take three distinct indices $i,j,k\in\{0,1,2,3\}$: if $j<k$, then $\Gamma_{ij}$ intersects $\mathcal{F}_{jk}$ along a $\mathbb{P}^1$, which is a fibre on $\Gamma_{ij}$ and a $(-1)$-curve on $\mathcal{F}_{jk}$. Similarly $\mathcal{F}_{kj}^2 \cdot \Gamma_{ij}=0$ and $\Gamma_{ij}^2\cdot \mathcal{F}_{kj}=-1$ if $k<j$. 
We also observe that $\Gamma_{ij}$ intersects $\mathcal{E}_i$ along a $\mathbb{P}^1$ belonging to the other ruling of $\Gamma_i$, so we have $\mathcal{E}_i^2\cdot\Gamma_{ij}=0$. Furthermore, we still have $\Gamma_{ij}^2\cdot \mathcal{E}_i = -1$, since $bl''' : Y \to Y''$ has no effect on $\widetilde{E}_i$. For this reason we will denote $\Gamma_{ij}\cap \mathcal{E}_{i}$ by $\widetilde{\gamma}_{ij}$, by abuse of notation.
Let us suppose now $i<j$ and let us consider the strict transforms $\widetilde{\alpha}_{kij}$ of the curves $\alpha_{kij}$ defined in Remark~\ref{rem:alphap13}. Then we have that $\mathcal{F}_{ij}^2\cdot \mathcal{E}_k = \widetilde{\alpha}_{kij}^2|_{\mathcal{E}_{k}}=-1$ and $\mathcal{E}_k^2\cdot \mathcal{F}_{ij} = \widetilde{\alpha}_{kij}^2|_{\mathcal{F}_{ij}}=-2$.
Finally we recall that a general line of $\mathbb{P}^3$ does not intersect the edges of $T$ and that a general plane of $\mathbb{P}^3$ intersects each one of them at one point. Hence we have that $\mathcal{H}^2\cdot \mathcal{F}_{ij}=0$ and $\mathcal{F}_{ij}^2\cdot\mathcal{H}=-1$.
\end{remark}

\begin{remark}\label{rem:Ek^3p13}
By construction we have 
${bl'''}^*(\widetilde{E}_k) = \mathcal{E}_k + \sum_{\substack{i=0 \\ i\ne k} }^3\Gamma_{ki},$
for $0\le k \le 3$.
If $\mathcal{L}_k$ is the strict transform of $L_k$ via $bl''' |_{\mathcal{E}_k} : \mathcal{E}_k \to \widetilde{E}_k$, then we have that $-\mathcal{E}_{k}|_{\mathcal{E}_{k}} \sim \mathcal{L}_k + \sum_{\substack{i=0 \\ i\ne k}}^3\widetilde{\gamma}_{ki} \sim 4\mathcal{L}_k-2\sum_{\substack{0\le i<j\le 3 \\ i,j\ne k}}\widetilde{\alpha}_{kij}$  and $\mathcal{E}_k^3=4$ (see Remark~\ref{rem:primeE^3p13}).
\end{remark}

\begin{remark}\label{rem:Fij^3p13}
Let us fix four distinct indices $i,j,k,h\in\{0,1,2,3\}$ with $i<j$. 
By \cite[Lemma 2.2.14]{IsPro99} we have that 
$F_{ij}^3 = - \deg (\mathcal{N}_{\widetilde{l}_{ij}|Y'})=2$ (see also \cite[Chap 4, \S 6]{GH}). 
Since $bl'''^*(F_{ij})=\mathcal{F}_{ij}$, 
%e dalla birazionalità di bl'''
then we still have $\mathcal{F}_{ij}^3=2$. 
\end{remark}

Let $\widetilde{\Sigma}$ be the strict transform on $Y$ of an 
%general 
element of $\mathcal{S}''$: then
$$\widetilde{\Sigma}\sim 6\mathcal{H}-\sum_{i=0}^3 3\mathcal{E}_k-\sum_{0\le i< j\le 3}2\mathcal{F}_{ij}-\sum_{ \substack{ i,j=0 \\ i\ne j } }^3 4\Gamma_{ij}.$$
Let us take the linear system $\widetilde{\mathcal{S}}:=|\mathcal{O}_{Y}(\widetilde{\Sigma})|$ on $Y$. It is base point free and it defines a morphism $\nu_{\widetilde{\mathcal{S}}}: Y \to \mathbb{P}^{13}$ birational onto the image $W_F^{13}:= \nu_{\widetilde{\mathcal{S}}} (Y)$, which is a threefold of degree $\deg W_F^{13} = 24$. This follows by Lemma~\ref{lem:birationalitanu13} and by the fact that $\widetilde{\Sigma}^3=24$  
(use Remarks~\ref{rem:Gamma^3p13},~\ref{rem:intersectionYp13},~\ref{rem:Ek^3p13},~\ref{rem:Fij^3p13}).
%$$\widetilde{\Sigma}^3 = 216\mathcal{H}^3-27\sum_{i=0}^3 \mathcal{E}_k^3-8\sum_{0\le i< j\le 3}\mathcal{F}_{ij}^3-64\sum_{ \substack{ i,j=0 \\ i\ne j } }^3 \Gamma_{ij}^3-3(36\mathcal{H}^2)\cdot\Big(2\sum_{0\le i< j\le 3}\mathcal{F}_{ij}\Big)+$$
%$$-3\Big(9\sum_{i=0}^3 \mathcal{E}_k^2\Big)\cdot\Big(2\sum_{0\le i< j\le 3}\mathcal{F}_{ij}\Big)-3\Big(9\sum_{i=0}^3 \mathcal{E}_k^2\Big)\cdot\Big(4\sum_{ \substack{ i,j=0 \\ i\ne j } }^3 \Gamma_{ij}\Big)+3\Big(4\sum_{0\le i< j\le 3}\mathcal{F}_{ij}^2\Big)\cdot(6\mathcal{H})+$$
%$$-3\Big(4\sum_{0\le i< j\le 3}\mathcal{F}_{ij}^2\Big)\cdot\Big(3\sum_{i=0}^3 \mathcal{E}_k\Big)-3\Big(4\sum_{0\le i< j\le 3}\mathcal{F}_{ij}^2\Big)\cdot\Big(4\sum_{ \substack{ i,j=0 \\ i\ne j } }^3 \Gamma_{ij}\Big)-3\Big(16\sum_{ \substack{ i,j=0 \\ i\ne j } }^3 \Gamma_{ij}^2\Big)\cdot\Big(3\sum_{i=0}^3 \mathcal{E}_k\Big)+$$
%$$-3\Big(16\sum_{ \substack{ i,j=0 \\ i\ne j } }^3 \Gamma_{ij}^2\Big)\cdot\Big(2\sum_{0\le i< j\le 3}\mathcal{F}_{ij}\Big)-6\Big(3\sum_{i=0}^3 \mathcal{E}_k\Big)\cdot\Big(2\sum_{0\le i< j\le 3}\mathcal{F}_{ij}\Big)\cdot\Big(4\sum_{ \substack{ i,j=0 \\ i\ne j } }^3 \Gamma_{ij}\Big)=$$
%$$=216-27\cdot 4 \cdot 4 -8\cdot 6\cdot 2-64\cdot 12\cdot 2 + 0 - 3\cdot 9 \cdot 2\cdot 4\cdot 3 \cdot (-2)+0+3\cdot 4 \cdot 6 \cdot 6 \cdot (-1)+$$
%$$-3\cdot 4 \cdot 3 \cdot 6 \cdot 2 \cdot (-1)+0-3\cdot 16 \cdot 3 \cdot 12 \cdot (-1)- 3\cdot 16 \cdot 2 \cdot 12 \cdot 2 \cdot (-1)-6\cdot 3\cdot 2 \cdot 4 \cdot 4\cdot 3\cdot 2=$$
%$$= 216-432-96-1536+0+1296-432+432+0+1728+2304-3456 = 24.$$
Then we have the following diagram:

$$\begin{tikzcd}
Y \arrow[d, "bl'''"] \arrow[drrr, "\nu_{\widetilde{\mathcal{S}}}"] & & & \\
Y''  \arrow{r}{bl''} & Y' \arrow{r}{bl'} & \mathbb{P}^3 \arrow[dashrightarrow]{r}{\nu_{\mathcal{S}}} & W_F^{13} \subset \mathbb{P}^{13}.
\end{tikzcd}$$

\begin{remark}\label{rem:contractionInpoints13}
Since $bl''' : Y \to Y''$ has no effect on the divisor $\widetilde{f}_i$, for $0\le i \le 3$, we will continue to use the same notation to denote its strict transform. The eight divisors $\mathcal{E}_0$, $\mathcal{E}_1$, $\mathcal{E}_2$, $\mathcal{E}_3$, $\widetilde{f}_0$, $\widetilde{f}_1$, $\widetilde{f}_2$, $\widetilde{f}_3$ are contracted by $\nu_{\widetilde{\mathcal{S}}} : Y \to W_F^{13} \subset \mathbb{P}^{13}$ to points of $W_F^{13}$. Indeed, if $\widetilde{\Sigma}$ is a general element of $\widetilde{\mathcal{S}}$, then by construction we have $\widetilde{\Sigma}\cdot \mathcal{E}_i = 0 = \widetilde{\Sigma}\cdot \widetilde{f}_i$ for all $0\le i \le 3$.
\end{remark}

\begin{remark}\label{rem:nu13blowdown}
The morphism $\nu_{\widetilde{\mathcal{S}}} : Y \to W_{F}^{13} \subset \mathbb{P}^{13}$ blows-down the twelve exceptional divisors $\Gamma_{ij}$ to twelve curves of $W_{F}^{13}$.
This follows by the fact that $\widetilde{\Sigma}\cdot \Gamma_{ij} \ne 0$ and 
% dato che è una (-1)-curva su $\widetilde{\Sigma}$ 
$\widetilde{\Sigma}^2\cdot \Gamma_{ij} = 0$ for a general element $\widetilde{\Sigma}\in \widetilde{\mathcal{S}}$ and for all $i,j\in\{0,1,2,3\}$ with $i\ne j$ (use Remarks~\ref{rem:Gamma^3p13},~\ref{rem:intersectionYp13}).
%$$\widetilde{\Sigma}^2\cdot \Gamma_{ij} = \widetilde{\Sigma}\cdot \bigg(-3\mathcal{E}_i\cdot \Gamma_{ij}-\sum_{\substack{0\le x<y\le 3 \\ i\not\in \{x,y\},\, j\in \{x,y\}}}2(\mathcal{F}_{xy}\cdot \Gamma_{ij})-4\Gamma_{ij}^2\bigg) =$$
%$$=9\mathcal{E}_i^2\cdot \Gamma_{ij}+\sum_{\substack{0\le x<y\le 3 \\ i\not\in \{x,y\},\, j\in \{x,y\}}}6(\mathcal{E}_i\cdot\mathcal{F}_{xy}\cdot \Gamma_{ij})+12\mathcal{E}_i\cdot\Gamma_{ij}^2+ \sum_{\substack{0\le x<y\le 3 \\ i\not\in \{x,y\},\, j\in \{x,y\}}}6(\mathcal{E}_i\cdot\mathcal{F}_{xy}\cdot \Gamma_{ij})+$$
%$$+\sum_{\substack{0\le x<y\le 3 \\ i\not\in \{x,y\},\, j\in \{x,y\}}}4(\mathcal{F}_{xy}^2\cdot \Gamma_{ij})+\sum_{\substack{0\le x<y\le 3 \\ i\not\in \{x,y\},\, j\in \{x,y\}}}8(\mathcal{F}_{xy}\cdot \Gamma_{ij}^2)+12\mathcal{E}_i\cdot \Gamma_{ij}^2+ \sum_{\substack{0\le x<y\le 3 \\ i\not\in \{x,y\},\, j\in \{x,y\}}}8(\mathcal{F}_{xy}\cdot \Gamma_{ij}^2)+16\Gamma_{ij}^3=$$
%$$=0+6\cdot 2 \cdot 1 + 12 \cdot (-1)+6\cdot 2 \cdot 1 + 0 + 8\cdot 2 \cdot (-1) + 12 \cdot (-1)+ 8\cdot 2 \cdot (-1) + 16\cdot 2 = 0.$$ 
\end{remark}

\begin{remark}\label{rem:noContractionSpigoli13}
Let 
%$i,j,k,h$ be four distinct indices in $\{0,1,2,3\}$ such that 
$0\le i<j\le 3$
and let $\widetilde{\Sigma}$ be a general element of $\widetilde{\mathcal{S}}$. 
By Remarks~\ref{rem:intersectionYp13},~\ref{rem:Fij^3p13} we obtain
$\widetilde{\Sigma}^2\cdot \mathcal{F}_{ij} = 4 >0$.
Thus,
the curve $\widetilde{\Sigma}\cap \mathcal{F}_{ij}$ is not contracted by the rational map defined by $\widetilde{\mathcal{S}}|_{\widetilde{\Sigma}}$. 
\end{remark}

\begin{theorem}\label{thm:WF13moderno}
Let $S$ be a general hyperplane section of the threefold $W_F^{13}\subset \mathbb{P}^{13}$. Then $S$ is an Enriques surface and $W_F^{13}$ is not a cone over $S$.
\end{theorem}
\begin{proof}
A general hyperplane section $S$ of $W_F^{13}$ is the image of a general element $\widetilde{\Sigma}\in \widetilde{\mathcal{S}}$ via the morphism $\nu_{\widetilde{\mathcal{S}}} : Y \to W_{F}^{13}\subset \mathbb{P}^{13}$. Let us take $\Sigma'':=bl'''(\widetilde{\Sigma})\in\mathcal{S}''$. Since $bl''' : Y \to Y''$ has no effect on $\Sigma''$, then $\widetilde{\Sigma}\cap \Gamma_{ij}$ is still a $(-1)$-curve on $\widetilde{\Sigma}$, for all $i,j\in\{0,1,2,3\}$ and $i\ne j$ (see Remark~\ref{rem:-1curvesSigma}). 
Since $\nu_{\widetilde{\mathcal{S}}} |_{\widetilde{\Sigma}} : \widetilde{\Sigma} \to S$ is the blow-down of these twelve $(-1)$-curves (see Remarks~\ref{rem:veryample13},~\ref{rem:contractionInpoints13},~\ref{rem:nu13blowdown},~\ref{rem:noContractionSpigoli13}), then $S$ is the minimal desingularization of the corresponding $\Sigma:=bl'(bl''(bl'''(\widetilde{\Sigma})))\in \mathcal{S}$ (see \cite[p.621]{GH}). It is known that the minimal desingularization of a sextic surface $\Sigma \in \mathcal{S}$ is an Enriques surface (see \cite[p.275]{CoDo89}). It remains to show that $W_F^{13}$ is not a cone over $S$.
Since $Y$ is rational by construction, then $W_F^{13}$ is rational too. If $W_F^{13}$ were a cone, then it would be birational to $S\times \mathbb{P}^1$, for a general hyperplane section $S$ of $W_F^{13}$. Thus, $S$ would be unirational, which is a contradiction because $S$ is an Enriques surface.
\end{proof} 

By Theorem~\ref{thm:WF13moderno} we have that $W_F^{13}\subset \mathbb{P}^{13}$ is an Enriques-Fano threefold of genus $p=\frac{S^3}{2}+1 = \frac{\widetilde{\Sigma}^3}{2}+1=13$. 
\end{proof}

%\subsection{Singularities}

\begin{proposition}\label{prop:quadruplePointsp13}
The points $P_{i+1} : = \nu_{\widetilde{\mathcal{S}}}(\mathcal{E}_i)$ and $P_{i+1}' : = \nu_{\widetilde{\mathcal{S}}}(\widetilde{f}_{i})$, $0\le i \le 3$,
are quadruple points of $W_F^{13}$ whose tangent cone is a cone over a Veronese surface.
\end{proposition}
\begin{proof}
We recall that $\nu_{\widetilde{\mathcal{S}}}(\mathcal{E}_i)$ and $\nu_{\widetilde{\mathcal{S}}}(\widetilde{f}_{i})$ actually are points of $W_F^{13}$, for $0\le i \le 3$ (see Remark~\ref{rem:contractionInpoints13}).
%Let $\widetilde{\Sigma}$ be a general element of $\widetilde{\mathcal{S}}$.
Let us consider the sublinear system 
%$|\mathcal{O}_{Y}(\widetilde{\Sigma}-\mathcal{E}_k)|$ 
$(\mathcal{S}-\mathcal{E}_k)\subset \widetilde{\mathcal{S}}$ for a fixed $0\le k \le 3$. It corresponds to taking the hyperplane sections of $W_F^{13}\subset \mathbb{P}^{13}$ passing through the point $P_{k+1}$. The linear system 
$(\mathcal{S}-\mathcal{E}_k)|_{\mathcal{E}_k}$ coincides with $|\mathcal{O}_{\mathcal{E}_k}(-\mathcal{E}_k)|$, which is isomorphic to the linear system of the quartic plane curves on $E_k$ with nodes at the three points $E_k\cap \widetilde{l}_{ij}$ for $0\le i<j \le 3$ and $i,j\ne k$ (see Remark~\ref{rem:Ek^3p13}). By applying a quadratic transformation, we obtain the linear system of the conics, whose image is the Veronese surface.
Let us now focus on the hyperplane sections of $W_F^{13}\subset \mathbb{P}^{13}$ passing through $P_{i+1}'$, for a fixed $0\le i \le 3$. It corresponds to taking the sublinear system $\mathcal{S}_{i}$ of the sextic surfaces of $\mathcal{S}$ containing the face $f_i$. The movable part of $\mathcal{S}_{i} $ is given by the quintic surfaces $Q_i$ of $\mathbb{P}^3$ containing the three edges of $T$ contained in $f_i$ and with double points along the other three edges of $T$. Such a surface $Q_i$ cuts on $f_i$ a quintic curve given by the three edges of $T$ contained in $f_i$ and a variable conic. Let us denote by $\widetilde{\mathcal{S}}_i$ the strict transform on $Y$ of $\mathcal{S}_i$ and let $\widetilde{Q}_i$ be the strict transform on $Y$ of $Q_i$. Then $\widetilde{\mathcal{S}}_i|_{\widetilde{f}_i} \cong |\mathcal{O}_{\widetilde{f}_i}(\widetilde{Q}_i)| \cong |\mathcal{O}_{\mathbb{P}^2}(2)|$, whose image is the Veronese surface.
\end{proof} 

Since $\nu_{\mathcal{S}} : \mathbb{P}^{3} \dashrightarrow  W_F^{13} \subset \mathbb{P}^{13}$ is an isomorphism outside $T$ (see Remark~\ref{rem:veryample13}), then $P_1$, $P_2$, $P_3$, $P_4$, $P_1'$, $P_2'$, $P_3'$ and $P_4'$ are the only singular points of $W_F^{13}$ (see Remarks~\ref{rem:contractionInpoints13},~\ref{rem:nu13blowdown},~\ref{rem:noContractionSpigoli13}).

\begin{theorem}\label{thm:associatedo13}
Each singular point of $W_F^{13}$ is associated with \textit{at least} $m=3$ of the other singular points.
\end{theorem}
\begin{proof}
We know that the twelve exceptional divisors of $bl''' : Y \to Y''$ are mapped by $\nu_{\widetilde{\mathcal{S}}} : Y \to W_F^{13}\subset \mathbb{P}^{13}$ to curves of $W_F^{13}$ (see Remark~\ref{rem:nu13blowdown}). 
Let us fix
two indices $i,j\in\{0,1,2,3\}$ with $j\ne i$.
Let $\widetilde{\Sigma}$ be a general element of $\widetilde{\mathcal{S}}$: by construction we have that $\widetilde{\Sigma}\cap \Gamma_{ij}$ belongs to one of the two rulings of $\Gamma_{ij}\cong \mathbb{F}_0$. Then $\widetilde{\mathcal{S}}|_{\Gamma_{ij}}\cong \mathbb{P}^1$ and so $\nu_{\widetilde{\mathcal{S}}}(\Gamma_{ij})\subset W_F^{13}$ is a line. In particular $\nu_{\widetilde{\mathcal{S}}}(\Gamma_{ij})$ joins the points $P_{i+1}=\nu_{\widetilde{\mathcal{S}}}(\mathcal{E}_i)$ and $P_{j+1}'=\nu_{\widetilde{\mathcal{S}}} (\widetilde{f}_j)$, since $\Gamma_{ij}\cap \mathcal{E}_i \ne \emptyset$ and $\Gamma_{ij}\cap \widetilde{f}_j \ne \emptyset$.
\end{proof}

\begin{remark}
Thanks to a computational anaylisis with Macaulay2, we see that each singular point of $W_F^{13}$ is associated with \textit{exactly} $m=3$ of the other singular points, and that the lines joining them and contained in $W_F^{13}$ are just the ones of proof of Theorem~\ref{thm:associatedo13} (see \S~\ref{code:fano13} of Appendix~\ref{app:code}).
\end{remark}

\section{The Enriques-Fano threefold of genus 9}\label{subsec:Fano9}

%\subsection{Construction}

Let us take the trihedron $T\subset \mathbb{P}^3$ with vertex $v$, faces $f_i$ and edges $l_{ij} := f_i \cap f_j$, and the trihedron $T'\subset \mathbb{P}^3$ with vertex $v'$, faces $f_i'$ and edges $l_{ij}' := f_i' \cap f_j'$, for $1\le i <j \le 3$. 
Let us consider the linear system $\mathcal{K}$ of the septic surfaces of $\mathbb{P}^{3}$ double along the six edges of the two trihedra $T$ and $T'$.

\begin{remark}\label{rem:rij contained in K}
A septic surface $K\in \mathcal{K}$ contains the nine lines $r_{ij}:=f_i\cap f_j'$, for $i,j\in\{1,2,3\}$. Assume the contrary: then, by Bezout's Theorem, $K\cap r_{ij}$ is given by $7$ points. Furthermore, each line $r_{ij}$ intersects two edges of $T$ contained in $f_i$ and two edges of $T'$ contained in $f_j'$. Hence $r_{ij}$ is a line through four double points of $K$.
We obtain that $K\cap r_{ij}$ contains at least $8$ points, counted with multiplicity, which is a contradiction. Thus, it must be $r_{ij} \subset K$. 
\end{remark}

\begin{proposition}\label{prop:dimK=9}
The linear system $\mathcal{K}$ is defined by the zero locus of the following homogeneous polynomial of degree seven
$$F(s_0,s_1,s_2,s_3) = f_1f_2f_3f_1'f_2'f_3'(\lambda_0s_0+\lambda_1s_1+\lambda_2s_2+\lambda_3s_3)+$$
$$+f_1'f_2'f_3'(\lambda_4 f_3^2f_2^2+\lambda_5 f_1^2f_3^2+\lambda_6 f_1^2f_2^2)+f_1f_2f_3(\lambda_7 {f_3'}^2{f_2'}^2+\lambda_8 {f_1'}^2{f_3'}^2+\lambda_9 {f_1'}^2{f_2'}^2),$$
where $\lambda_0,\dots ,\lambda_9\in \mathbb{C}$ and where $f_i$ and $f_i'$ denote, by abuse of notation, the linear homogeneous polynomials defining, respectively, the faces $f_i$ and $f_i'$, for $1\le i \le 3$. The linear system $\mathcal{K}$ therefore has $\dim \mathcal{K} = 9$.
\end{proposition} 
\begin{proof}
Let $F\in \mathbb{C}\left[ s_0:s_1:s_2:s_3 \right]$ be the homogeneous polynomial of degree $7$ defining a general element $K$ of $\mathcal{K}$. We recall that the intersection of an irreducible septic surface of $\mathbb{P}^3$ with a plane is a septic curve: in particular, $K$ intersects each face $f_i$ of $T$ along the septic curve given by the two double edges contained in that face plus the three lines $r_{ij}$, for $1\le j \le 3$. The same happens with the faces of $T'$. 
This implies that it must be $K\cap f_i=\{f_1'f_2'f_3'f_k^2f_h^2=0,\, f_i = 0\}=2l_{ik}+2l_{ih}+\sum_{j=1}^3r_{ij}$ and $K\cap f_i'=\{f_1f_2f_3{f_k'}^2{f_h'}^2=0,\, f_i' = 0\}= 2l_{ik}'+2l_{ih}'+\sum_{j=1}^3r_{ji}$, for distinct indices $i,k,h\in\{1,2,3\}$. 
Then it must be
$$F(s_0,s_1,s_2,s_3) = f_1g_6(s_0,s_1,s_2,s_3)+\lambda_4 f_1'f_2'f_3'f_3^2f_2^2,$$
where $\lambda_4\in \mathbb{C}$ and $g_6$ is a homogeneous polynomial of degree $6$ such that
$$g_6(s_0,s_1,s_2,s_3) = f_2g_5(s_0,s_1,s_2,s_3)+\lambda_5 f_1'f_2'f_3'f_1f_3^2,$$
where $\lambda_5\in \mathbb{C}$ and $g_5$ is a homogeneous polynomial of degree $5$ such that
$$g_5(s_0,s_1,s_2,s_3) = f_3g_4(s_0,s_1,s_2,s_3)+\lambda_6 f_1'f_2'f_3'f_1f_2,$$
where $\lambda_6\in \mathbb{C}$ and $g_4$ is a homogeneous polynomial of degree $4$ such that
$$g_4(s_0,s_1,s_2,s_3) = f_1'g_3(s_0,s_1,s_2,s_3)+\lambda_7 f_2'^2f_3'^2,$$
where $\lambda_7\in \mathbb{C}$ and $g_3$ is a homogeneous polynomial of degree $3$ such that
$$g_3(s_0,s_1,s_2,s_3) = f_2'g_3(s_0,s_1,s_2,s_3)+\lambda_8 f_1'f_3'^2,$$
where $\lambda_8\in \mathbb{C}$ and $g_2$ is a homogeneous polynomial of degree $2$ such that
$$g_2(s_0,s_1,s_2,s_3) = f_3'(\lambda_0s_0+\lambda_1s_1+\lambda_2s_2+\lambda_3s_3)+\lambda_9 f_1'f_2',$$
where $\lambda_0, \lambda_1, \lambda_2, \lambda_3, \lambda_9\in \mathbb{C}$.
So $F$ has the expression of the statement. Since $\{K\in \mathcal{K}| K\supset f_1\} = \{F=0 | \lambda_4 = 0\}$, then $\operatorname{codim}\left(\{K\in \mathcal{K}| K\supset f_1\},\mathcal{K}\right)=1$. 
Let us see that containing the six faces $f_1$, $f_2$, $f_3$, $f_1'$, $f_2'$, $f_3'$ imposes independent conditions: there exists a septic surface in $\mathcal{K}$ containing $f_1$ but not $f_2$, that is
$\{F=0 | \lambda_4 = 0, \lambda_5 \ne 0\}$;
there exists a septic surface in $\mathcal{K}$ containing $f_1$ and $f_2$ but not $f_3$, that is
$\{F=0 | \lambda_4 = \lambda_5 = 0, \lambda_6\ne 0\}$;
there exists a septic surface in $\mathcal{K}$ containing $f_1$, $f_2$ and $f_3$ but not $f_1'$, that is
$\{F=0 | \lambda_4 = \lambda_5 = \lambda_6 = 0, \lambda_7 \ne 0\}$;
there exists a septic surface in $\mathcal{K}$ containing $f_1$, $f_2$, $f_3$ and $f_1'$ but not $f_2'$, that is
$\{F=0 | \lambda_4 = \lambda_5 = \lambda_6 = \lambda_7 = 0, \lambda_8 \ne 0\}$;
there exists a septic surface in $\mathcal{K}$ containing $f_1$, $f_2$, $f_3$, $f_1'$, and $f_2'$ but not $f_3'$, that is
$\{F=0 | \lambda_4 = \lambda_5 = \lambda_6 = \lambda_7 = \lambda_8 = 0, \lambda_9 \ne 0\}.$
Thus, we obtain $\operatorname{codim}(\{K\in \mathcal{K}| K\supset T\cup T'\},\mathcal{K})=6$.
Furthermore, each element of $\{K\in \mathcal{K}| K\supset T\cup T'\}$ is of the form $T\cup T'\cup \pi$, where $\pi$ is a general plane of $\mathbb{P}^3$. Thus, we have $\dim \{K\in \mathcal{K}| K\supset T\cup T'\}= \dim |\mathcal{O}_{\mathbb{P}^3}(1)|=3$ and finally $\dim \mathcal{K}=3+6=9$.
\end{proof}

Let us consider the points mentioned in Remark~\ref{rem:rij contained in K}: they are
$q_{ijk}:=l_{ij}\cap r_{ik} = l_{ij}\cap r_{jk}$ and $q_{ijk}':=l_{ij}'\cap r_{ki}=l_{ij}'\cap r_{kj}$ for $i,j,k\in\{1,2,3\}$ with $i<j$. 
These points also represent the intersection points between the faces of a trihedron and the edges of the other trihedron. Indeed, we have that $q_{ijk} = l_{ij}\cap f_k'$ and $q_{ijk}' =l_{ij}'\cap f_k$.

\begin{remark}\label{rem:variabeleTC9}
Let $K$ be a general element of $\mathcal{K}$. By looking locally at the equation of $\mathcal{K}$ (see Proposition~\ref{prop:dimK=9}), then we find that:
\begin{itemize}
\item[(i)] $K$ has triple points at the vertices of $T$ and $T'$ and $TC_{v}K=\bigcup_{i=1}^3f_i$ and $TC_{v'}K=\bigcup_{i=1}^3f_i'$;
\item[(ii)] $TC_{q_{ijk}}K = f_i\cup f_j$ and $TC_{q_{ijk}'}=f_i'\cup f_j'$, for $i,j,k\in\{1,2,3\}$ with $i<j$;
\item[(iii)] if $p\in l_{ij}$, with $p\ne v$ and $p\ne q_{ijk}$, then $TC_{p}K$ is the union of two variable planes containing $l_{ij}$, depending on the choice of the point $p$ and of the surface $K$, and coinciding for finitely many points $p$. Similarly if $p\in l_{ij}'$, with $p\ne v'$ and $p\ne q_{ijk}'$, then $TC_{p}K$ is the union of two elements of $|\mathcal{I}_{l_{ij}'|\mathbb{P}^3}(1)|$ that depend on the choice of $p$ and $K$ and that can also coincide for finitely many points $p$;
\item[(iv)] $K$ is smooth along $r_{ik}$, except at the points contained in the edges of the two trihedra.
\end{itemize}
\end{remark}

\begin{lemma}\label{lem:birationalitanu9}
The rational map $\nu_{\mathcal{K}} : \mathbb{P}^{3} \dashrightarrow \mathbb{P}^{9}$ defined by $\mathcal{K}$ is birational onto the image.
\end{lemma}
\begin{proof}
It is sufficient to prove that the map defined by $\mathcal{K}$ on a general $K \in \mathcal{K}$ is birational onto the image. 
This actually happens because $\mathcal{K}|_{K}$ contains a sublinear system that defines a birational map. Indeed, $\mathcal{K}$ contains a sublinear system $\overline{\mathcal{K}} \subset \mathcal{K}$ 
whose fixed part is given by the two trihedra $T$ and $T'$ and such that $\overline{\mathcal{K}}|_{K}$ coincides with the linear system on $K$ cut out by the planes of $\mathbb{P}^3$.
\end{proof} 

\begin{remark}\label{rem:veryample9}
The proof of Lemma~\ref{lem:birationalitanu9} tells us that the linear system $\mathcal{K}$ is very ample outside the two trihedra $T$ and $T'$. So $\nu_{\mathcal{K}} : \mathbb{P}^{3} \dashrightarrow  \nu_{\mathcal{K}}(\mathbb{P}^3)\subset \mathbb{P}^{9}$ is an isomorphism outside $T\cup T'$.
\end{remark}

\begin{theorem}\cite[\S 7]{Fa38}\label{thm:WF9 isEF}
The image of $\nu_{\mathcal{K}} : \mathbb{P}^{3} \dashrightarrow \mathbb{P}^{9}$ is an Enriques-Fano threefold $W_{F}^{9}$ of genus $p=9$.
\end{theorem}

\begin{proof}

We aim at proving the theorem by means of using the approaches of the Theorem~\ref{thm:WF13 isEF}. In particular the proof is divided into several steps, given by the Remark~\ref{rem:selfalphap9}, the Proposition~\ref{prop:tildeKSmoothePa0}, the Remarks~\ref{rem:-1curvesK},$\dots$,~\ref{rem:uniqueTT'} and the Theorem~\ref{thm:S9isEnriques} below.

We blow-up first the vertices of the trihedra and the $18$ points $q_{ijk}$ and $q_{ijk}'$ for $i,j,k\in\{1,2,3\}$ and $i<j$. We obtain a smooth threefold $Y'$ and a birational morphism $bl' : Y '\to \mathbb{P}^{3}$ with exceptional divisors $E := (bl')^{-1}(v)$, $E' := (bl')^{-1}(v')$, $E_{ijk}:=(bl')^{-1}(q_{ijk})$, $E_{ijk}':=(bl')^{-1}(q_{ijk}')$. 
Let $\mathcal{K}'$ be the strict transform of $\mathcal{K}$ and let us denote by $H$ the pullback on $Y'$ of the hyperplane class on $\mathbb{P}^{3}$. Then an 
%general 
element of $\mathcal{K}'$ is linearly equivalent to $7H-3E-3E'-2\sum_{\substack{i,j,k=1 \\ i<j }}^3(E_{ijk}+E_{ijk}')$. Let $\widetilde{f}_i$ and $\widetilde{f}_i'$ be the strict transforms of the faces $f_i$ and $f_i'$, for $1\le i \le 3$. 
We denote by $\gamma_{i}:=E\cap \widetilde{f}_i$ the line cut out by $\widetilde{f}_i$ on $E$ and by $\gamma_{i}':=E'\cap \widetilde{f}_i'$ the one cut out by $\widetilde{f}_i'$ on $E'$. By construction, the curves $\gamma_{i}$ and $\gamma_{i}'$ are $(-1)$-curves respectively on $\widetilde{f}_i$ and $\widetilde{f}_i'$. If $K'$ is the strict transform of a general $K\in\mathcal{K}$, then $K'\cap E = \bigcup_{i= 0}^3 \gamma_{i}$ and $K'\cap E' = \bigcup_{i=0}^3 \gamma_{i}'$ and $K'$ is smooth at a general point of $\gamma_{i}$ and of $\gamma_{i}'$ (see Remark~\ref{rem:variabeleTC9}).
%(i). 
We also consider the lines $\lambda_{ijk,h}:=E_{ijk}\cap \widetilde{f}_h$ and $\lambda_{ijk,h}':=E_{ijk}'\cap \widetilde{f}_h'$, where $i,j,k\in\{1,2,3\}$ with $i<j$ and $h\in\{i,j\}$. They are $(-1)$-curves on the strict transforms of the faces containing them. Furthermore, we have that $K'\cap E_{ijk} =\bigcup_{h=i,j} \lambda_{ijk,h}$ and $K'\cap E_{ijk}' = \bigcup_{h=i,j}\lambda_{ijk,h}'$ (see Remark~\ref{rem:variabeleTC9}). 
%(ii). 
Let us consider the strict transforms $\widetilde{l}_{ij}$, $\widetilde{l}_{ij}'$ and $\widetilde{r}_{ij}$ of the lines $l_{ij}$, $l_{ij}'$ and $r_{ik}$, for $i,j,k\in\{1,2,3\}$ and $i<j$. Then the base locus of $\mathcal{K}'$ is given by the union of the six curves $\widetilde{l}_{ij}$, $\widetilde{l}_{ij}'$ (along which a general $K'\in\mathcal{K}'$ has double points), of the nine curves $\widetilde{r}_{ik}$, of the six lines $\gamma_i$, $\gamma_i'$, and of the $36$ lines $\lambda_{ijk,h}$, $\lambda_{ijk,h}'$ (see Remark~\ref{rem:variabeleTC9}).
%(iii).
Let us blow-up $Y'$ along the strict transforms of the edges of the trihedra and of the nine lines $r_{ij}$. We obtain a smooth threefold $Y''$ and a birational morphism $bl'' : Y'' \to Y'$ with exceptional divisors 
$$(bl'')^{-1}(\widetilde{l}_{ij})=: F_{ij}\cong \mathbb{P}(\mathcal{N}_{\widetilde{l}_{ij}|Y'}) \cong \mathbb{P}(\mathcal{O}_{\mathbb{P}^1}(-3)\oplus \mathcal{O}_{\mathbb{P}^1}(-3))\cong \mathbb{F}_0,$$
$$(bl'')^{-1}(\widetilde{l}_{ij}')=: F_{ij}'\cong \mathbb{P}(\mathcal{N}_{\widetilde{l}_{ij}'|Y'}) \cong \mathbb{P}(\mathcal{O}_{\mathbb{P}^1}(-3)\oplus \mathcal{O}_{\mathbb{P}^1}(-3))\cong \mathbb{F}_0,$$
$$(bl'')^{-1}(\widetilde{r}_{ij})=: R_{ij}\cong \mathbb{P}(\mathcal{N}_{\widetilde{r}_{ij}|Y'}) \cong \mathbb{P}(\mathcal{O}_{\mathbb{P}^1}(-3)\oplus \mathcal{O}_{\mathbb{P}^1}(-3))\cong \mathbb{F}_0.$$
This blow-up has no effect on $\widetilde{f}_i$ and $\widetilde{f}_i'$, for $1\le i \le 3$, so, by abuse of notation, we use the same symbols to indicate their strict transforms on $Y''$. 
Let us denote by $\widetilde{E}$, $\widetilde{E}'$, $\widetilde{E}_{ijk}$ and $\widetilde{E}_{ijk}'$ respectively the strict transforms of $E$, $E'$, $E_{ijk}$ and $E_{ijk}'$. 

\begin{remark}\label{rem:selfalphap9}
Let us take the curves
$\alpha_{ij}:=\widetilde{E}\cap F_{ij}$, $\alpha_{ij}':=\widetilde{E}'\cap F_{ij}'$, $\alpha_{ijk}:=\widetilde{E}_{ijk}\cap F_{ij}$, $\alpha_{ijk}':=\widetilde{E}_{ijk}'\cap F_{ij}'$, $\alpha_{ijk,h}:=\widetilde{E}_{ijk}\cap R_{hk}$, $\alpha_{ijk,h}':=\widetilde{E}_{ijk}\cap R_{kh}$, where $i,j,k\in\{1,2,3\}$ with $i<j$ and $h\in\{i,j\}$.
By construction, $\alpha_{ij}$ and $\alpha_{ij}'$ are $(-1)$-curves respectively on $\widetilde{E}$ and $\widetilde{E}'$; $\alpha_{ijk}$ and $\alpha_{ijk,h}$ are $(-1)$-curves on $\widetilde{E}_{ijk}$; $\alpha_{ijk}'$ and $\alpha_{ijk,h}'$ are $(-1)$-curves on $\widetilde{E}_{ijk}'$; $\alpha_{ij}$ and $\alpha_{ijk}$ are fibres on $F_{ij}$; $\alpha_{ij}'$ and $\alpha_{ijk}'$ are fibres on $F_{ij}'$; $\alpha_{ijk,h}$ and $\alpha_{ijk,h}'$ are fibres respectively on $R_{hk}$ and $R_{kh}$.
\end{remark}

Let $\mathcal{K}''$ be the strict transform of $\mathcal{K}'$: an 
%general 
element of $\mathcal{K}''$ is linearly equivalent to 
$7H-3\widetilde{E}-3\widetilde{E}'-2\sum_{\substack{i,j,k=1 \\ i<j }}^3(\widetilde{E}_{ijk}+\widetilde{E}_{ijk}')-2\sum_{1\le i < j \le 3} (F_{ij}+F_{ij}')-\sum_{i,j=1}^3R_{ij},$
where, by abuse of notation, $H$ also denotes the pullback $bl''^* H$.

\begin{proposition}\label{prop:tildeKSmoothePa0}
A general element $K''\in\mathcal{K}''$ is a smooth surface with zero arithmetic genus $p_a(K'')=0$.
\end{proposition}
\begin{proof}
The smoothness of $K''$ is shown in \cite[p.620-621]{GH}, since $K''$ is the blow-up of a surface $K\in\mathcal{K}$ with ordinary singularities along its singular curves (see Definition~\ref{def:ordinarysingularities} and Remark~\ref{rem:variabeleTC9}).
We have to compute the arithmetic genus $p_a(K'')= \chi (\mathcal{O}_{K''})-1$.
By Serre Duality, we have that $p_a(K'')=\chi (\mathcal{O}_{K''}(K_{K''}))-1$.
By the adjunction formula, we have the following exact sequence
$$0 \to \mathcal{O}_{Y''}(K_{Y''}) \to \mathcal{O}_{Y''}(K_{Y''}+K'') \to \mathcal{O}_{K''} (K_{K''}) \to 0.$$
Since $Y''$ is a smooth rational threefold, then we have that
%$h^0(Y'', \mathcal{O}_{Y''}(K_{Y''}))=p_g(Y'')=0$. By Serre Duality, we have that
%$h^i(Y'', \mathcal{O}_{Y''}(K_{Y''}))=h^{3-i}(Y'', \mathcal{O}_{Y''})=0$ for $i=1,2,$ 
%and $h^3(Y'', \mathcal{O}_{Y''}(K_{Y''}))=h^{0}(Y'', \mathcal{O}_{Y''})=1.$
%Hence 
%$\chi (\mathcal{O}_{Y''}(K_{Y''})) = -1$ and 
$$p_a(K'')= \chi ( \mathcal{O}_{Y''}(K_{Y''}+K''))- \chi ( \mathcal{O}_{Y''}(K_{Y''}))- 1 = \chi ( \mathcal{O}_{Y''}(K_{Y''}+K'')).$$
Since the canonical divisor of $Y''$ is linearly equivalent to
$$-4H+2\widetilde{E}+2\widetilde{E}'+2\sum_{\substack{i,j,k=1 \\ i<j }}^3(\widetilde{E}_{ijk}+\widetilde{E}_{ijk}')+\sum_{1\le i < j \le 3} (F_{ij}+F_{ij}')+\sum_{i,j=1}^3R_{ij}$$
(see \cite[p.187]{GH}), then we have
$K_{Y''}+K'' \sim 3H - \widetilde{E}-\widetilde{E}'-\sum_{1\le i < j \le 3} (F_{ij}+F_{ij}')$. Let us denote by $f_{ij}$ and $f_{ij}'$ respectively the fibre class of $F_{ij}$ and $F_{ij}'$. Then we have the following two exact sequences
$$0 \to \mathcal{O}_{Y''}(3H-\widetilde{E}-\widetilde{E}') \to \mathcal{O}_{Y''}(3H) \to \mathcal{O}_{\widetilde{E}}\oplus \mathcal{O}_{\widetilde{E}'} \to 0,$$
$$0 \to \mathcal{O}_{Y''}(K_{Y''}+K'') \to \mathcal{O}_{Y''}(3H-\widetilde{E}-\widetilde{E}') \to \oplus_{1\le i<j\le 3}\mathcal{O}_{F_{ij}}(2f_{ij}) \oplus \mathcal{O}_{F_{ij}'}(2f_{ij}')\to 0,$$
and we obtain $\chi (\mathcal{O}_{Y''}(K_{Y''}+K'')) = \binom{3+3}{3}-2-6\cdot 3=0$.
\end{proof}

By Remark~\ref{rem:variabeleTC9} 
%(iii) 
we have that the base locus of $\mathcal{K}''$ is given by the disjoint union of the strict transforms $\widetilde{\gamma}_{i}$, $\widetilde{\gamma}_{i}'$, $\widetilde{\lambda}_{ijk,h}$, $\widetilde{\lambda}_{ijk,h}'$ of the 42 lines defined as above. 

\begin{remark}\label{rem:-1curvesK}
We have 
$\widetilde{\gamma}_{i}^2|_{\widetilde{E}}=\widetilde{\gamma}_{i}^2|_{\widetilde{f}_{i}}=-1$, $\widetilde{\gamma}_{i}'^2|_{\widetilde{E}'}=\widetilde{\gamma}_{i}'^2|_{\widetilde{f}_{i}'}=-1$,
$\widetilde{\lambda}_{ijk,h}^2|_{\widetilde{E}_{ijk}}=\widetilde{\lambda}_{ijk,h}^2|_{\widetilde{f}_{h}}=-1$, $\widetilde{\lambda}_{ijk,h}'^2|_{\widetilde{E}_{ijk}'}=\widetilde{\lambda}_{ijk,h}^2|_{\widetilde{f}_{h}'}=-1$.  
Furthermore, by using similar arguments to the ones in Remark~\ref{rem:-1curvesSigma} we also have that $\widetilde{\gamma}_{i}$, $\widetilde{\gamma}_{i}'$, $\widetilde{\lambda}_{ijk,h}$, $\widetilde{\lambda}_{ijk,h}'$ are $(-1)$-curves on the strict transform $K''$ of a general $K'\in\mathcal{K}'$.
\end{remark}

Finally let us consider the blow-up of $Y''$ along the above $42$ curves, which is the map $bl''' : Y \to Y''$ with exceptional divisors $\Gamma_{i}:=bl'''^{-1}(\widetilde{\gamma}_{i})$, $\Gamma_{i}':=bl'''^{-1}(\widetilde{\gamma}_{i}')$, $\Lambda_{ijk,h}:=bl'''^{-1}(\widetilde{\lambda}_{ijk,h})$, $\Lambda_{ijk,h}':=bl'''^{-1}(\widetilde{\lambda}_{ijk,h}')$.  
We denote by $\mathcal{E}$, $\mathcal{E}'$, $\mathcal{E}_{ijk}$, $\mathcal{E}_{ijk}'$, respectively, the strict transform of $\widetilde{E}$, $\widetilde{E}'$, $\widetilde{E}_{ijk}$, $\widetilde{E}_{ijk}'$; by $\mathcal{F}_{ij}$ the strict transform of $F_{ij}$; by $\mathcal{R}_{ik}$ the strict transform of $R_{ik}$; by $\mathcal{H}$ the pullback of $H$, for $i,j,k\in\{1,2,3\}$ with $i< j$ and $h\in\{i,j\}$. 

\begin{remark}\label{rem:Gamma^3p9}
We have that
$$\Gamma_{i}=\mathbb{P}(\mathcal{N}_{\widetilde{\gamma}_{i}|Y''})\cong \mathbb{P}(\mathcal{O}_{\widetilde{\gamma}_{i}}(\widetilde{E})\oplus \mathcal{O}_{\widetilde{\gamma}_{i}}(\widetilde{f}_{i})) \cong \mathbb{P}(\mathcal{O}_{\mathbb{P}^1}(-1)\oplus \mathcal{O}_{\mathbb{P}^1}(-1))\cong \mathbb{F}_0,$$
$$\Gamma_{i}'=\mathbb{P}(\mathcal{N}_{\widetilde{\gamma}_{i}'|Y''})\cong \mathbb{P}(\mathcal{O}_{\widetilde{\gamma}_{i}'}(\widetilde{E}')\oplus \mathcal{O}_{\widetilde{\gamma}_{i}'}(\widetilde{f}_{i}')) \cong \mathbb{P}(\mathcal{O}_{\mathbb{P}^1}(-1)\oplus \mathcal{O}_{\mathbb{P}^1}(-1))\cong \mathbb{F}_0,$$
\begin{small}
$$\Lambda_{ijk,h}=\mathbb{P}(\mathcal{N}_{\widetilde{\lambda}_{ijk,h}|Y''})\cong \mathbb{P}(\mathcal{O}_{\widetilde{\lambda}_{ijk,h}}(\widetilde{E}_{ijk})\oplus \mathcal{O}_{\widetilde{\lambda}_{ijk,h}}(\widetilde{f}_{h})) \cong \mathbb{P}(\mathcal{O}_{\mathbb{P}^1}(-1)\oplus \mathcal{O}_{\mathbb{P}^1}(-1))\cong \mathbb{F}_0,$$
\end{small}
\begin{small}
$$\Lambda_{ijk,h}'=\mathbb{P}(\mathcal{N}_{\widetilde{\lambda}_{ijk,h}'|Y''})\cong \mathbb{P}(\mathcal{O}_{\widetilde{\lambda}_{ijk,h}'}(\widetilde{E}_{ijk}')\oplus \mathcal{O}_{\widetilde{\lambda}_{ijk,h}'}(\widetilde{f}_{h}')) \cong \mathbb{P}(\mathcal{O}_{\mathbb{P}^1}(-1)\oplus \mathcal{O}_{\mathbb{P}^1}(-1))\cong \mathbb{F}_0.$$
\end{small}
Furthermore, we have $\Gamma_{i}^{3}=-\deg (\mathcal{N}_{\widetilde{\gamma}_{i}|Y''})=2$, $ \Gamma_{i}'^{3}=-\deg (\mathcal{N}_{\widetilde{\gamma}_{i}'|Y''})=2$,
$\Lambda_{ijk,h}^{3}=-\deg (\mathcal{N}_{\widetilde{\lambda}_{ijk,h}|Y''})=2$, $\Lambda_{ijk,h}'^{3}=-\deg (\mathcal{N}_{\widetilde{\lambda}_{ijk,h}'|Y''})=2$
(see \cite[Chap 4, \S 6]{GH} and \cite[Lemma 2.2.14]{IsPro99}).
\end{remark}

\begin{remark}\label{rem:intersectionYp9}
Let us take $i,j,k\in\{1,2,3\}$ with $i< j$ and $h\in\{i,j\}$. The divisor $\mathcal{F}_{ij}$ intersects $\Gamma_i$, $\Gamma_j$, $\Lambda_{ijk,h}$ each along a $\mathbb{P}^1$, which is a $(-1)$-curve on $\mathcal{F}_{ij}$ and a fibre on $\Gamma_i$, $\Gamma_j$, $\Lambda_{ijk,h}$. The same happens with $\mathcal{F}_{ij}'$ and $\Gamma_i'$, $\Gamma_j'$, $\Lambda_{ijk,h}'$. Similarly we have
$\Lambda_{ijk,h}^2\cdot\mathcal{R}_{hk}=\Lambda_{ijk,h}'^2\cdot\mathcal{R}_{kh}=-1$ and 
$\Lambda_{ijk,h}\cdot\mathcal{R}_{hk}^2=\Lambda_{ijk,h}'\cdot\mathcal{R}_{kh}^2=0.$
Let us consider the strict transforms $\widetilde{\alpha}_{ij}$, $\widetilde{\alpha}_{ij}'$, $\widetilde{\alpha}_{ijk}$, $\widetilde{\alpha}_{ijk}'$, $\widetilde{\alpha}_{ijk,h}$, $\widetilde{\alpha}_{ijk,h}'$ of the curves defined in Remark~\ref{rem:selfalphap9}. Then we have 
$$\widetilde{\alpha}_{ij}^2|_{\mathcal{E}}=\mathcal{F}_{ij}^2\cdot \mathcal{E}=-1, \quad \widetilde{\alpha}_{ij}^2|_{\mathcal{F}_{ij}}=\mathcal{E}^2\cdot \mathcal{F}_{ij}=-2,$$
$$\widetilde{\alpha}_{ij}'^2|_{\mathcal{E}'}={\mathcal{F}_{ij}'}^2\cdot \mathcal{E}'=-1,\quad \widetilde{\alpha}_{ij}'^2|_{\mathcal{F}_{ij}}=\mathcal{E}'^2\cdot \mathcal{F}_{ij}'=-2,$$
$$\widetilde{\alpha}_{ijk}^2|_{\mathcal{E}_{ijk}}=\mathcal{F}_{ij}^2\cdot\mathcal{E}_{ijk}=-1, \quad \widetilde{\alpha}_{ijk}^2|_{\mathcal{F}_{ij}^2}=\mathcal{E}_{ijk}^2\cdot \mathcal{F}_{ij}=-2,$$
$$\widetilde{\alpha}_{ijk}'^2|_{\mathcal{E}_{ijk}'}=\mathcal{F}_{ij}^2\cdot\mathcal{E}_{ijk}'=-1,\quad \widetilde{\alpha}_{ijk}'^2|_{\mathcal{F}_{ij}}=\mathcal{E}_{ijk}'^2\cdot \mathcal{F}_{ij}=-2,$$
$$\widetilde{\alpha}_{ijk,h}^2|_{\mathcal{E}_{ijk}}=\mathcal{R}_{hk}^2\cdot\mathcal{E}_{ijk}=-1,\quad \widetilde{\alpha}_{ijk,h}^2|_{\mathcal{R}_{hk}}=\mathcal{E}_{ijk}^2\cdot \mathcal{R}_{hk}=-1,$$
$$\widetilde{\alpha}_{ijk,h}'^2|_{\mathcal{E}_{ijk}'}=\mathcal{R}_{kh}^2\cdot\mathcal{E}_{ijk}'=-1,\quad \widetilde{\alpha}_{ijk,h}'^2|_{\mathcal{R}_{kh}}=\mathcal{E}_{ijk}'^2\cdot \mathcal{R}_{kh}=-1.$$
Finally we recall that a general line of $\mathbb{P}^3$ does not intersect the edges of the trihedra and the nine lines $r_{ij}$, while a general plane of $\mathbb{P}^3$ intersects each of these lines at one point. Hence we have that $\mathcal{H}^2\cdot \mathcal{F}_{ij}=\mathcal{H}^2\cdot \mathcal{F}_{ij}'= \mathcal{H}^2\cdot \mathcal{R}_{ik}=0$ and $\mathcal{F}_{ij}^2\cdot\mathcal{H}=\mathcal{F}_{ij}'^2\cdot\mathcal{H}=\mathcal{R}_{ik}^2\cdot\mathcal{H}=-1$.
\end{remark}

\begin{remark}\label{rem:Ek^3p9}
By construction we have that
$${bl'''}^*(\widetilde{E}) = \mathcal{E}+\sum_{1\le x<y\le 3}\Gamma_{xy},\quad
{bl'''}^*(E_{ijk}) =\mathcal{E}_{ijk}+\Lambda_{ijk,i}+\Lambda_{ijk,j},$$
$${bl'''}^*(\widetilde{E}') = \mathcal{E}'+\sum_{1\le x<y\le 3}\Gamma_{xy}',\quad
{bl'''}^*(E_{ijk}')= \mathcal{E}_{ijk}'+\Lambda_{ijk,i}'+\Lambda_{ijk,j}',$$
where $i,j,k\in\{1,2,3\}$ and $i<j$. By abuse of notation, we denote $\mathcal{E}\cap \Gamma_{ij}$, $\mathcal{E}'\cap \Gamma_{ij}'$, $\mathcal{E}_{ijk}\cap \Lambda_{ijk,h}$, $\mathcal{E}_{ijk}'\cap \Lambda_{ijk,h}'$, respectively, by $\widetilde{\gamma}_{ij}$, $\widetilde{\gamma}_{ij}'$, $\widetilde{\lambda}_{ijk,h}$, $\widetilde{\lambda}_{ijk,h}'$, where $h\in\{i,j\}$. 
Let $\mathcal{L}$, $\mathcal{L}'$, $\mathcal{L}_{ijk}$, $\mathcal{L}_{ijk}'$ be the strict transforms on $Y$ of a general line respectively of $E$, $E'$, $E_{ijk}$, $E_{ijk}'$.   
By using similar arguments to the ones in Remark~\ref{rem:Ek^3p13} we obtain
$$\mathcal{E}|_{\mathcal{E}} \sim - ( \mathcal{L} + \sum_{t=1}^3\widetilde{\gamma}_{i} ) \sim -(4\mathcal{L}-2\sum_{0\le x<y\le 3}\widetilde{\alpha}_{xy}),$$
$$\mathcal{E}'|_{\mathcal{E}'} \sim - ( \mathcal{L}' + \sum_{t=1}^3\widetilde{\gamma}_{i}' ) \sim -(4\mathcal{L}'-2\sum_{0\le x<y\le 3}\widetilde{\alpha}_{xy}'),$$
$$\mathcal{E}_{ijk}|_{\mathcal{E}_{ijk}} \sim - ( \mathcal{L}_{ijk} + \widetilde{\lambda}_{ijk,i}+\widetilde{\lambda}_{ijk,j} ) \sim -(3\mathcal{L}_{ijk}-2\widetilde{\alpha}_{ijk}-\widetilde{\alpha}_{ijk,i}-\widetilde{\alpha}_{ijk,j}),$$
$$\mathcal{E}_{ijk}'|_{\mathcal{E}_{ijk}'} \sim - ( \mathcal{L}_{ijk}' + \widetilde{\lambda}_{ijk,i}'+\widetilde{\lambda}_{ijk,j}' ) \sim -(3\mathcal{L}_{ijk}'-2\widetilde{\alpha}_{ijk}'-\widetilde{\alpha}_{ijk,i}'-\widetilde{\alpha}_{ijk,j}'),$$
so we have $\mathcal{E}^3=4$, $\mathcal{E}'^3=4$, $\mathcal{E}_{ijk}^3=3$ and $\mathcal{E}_{ijk}'^3=3$.
\end{remark}

\begin{remark}\label{rem:Fij^3p9}
With similar arguments to the ones in Remark~\ref{rem:Fij^3p13}, we have 
$\mathcal{F}_{ij}^3=-\deg (\mathcal{N}_{\widetilde{l}_{ij}|Y'})=6$, $\mathcal{F}_{ij}'^3=-\deg (\mathcal{N}_{\widetilde{l}_{ij}'|Y'}) = 6$, $\mathcal{R}_{ki}^3=-\deg (\mathcal{N}_{\widetilde{r}_{ki}|Y'})=6$, for $i,j,k\in\{1,2,3\}$ with $i<j$.
\end{remark}

Let $\widetilde{K}$ be the strict transform on $Y$ of an 
%general 
element of $\mathcal{K}''$: then
$$\widetilde{K} \sim 7\mathcal{H}-3\mathcal{E}-3\mathcal{E}'-2\sum_{\substack{i,j,k=1 \\ i<j }}^3(\mathcal{E}_{ijk}+\mathcal{E}_{ijk}')-2\sum_{1\le i < j \le 3}(\mathcal{F}_{ij}+\mathcal{F}_{ij}')-\sum_{i,j=1}^{3}\mathcal{R}_{ij}+$$
$$-4\sum_{i=1}^{3}(\Gamma_i+\Gamma_i')-3\sum_{ \substack{i,j,k=1 \\ i<j,\, h=i,j }}^3 (\Lambda_{ijk,h}+\Lambda_{ijk,h}').$$

Let us take the linear system $\widetilde{\mathcal{K}}:=|\mathcal{O}_{Y}(\widetilde{K})|$ on $Y$. It is base point free and it defines a morphism $\nu_{\widetilde{\mathcal{K}}}: Y \to \mathbb{P}^{9}$ birational onto the image $W_F^{9}:= \nu_{\widetilde{\mathcal{K}}} (Y)$, which is a threefold of degree $\deg W_F^{9} = 16$. This follows by Lemma~\ref{lem:birationalitanu9} and by the fact that $\widetilde{K}^3=16$ (use Remarks~\ref{rem:Gamma^3p9},~\ref{rem:intersectionYp9},~\ref{rem:Ek^3p9},~\ref{rem:Fij^3p9}).
Then we have the following diagram:

$$\begin{tikzcd}
Y \arrow[d, "bl'''"] \arrow[drrr, "\nu_{\widetilde{\mathcal{K}}}"] & & & \\
Y''  \arrow{r}{bl''} & Y' \arrow{r}{bl'} & \mathbb{P}^3 \arrow[dashrightarrow]{r}{\nu_{\mathcal{K}}} & W_F^{9} \subset \mathbb{P}^{9}.
\end{tikzcd}$$

It remains to show that the general hyperplane section of the threefold $W_F^9$ is an Enriques surface. 

\begin{remark}\label{rem:K|E=0}
By construction we have
$\widetilde{K}\cdot \mathcal{E} = \widetilde{K}\cdot \mathcal{E} = \widetilde{K}\cdot \mathcal{E}_{ijk} = \widetilde{K}\cdot \mathcal{E}_{ijk}' = 0$, for all $i,j,k\in\{1,2,3\}$ with $i<j$. 
\end{remark}

\begin{remark}\label{rem:contractionInpoints9}
Let $1\le i \le 3$.
Since $bl''' : Y \to Y''$ has no effect on the divisors $\widetilde{f}_i$ and $\widetilde{f}_i'$, we will continue to use the same notations to denote their strict transforms. By construction we have $\widetilde{K}\cdot \widetilde{f}_i=\widetilde{K}\cdot \widetilde{f}_i'=0$ for a general $\widetilde{K}\in\widetilde{\mathcal{K}}$.
\end{remark}

\begin{remark}\label{rem:nu9blowdown}
The morphism $\nu_{\widetilde{\mathcal{K}}} : Y \to W_{F}^{9} \subset \mathbb{P}^{9}$ blows-down the 42 exceptional divisors of $bl''':Y \to Y''$ and the nine divisors $\mathcal{R}_{ik}$ to curves of $W_{F}^{9}$.
This follows by the fact that $\widetilde{K}\cdot \Gamma_{i}$, $\widetilde{K}\cdot \Gamma_{i}'$, $\widetilde{K}\cdot \Lambda_{ijk,h}$, $\widetilde{K}\cdot \Lambda_{ijk,h}'$, $\widetilde{K}\cdot \mathcal{R}_{ik} \ne 0$ and 
$\widetilde{K}^2\cdot \Gamma_{i} = \widetilde{K}^2\cdot \Gamma_{i}' = \widetilde{K}^2\cdot \Lambda_{ijk,h}= \widetilde{K}^2\cdot \Lambda_{ijk,h}'=\widetilde{K}^2\cdot \mathcal{R}_{ik} = 0$, for all $i,j,k\in\{1,2,3\}$ with $i< j$ and $h\in\{i,j\}$ (use Remarks~\ref{rem:Gamma^3p9},~\ref{rem:intersectionYp9}).
\end{remark}

\begin{remark}\label{rem:nocontractionSpigoli9}
Let $\widetilde{K}$ be a general element of $\widetilde{\mathcal{K}}$. 
By Remarks~\ref{rem:intersectionYp9},~\ref{rem:Fij^3p9} we have
that $\widetilde{K}^2\cdot \mathcal{F}_{ij}=\widetilde{K}^2\cdot \mathcal{F}_{ij}'=8>0$ for $0\le i<j\le 3$. Thus,
the curves $\widetilde{K}\cap \mathcal{F}_{ij}$ and $\widetilde{K}\cap \mathcal{F}_{ij}'$ are not contracted by the rational map defined by $\widetilde{\mathcal{K}}|_{\widetilde{K}}$. 
\end{remark}

\begin{remark}\label{rem:-1curvenuove9}
Let us fix a general element $\widetilde{K}\in \widetilde{\mathcal{K}}$ and let us take $S:=\nu_{\widetilde{\mathcal{K}}}(\widetilde{K})$ and $K'':=bl'''(\widetilde{K})\in\mathcal{K}''$. Since $bl''' : Y \to Y''$ has no effect on $K''$, then $\widetilde{K}\cap \Gamma_{i}$, $\widetilde{K}\cap \Gamma_{i}'$, $\widetilde{K}\cap \Lambda_{ijk,h}$, $\widetilde{K}\cap \Lambda_{ijk,h}'$ are still $(-1)$-curves on $\widetilde{K}$, for all $i,j,k\in\{1,2,3\}$ with $i< j$ and $h\in\{i,j\}$ (see Remark~\ref{rem:-1curvesK}). 
By Remarks~\ref{rem:intersectionYp9},~\ref{rem:Fij^3p9}, we also have that 
$(\widetilde{K}\cap \mathcal{R}_{ik})|_{\widetilde{K}}^2 = \mathcal{R}_{ik}^2\cdot \widetilde{K}=\mathcal{R}_{ik}^2 \cdot (7\mathcal{H}-2\sum_{\substack{ 1\le a<b\le 3, \, i\in\{a,b\} \\ 1\le x<y\le 3,\, k\in\{x,y\}}}(\mathcal{E}_{abk}+\mathcal{E}_{xyi}')-\mathcal{R}_{ik}) = -5.$
Furthermore,
we have $\mathcal{R}_{ik}\cdot_{\widetilde{K}} \Lambda_{abk,i} = 1$ and $ \mathcal{R}_{ik}\cdot_{\widetilde{K}} \Lambda_{xyi,k}' = 1$ for $1\le a<b\le 3$ and $1\le x<y\le 3$ with $i\in\{a,b\}$ and $k\in\{x,y\}$ (use Remark~\ref{rem:intersectionYp9}).
Thus, we can see the map $\nu_{\widetilde{\mathcal{K}}} |_{\widetilde{K}} : \widetilde{K} \to S$ as the blow-up of $S$ at the six points $\nu_{\widetilde{\mathcal{K}}}(\widetilde{K}\cap \Gamma_i)$ and $\nu_{\widetilde{\mathcal{K}}}(\widetilde{K}\cap \Gamma_i')$, at the nine points $\nu_{\widetilde{\mathcal{K}}}(\widetilde{K}\cap \mathcal{R}_{ik})$ and at the four points $\nu_{\widetilde{\mathcal{K}}}(\widetilde{K}\cap \Lambda_{abk,i})$, $\nu_{\widetilde{\mathcal{K}}}(\widetilde{K}\cap \Lambda_{xyi,k}')$ which are infinitely near to each $\nu_{\widetilde{\mathcal{K}}}(\widetilde{K}\cap \mathcal{R}_{ik})$ (see Remarks~\ref{rem:veryample9},~\ref{rem:K|E=0},~\ref{rem:contractionInpoints9},~\ref{rem:nu9blowdown},~\ref{rem:nocontractionSpigoli9}). Then $S$ is a smooth surface.
\end{remark}

\begin{remark}\label{rem:uniqueTT'}
The surface $T\cup T'$ is the only sextic surface of $\mathbb{P}^3$ which is singular along the edges of the two trihedra. Let us consider the strict transforms $\widetilde{T}$ and $\widetilde{T}'$ on $Y$ of the trihedra: 
$$\widetilde{T}\sim 3\mathcal{H}-3\mathcal{E}-\sum_{\substack{i,j,k=1 \\ i<j}}^{3}(2\mathcal{E}_{ijk}+\mathcal{E}_{ijk}')-\sum_{1\le i<j\le 3}2\mathcal{F}_{ij}-\sum_{i,j=1}^{3}\mathcal{R}_{ij}+$$
$$-\sum_{i=1}^{3}4\Gamma_i-\sum_{ \substack{ i,j,k\in \{1,2,3\} \\ i< j,\,\, h=i,j } } (3\Lambda_{ijk,h}+\Lambda_{ijk,h}'),$$
$$\widetilde{T}'\sim 3\mathcal{H}-3\mathcal{E}'-\sum_{\substack{i,j,k=1 \\ i<j}}^{3}(\mathcal{E}_{ijk}+2\mathcal{E}_{ijk}')-\sum_{i=1}^{3}2\mathcal{F}_{ij}'-\sum_{i,j=1}^{3}\mathcal{R}_{ij}+$$
$$-\sum_{i=1}^{3}4\Gamma_i'-\sum_{ \substack{ i,j,k\in \{1,2,3\} \\ i< j,\,\, h=i,j } }(\Lambda_{ijk,h}+3\Lambda_{ijk,h}').$$
Let $\widetilde{K}$ be a general element of $\widetilde{\mathcal{K}}$. Then we have that 
$$0\sim (\widetilde{T}+\widetilde{T}')|_{\widetilde{K}} \sim
\Big(6\mathcal{H}-\sum_{1\le i<j\le 3}(2\mathcal{F}_{ij}+2\mathcal{F}_{ij}')-\sum_{i,j=1}^{3}2\mathcal{R}_{ij}+$$
$$-\sum_{i=1}^{3}(4\Gamma_i+4\Gamma_i')-\sum_{ \substack{ i,j,k\in \{1,2,3\} \\ i< j,\,\, h=i,j } }( 4\Lambda_{ijk,h}+4 \Lambda_{ijk,h}')\Big)|_{\widetilde{K}}.$$
\end{remark}

\begin{theorem}\label{thm:S9isEnriques}
Let $S$ be a general hyperplane section of the threefold $W_F^{9}\subset \mathbb{P}^{9}$. Then $S$ is an Enriques surface.
\end{theorem} 
\begin{proof}
We recall that $S$ is the image of a general element $\widetilde{K}\in \widetilde{\mathcal{K}}$ via the birational morphism $\nu_{\widetilde{\mathcal{K}}} : Y \to W_{F}^{9} \subset \mathbb{P}^{9}$. Furthermore, $S$ is smooth (see Remark~\ref{rem:-1curvenuove9}).
By Proposition~\ref{prop:tildeKSmoothePa0} we have that $p_g(\widetilde{K})-q(\widetilde{K}) = p_a(\widetilde{K}) = 0.$ Let us consider the following exact sequence
$$0 \to \mathcal{O}_{Y}(-\widetilde{K}) \to \mathcal{O}_{Y}\to \mathcal{O}_{\widetilde{K}} \to 0.$$
Since $Y$ is a smooth rational threefold and $\widetilde{K}$ is a big and nef divisor on $Y$, by Serre Duality and by Kawamata-Viehweg vanishing theorem we have
$h^i(Y, \mathcal{O}_{Y}(-\widetilde{K}))=0$ for  $i=1,2,$ 
and so $q(\widetilde{K})=h^1(\widetilde{K},\mathcal{O}_{\widetilde{K}}) = h^1(Y,\mathcal{O}_{Y})=0.$
Thus, we also obtain $p_g(\widetilde{K})=0.$
It remains to prove that $2K_{S} \sim 0$.
Since
$$K_{Y} =  {bl'''}^*(K_{Y''})+\sum_{i=1}^{3}(\Gamma_i+\Gamma_i')+\sum_{ \substack{ i,j,k =1 \\ i<j,\,\, h=i,j}}^3(\Lambda_{ijk,h}+\Lambda_{ijk,h}') \sim$$
$$\sim -4\mathcal{H}+2\mathcal{E}+2\mathcal{E}'+2\sum_{\substack{i,j,k=1 \\ i<j }}^3(\mathcal{E}_{ijk}+\mathcal{E}_{ijk}')+$$
$$+\sum_{\\le i<j\le 3}(\mathcal{F}_{ij}+\mathcal{F}_{ij}')+\sum_{i,j=1}^{3}\mathcal{R}_{ij}+\sum_{i=1}^{3}3(\Gamma_i+\Gamma_i')+\sum_{ \substack{ i,j,k =1 \\ i<j,\,\, h=i,j}}^33(\Lambda_{ijk,h}+\Lambda_{ijk,h}')$$
(see \cite[p.187]{GH}), then, by the adjunction formula, we obtain that
$$2K_{\widetilde{K}} = 2(K_{Y}+\widetilde{K})|_{\widetilde{K}}\sim \Big(6\mathcal{H}-\sum_{1\le j\le 3}2(\mathcal{F}_{ij}+\mathcal{F}_{ij}')- \sum_{i=1}^{3}2(\Gamma_i+\Gamma_i') \Big)|_{\widetilde{K}}.$$
Furthermore, by Remark~\ref{rem:uniqueTT'}, we have
$$2K_{\widetilde{K}} \sim \Big(\widetilde{T}+\widetilde{T}'+\sum_{i,j=1}^{3}2\mathcal{R}_{ij}+\sum_{i=1}^{3}2(\Gamma_i+\Gamma_i')+\sum_{ \substack{ i,j,k =1 \\ i<j,\,\, h=i,j}}^34(\Lambda_{ijk,h}+\Lambda_{ijk,h}')\Big)|_{\widetilde{K}}\sim$$
$$\sim \Big(\sum_{i,j=1}^{3}2\mathcal{R}_{ij}+\sum_{i=1}^{3}2(\Gamma_i+\Gamma_i')+\sum_{ \substack{ i,j,k =1 \\ i<j,\,\, h=i,j}}^34(\Lambda_{ijk,h}+\Lambda_{ijk,h}')\Big)|_{\widetilde{K}}=$$
$$=\sum_{i,k=1}^{3}2\Big(\mathcal{R}_{ik}+\sum_{\substack{a,b,x,y\in\{1,2,3\} \\ a<b,\, x<y \\ i\in\{a,b\},\,\, k\in\{x,y\}}}(\Lambda_{abk,i}+\Lambda_{xyi,k}')\Big)|_{\widetilde{K}}+$$
$$+\sum_{i=1}^{3}2(\Gamma_i+\Gamma_i')|_{\widetilde{K}}+\sum_{ \substack{ i,j,k =1 \\ i<j,\,\, h=i,j}}^32(\Lambda_{ijk,h}+\Lambda_{ijk,h}')|_{\widetilde{K}}.$$
Finally, by Remark~\ref{rem:-1curvenuove9}, we obtain
$2K_{S}\sim (\nu_{\widetilde{\mathcal{S}}})_{*}(2K_{\widetilde{K}})\sim 0$.
\end{proof}

One can prove that $W_F^{9}\subset \mathbb{P}^{9}$ is not a cone over a general hyperplane section, as in the proof of Theorem~\ref{thm:WF13moderno}. So $W_F^9\subset \mathbb{P}^9$ is an Enriques-Fano threefold of genus $p=\frac{S^3}{2}+1=\frac{\widetilde{K}^3}{2}+1=9$. 
\end{proof}

%\subsection{Singularities}

We recall that the eight divisors $\mathcal{E}$, $\mathcal{E}'$, $\widetilde{f}_1$, $\widetilde{f}_2$, $\widetilde{f}_3$, $\widetilde{f}_1'$, $\widetilde{f}_2'$, $\widetilde{f}_3'$ are contracted by $\nu_{\widetilde{\mathcal{K}}} : Y \to W_F^{9} \subset \mathbb{P}^{9}$ to points of $W_F^{9}$ (see Remarks~\ref{rem:K|E=0},~\ref{rem:contractionInpoints9}).
Let us define
$$P_1 : = \nu_{\widetilde{\mathcal{K}}}(\mathcal{E}'), \, P_2 : = \nu_{\widetilde{\mathcal{K}}}(\widetilde{f}_1),\, P_3 : = \nu_{\widetilde{\mathcal{K}}}(\widetilde{f}_2), \,  P_4 : = \nu_{\widetilde{\mathcal{K}}}(\widetilde{f}_3),$$
$$P_1' : = \nu_{\widetilde{\mathcal{K}}}(\mathcal{E}), \, P_2' : = \nu_{\widetilde{\mathcal{K}}}(\widetilde{f}_1'), \, P_3' : = \nu_{\widetilde{\mathcal{K}}}(\widetilde{f}_2'), \, P_4' : = \nu_{\widetilde{\mathcal{K}}}(\widetilde{f}_3').$$

\begin{lemma}\label{lem:36to6p9}
The $18$ divisors $\mathcal{E}_{ijk}$ and $\mathcal{E}_{ijk}'$ are mapped by $\nu_{\widetilde{\mathcal{K}}} : Y \to W_F^{9} \subset \mathbb{P}^{9}$ to the six points $P_2$, $P_3$, $P_4$, $P_2'$, $P_3'$ and $P_4'$ of $W_F^{9}$ in the following way: 
$$P_{i+1} = \nu_{\widetilde{\mathcal{K}}}(\widetilde{f}_i)=\nu_{\widetilde{\mathcal{K}}}(\mathcal{E}_{rsi}'), \quad 
P_{i+1}' = \nu_{\widetilde{\mathcal{K}}}(\widetilde{f}_i')=\nu_{\widetilde{\mathcal{K}}}(\mathcal{E}_{rsi}),$$
for all $i,r,s\in\{1,2,3\}$ and $r<s$. 
\end{lemma}
\begin{proof}
By Remark~\ref{rem:K|E=0} we have that $\nu_{\widetilde{\mathcal{K}}}(\mathcal{E}_{ijk})$ and $\nu_{\widetilde{\mathcal{K}}}(\mathcal{E}_{ijk}')$ are points of $W_F^{9}$ for all $i,j,k\in\{1,2,3\}$ and $i<j$. Since $\widetilde{f}_i\cap \mathcal{E}_{rsi}'\ne \emptyset$ for all $i,r,s\in\{1,2,3\}$ and $r<s$, then the three divisors $\mathcal{E}_{rsi}'$ are mapped to the same point $\nu_{\widetilde{\mathcal{K}}}(\widetilde{f}_i)=P_{i+1}$. Similarly the three divisors $\mathcal{E}_{rsi}$ are mapped to the same point $\nu_{\widetilde{\mathcal{K}}}(\widetilde{f}_i')=P_{i+1}'$.
\end{proof}

\begin{proposition}\label{prop:quadruplePointsp9}
The points $P_1, \dots , P_4, P_1', \dots , P_4'$ are eight quadruple points of $W_F^{9}$ whose tangent cone is a cone over a Veronese surface.
\end{proposition}
\begin{proof}
The analysis of the points $P_1'$ and $P_1$ follows by Remark~\ref{rem:Ek^3p9} as in the proof of Proposition~\ref{prop:quadruplePointsp13}. Let us fix now $1\le i \le 3$. 
Let us find the tangent cone to $W_F^9$ at $P_{i+1}$. Similarly one can study the tangent cone to $W_F^{9}$ at $P_{i+1}'$.
The hyperplane sections of $W_F^{9}\subset \mathbb{P}^{9}$ passing through $P_{i+1}$ correspond to the elements of $\widetilde{\mathcal{K}}$ containing $\widetilde{f}_i\cup \mathcal{E}_{12i}'\cup \mathcal{E}_{13i}'\cup \mathcal{E}_{23i}'$ (see Lemma~\ref{lem:36to6p9}). Let $\widetilde{\mathcal{K}}_i:=\widetilde{\mathcal{K}}-\widetilde{f}_i-\mathcal{E}_{12i}'-\mathcal{E}_{13i}'-\mathcal{E}_{23i}'$ be the sublinear system of $\widetilde{\mathcal{K}}$ defined by these elements.
Let us study 
$$\widetilde{\mathcal{K}}_i|_{\widetilde{f}_{i}} = |\mathcal{O}_{\widetilde{f}_{i}}(-\widetilde{f}_i-\mathcal{E}_{12i}'-\mathcal{E}_{13i}'-\mathcal{E}_{23i}')|.$$
Let us consider the case $i=1$. Since 
$$\widetilde{f}_1 \sim_{Y} \mathcal{H}-\mathcal{E}_{v}-\sum_{j=1}^3\mathcal{E}_{13j}-\sum_{j=1}^3\mathcal{E}_{12j}-\sum_{1\le r<s\le 3}\mathcal{E}_{rs1}'-\mathcal{F}_{13}-\mathcal{F}_{12}-\sum_{j=1}^{3}\mathcal{R}_{1j}-2\Gamma_1-\Gamma_2-\Gamma_3+$$
$$-\sum_{j=1}^3(2\Lambda_{13j,1}+\Lambda_{13j,3})-\sum_{j=1}^3(2\Lambda_{12j,1}+\Lambda_{12j,2})-\sum_{1\le r<s \le 3 }(\Lambda_{rs1,r}'+\Lambda_{rs1,s}'),$$ 
we have that
$$\widetilde{f}_1|_{\widetilde{f}_1} \sim_{\widetilde{f}_1} ( \mathcal{H}-\sum_{1\le r<s\le 3}\mathcal{E}_{rs1}'-\mathcal{F}_{13}-\mathcal{F}_{12}-\sum_{j=1}^{3}\mathcal{R}_{1j}-2\Gamma_1-\sum_{j=1}^3 2\Lambda_{13j,1}-\sum_{j=1}^32\Lambda_{12j,1} )|_{\widetilde{f}_1}.$$
Let $\mathcal{L}_1$ be the pullback on $\widetilde{f}_1$ of the linear equivalence class of the lines of the face $f_1\cong \mathbb{P}^2$. By abuse of notation, let us denote by $\widetilde{\gamma}_1$, $\widetilde{\lambda}_{131,1}$, $\widetilde{\lambda}_{132,1}$, $\widetilde{\lambda}_{133,1}$, $\widetilde{\lambda}_{121,1}$, $\widetilde{\lambda}_{122,1}$, $\widetilde{\lambda}_{123,1}$ the $(-1)$-curves on $\widetilde{f}_1$ given by $\Gamma_1|_{\widetilde{f}_1}$, $\Lambda_{131,1}|_{\widetilde{f}_1}$, $\Lambda_{132,1}|_{\widetilde{f}_1}$, $\Lambda_{133,1}|_{\widetilde{f}_1}$, $\Lambda_{121,1}|_{\widetilde{f}_1}$, $\Lambda_{122,1}|_{\widetilde{f}_1}$, $\Lambda_{123,1}|_{\widetilde{f}_1}$.
Let us also consider the $(-1)$-curves on $\widetilde{f}_1$ defined by $\epsilon_{rs1}:=\mathcal{E}_{rs1}'|_{\widetilde{f}_1}$ for $1\le r<s \le 3$. Then we have
\begin{center}
$\widetilde{f}_1|_{\widetilde{f}_1} \sim \mathcal{L}_1-\sum_{1\le r<s\le 3}\epsilon_{rs1}'-(\mathcal{L}_1-\widetilde{\gamma}_1-\sum_{j=1}^3 \widetilde{\lambda}_{13j,1})-(\mathcal{L}_{1}-\widetilde{\gamma}_1-\sum_{j=1}^3 \widetilde{\lambda}_{12j,1})+$

$-(3\mathcal{L}_1-\sum_{1\le r<s 3}2\epsilon_{rs1}'-\sum_{j=1}^3 \widetilde{\lambda}_{13j,1}-\sum_{j=1}^3\widetilde{\lambda}_{12j,1})-2\widetilde{\gamma}_1-\sum_{j=1}^3 2\widetilde{\lambda}_{13j,1}+$

$-\sum_{j=1}^32\widetilde{\lambda}_{12j,1}=-4\mathcal{L}_1+\sum_{1\le r<s\le 3}\epsilon_{rs1}'$.
\end{center}
Similarly $\widetilde{f}_i|_{\widetilde{f}_i} \sim-4\mathcal{L}_i+\sum_{1\le r<s\le 3}\epsilon_{rsi}'$, for $i=2,3$.
Thus, we obtain $\widetilde{\mathcal{K}}_i|_{\widetilde{f}_{i}} = |\mathcal{O}_{\widetilde{f}_{i}}( 4\mathcal{L}_i-\sum_{1\le r<s\le 3}2\epsilon_{rsi}')|$, which is isomorphic to the linear system of the quartic plane curves on $f_i$ with double points at the three points $q_{rsi}'=l_{rs}'\cap f_i$ for $1\le r<s \le 3$. 
By applying a quadratic transformation, we obtain that $\widetilde{\mathcal{K}}_i|_{\widetilde{f}_i} \cong |\mathcal{O}_{\mathbb{P}^2}(2)|$, 
whose image is a Veronese surface $V_i$. Furthermore, we have that $\widetilde{\mathcal{K}}_i|_{\mathcal{E}_{rsi}'} = |\mathcal{O}_{\mathcal{E}_{rsi}'}(-\widetilde{f}_i-\mathcal{E}_{rsi}')|=|\mathcal{O}_{\mathcal{E}_{rsi}'}(2\mathcal{L}_{rsi}'-2\widetilde{\alpha}_{rsi}')|\cong \mathbb{P}^2$ for $1\le r<s \le 3$ (see Remark~\ref{rem:Ek^3p9}). Since $\widetilde{\mathcal{K}}_i|_{\mathcal{E}_{rsi}'}$ is isomorphic to the linear system of the conics of $E_{rsi}'$ with node at the point $E_{rsi}'\cap \widetilde{l}_{ij}'$, then its image is a conic $C_{rsi}'$. 
Since $V_i\cup C_{12i}'\cup C_{13i}'\cup C_{23i}' = \mathbb{P}(TC_{P_{i+1}}W_F^9)$, then it must be $C_{12i}',C_{13i}',C_{23i}' \subset V_i = \mathbb{P}(TC_{P_{i+1}}W_F^9)$.
Therefore $\widetilde{f}_i$ is contracted by $\nu_{\widetilde{\mathcal{K}}}$ to the point $P_{i+1}$, which is a quadruple point whose tangent cone tangent is a cone over a Veronese surface, and the divisors $\mathcal{E}_{12i}'$, $\mathcal{E}_{13i}'$, $\mathcal{E}_{23i}'$ are contracted in three conics contained in the Veronese surface given by the exceptional divisor of the minimal resolution of $P_{i+1}$. 
\end{proof} 

We recall that $\nu_{\mathcal{K}} : \mathbb{P}^{3} \dashrightarrow  W_F^{9} \subset \mathbb{P}^{9}$ is an isomorphism outside $T\cup T'$ (see Remark~\ref{rem:veryample9}). Then $P_1$, $P_2$, $P_3$, $P_4$, $P_1'$, $P_2'$, $P_3'$ and $P_4'$ are the only singular points of $W_F^{9}$ (see Remarks~\ref{rem:contractionInpoints9},~\ref{rem:nu9blowdown},~\ref{rem:nocontractionSpigoli9}). Furthermore, $\nu_{\widetilde{\mathcal{K}}} : Y \to W_F^9$ is a desingularization of $W_F^9$ but it is not the minimal one: indeed, the proof of Proposition~\ref{prop:quadruplePointsp9} says us that $\nu_{\widetilde{\mathcal{K}}} : Y \to W_F^9$ is the blow-up of the minimal desingularization of $W_F^9$ along curves (conics) contained in the minimal resolutions of $P_2$, $P_3$, $P_4$, $P_2'$, $P_3'$ and $P_4'$. Finally we have the following result.

\begin{theorem}\label{thm:associatedo9}
Each singular point of $W_F^{9}$ is associated with \textit{at least} $m=4$ of the other singular points.
\end{theorem}
\begin{proof}
We know that the $42$ exceptional divisors of $bl''' : Y \to Y''$ are mapped by $\nu_{\widetilde{\mathcal{K}}} : Y \to W_F^{9}\subset \mathbb{P}^{9}$ to curves of $W_F^{9}$ (see Remark~\ref{rem:nu9blowdown}). In particular they are mapped to lines of $W_F^{9}$ (use similar arguments to the ones in the proof of Theorem~\ref{thm:associatedo13}).
Since $\Gamma_i\cap \mathcal{E} \ne \emptyset$ and $\Gamma_i\cap \widetilde{f}_i \ne \emptyset$ for $1\le i \le 3$, we have that $P_1'$ is associated with $P_{i+1}$. Similarly $P_1$ is associated with $P_{i+1}'$. 
One can verify that the other $36$ exceptional divisors are mapped to nine lines in the following way:
$$\left\langle P_{i+1},P_{j+1}' \right\rangle = \left\langle \nu_{\widetilde{\mathcal{K}}}(\mathcal{E}_{rsi}'), \nu_{\widetilde{\mathcal{K}}}(\mathcal{E}_{khj}) \right\rangle = \nu_{\widetilde{\mathcal{K}}}(\Lambda_{rsi,j}')=\nu_{\widetilde{\mathcal{K}}}(\Lambda_{khj,i})$$
for $i,j,r,s,k,h\in\{1,2,3\}$ and $r<s$ and $k<h$. 
It remains to show that $P_1= \nu_{\widetilde{\mathcal{K}}}(\mathcal{E}')$ is associated with $P_1'=\nu_{\widetilde{\mathcal{K}}}(\mathcal{E})$.
Let us consider the line $l_{vv'}:= \left\langle v,v' \right\rangle \subset \mathbb{P}^3$ joining the two vertices of the trihedra $T$ and $T'$. Let $\widetilde{l}_{vv'}$ be its strict transform on $Y$. We obtain that $\nu_{\widetilde{\mathcal{K}}}(\widetilde{l}_{vv'}) = \left\langle P_1, P_1' \right\rangle \subset W_F^{9}$, since $\widetilde{l}_{vv'} \cap \mathcal{E} \ne \emptyset$, $\widetilde{l}_{vv'} \cap \mathcal{E}' \ne \emptyset$ and $\deg (\nu_{\widetilde{\mathcal{K}}}(\widetilde{l}_{vv'}))=\widetilde{K}\cdot (\mathcal{H}-\mathcal{E}-\mathcal{E}'-\sum_{i=1}^3\Gamma_i-\sum_{i=1}^3\Gamma_i')^2=1$.
\end{proof}

\begin{remark}
Thanks to a computational anaylisis with Macaulay2, we can say that each singular point of $W_F^{9}$ is associated with \textit{exactly} $m=4$ of the other singular points, and that the lines joining them and contained in $W_F^{9}$ are just the ones of proof of Theorem~\ref{thm:associatedo9} (see \S~\ref{code:fano9} of Appendix~\ref{app:code}).
\end{remark}

\section{The Enriques-Fano threefold of genus 7}\label{subsec:Fano7}

%\subsection{Construction}\label{subsec:constructionF7}

Let us take a tetrahedron $T=\bigcup_{i=0}^3f_i\subset \mathbb{P}^3$ and let us denote by $v_i$ the vertex opposite to the face $f_i$, for $0\le i \le 3$. Let $l_{ij}$ be the edge $f_i\cap f_j$, for $0\le i<j \le 3$. Furthermore, let us fix a general plane $\pi$ of $\mathbb{P}^{3}$. The plane $\pi$ intersects each edge $l_{ij}$ of $T$ at one point, which is denoted by $p_{ij} := l_{ij}\cap \pi$. In the plane $\pi$ there is a $3$-dimensional linear system of cubic curves passing through the six points $p_{ij}$ (see \cite[\S 9.2.2]{Dolg12}). Let us fix a general element $\delta$ of this linear system: it is an elliptic smooth cubic plane curve. Let us consider the linear system $\mathcal{X}$ of the sextic surfaces in $\mathbb{P}^{3}$ double along the six edges of the tetrahedron $T$
and containing the cubic plane curve $\delta$. 

\begin{proposition}\label{prop:dimX=7}
The linear system $\mathcal{X}$ defined as above has $\dim \mathcal{X} = 7$.
\end{proposition} 
\begin{proof}
The linear system $\mathcal{X}$ is a sublinear system of the linear system $\mathcal{S}$ of the sextic surfaces double along the six edges of $T$. In particular we have that $\mathcal{S}=|\mathcal{I}_{\gamma^2 | \mathbb{P}^3}(6)|$ and $\mathcal{X}=|\mathcal{I}_{\gamma^2 \cup \delta | \mathbb{P}^3}(6)|$, where $\gamma$ is the sextic reducible curve given by the union of the edges of $T$.
We also have that $\mathcal{S}$ cuts on $\delta$ a complete linear system $|\mathcal{O}_{\delta}(D)|$. Indeed, we recall that $\mathcal{S}$ contains a sublinear system whose fixed part is given by the tetrahedron and whose movable part, given by the quadric surfaces of $\mathbb{P}^3$, cuts a complete linear system on $\delta$. Hence we have the following exact sequence
$$0 \to H^0(\mathcal{I}_{\gamma^2\cup \delta | \mathbb{P}^3}(6)) \to H^0(\mathcal{I}_{\gamma^2 | \mathbb{P}^3}(6)) \to H^0(\mathcal{O}_{\delta}(D))\to 0.$$
Let $\Sigma$ be a general element of $\mathcal{S}$. The cubic plane curve $\delta$ intersects $\Sigma$, outside the base locus of $\mathcal{S}$, in $3\cdot 6 - 2\cdot 6 = 6$ points. Hence $\deg D =6$. We recall that $\dim H^0(\mathbb{P}^3, \mathcal{I}_{\gamma^2 | \mathbb{P}^3}(6)) = \dim \mathcal{S}+1 = 14$. Since $\deg D = 6 > 2p_g(\delta)-2 = 0$, then $\dim H^1(\delta, \mathcal{O}_{\delta} (D))=0$ (see \cite[Example 1.3.4]{Hart}) and we have $\dim H^0(\delta, \mathcal{O}_{\delta} (D)) = 6$
%$-p_g(\delta)+1 = 6$ 
by Riemann-Roch. So the above exact seguence implies that 
$$\dim \mathcal{X} = \dim H^0(\mathbb{P}^3, \mathcal{I}_{\gamma^2\cup \delta | \mathbb{P}^3}(6)) -1= 14-6 -1= 7.$$
\end{proof}

\begin{remark}\label{rem:WF7projWF13}
Let $\nu_{\mathcal{S}} : \mathbb{P}^3 \dashrightarrow \mathbb{P}^{13}$ be the rational map defined by the linear system $\mathcal{S}$ of the sextic surfaces of $\mathbb{P}^3$ singular along the edges of $T$, whose image is the Enriques-Fano threefold $W_F^{13}$. Let $W_F^7$ be the image of $\mathbb{P}^3$ via the rational map defined by the linear system $\mathcal{X}$. Then $W_F^{7}$ is the projection of $W_F^{13}$ from the linear subspace of $\mathbb{P}^{13}$ spanned by the sextic elliptic curve $\nu_{\mathcal{S}}(\delta)$.  
\end{remark}

\begin{lemma}\label{lem:birationalitanu7}
The rational map $\nu_{\mathcal{X}} : \mathbb{P}^{3} \dashrightarrow \mathbb{P}^{7}$ defined by $\mathcal{X}$ is a birational map onto the image.
\end{lemma}
\begin{proof}
Let $X$ be a general element of $\mathcal{X}$.
The linear system $\mathcal{X}$ contains a sublinear system $\overline{\mathcal{X}} \subset \mathcal{X}$ 
whose fixed part is given by $T\cup \pi$ and such that $\overline{\mathcal{X}}|_{X}$ coincides with the linear system on $X$ cut out by the planes of $\mathbb{P}^3$. Then we obtain the birationality of the maps defined by $\mathcal{X}|_{X}$ and by $\mathcal{X}$.
\end{proof}

\begin{remark}\label{rem:veryample7}
The proof of Lemma~\ref{lem:birationalitanu7} tells us that the linear system $\mathcal{X}$ is very ample outside the tetrahedron $T$ and the plane $\pi$. So $\nu_{\mathcal{X}} : \mathbb{P}^{3} \dashrightarrow  \nu_{\mathcal{X}}(\mathbb{P}^3)\subset \mathbb{P}^{7}$ is an isomorphism outside $T\cup \pi$.
\end{remark}

Proposition~\ref{prop:dimX=7} proves the existence of sextic surfaces of $\mathbb{P}^3$ double along the edges of the tetrahedron $T$ (shortly, \textit{Enriques sextics}) and containing a given cubic plane curve $\delta$. However, a priori, these surfaces could have further singularities and their %normalizations
desingularizations could be not Enriques surfaces. Let us study the surfaces of $\mathcal{X}$.
Up to a change of coordinates, we can consider in $\mathbb{P}^3_{\left[ s_0:s_1:s_2:s_3 \right]}$ the tetrahedron $T=\{s_0s_1s_2s_3=0\}$ with faces $f_i = \{s_i=0\}$ for $0\le i\le 3$. Let us fix the cubic plane curve $\widehat{\delta}:=\{\sum_{i=0}^3s_i=0, s_1^2s_2 + s_1s_2^2 + s_1^2s_3 + s_1s_2s_3 + s_2^2s_3 + s_1s_3^2 + s_2s_3^2=0\}$, which intersects the edges of $T$ at one point each. 
Thanks to Macaulay2, we can construct the linear system $\widehat{\mathcal{X}}$ of the sextic surfaces of $\mathbb{P}^3$ which are singular along the edges of $T$ and which contain the curve $\widehat{\delta}$ (see Code~\ref{code:fano7} of Appendix~\ref{app:code}). Let us take $\widehat{X}:= \{s_0^2s_1^2s_2^2 - s_0^3s_1s_2s_3 - s_0s_1^3s_2s_3 - 2s_0s_1s_2^3s_3 + s_0^2s_1^2s_3^2 + 2s_0^2s_2^2s_3^2 + s_0s_1s_2^2s_3^2 + 2s_1^2s_2^2s_3^2 - 2s_0s_1s_2s_3^3=0\}\in \widehat{\mathcal{X}}$. By the computational analysis, we see that $\widehat{X}$ only has singular points along the edges of $T$. In particular the tangent cone to $\widehat{X}$ at a vertex of $T$ is given by the union of the three faces of $T$ containing that vertex; the tangent cone to $\widehat{X}$ at a point $p\in l_{ij}$, with $p\ne v_k$ and $p\ne v_h$, 
is the union of two planes containing $l_{ij}$, where $i,j,k,h$ are four distinct indices in $\{0,1,2,3\}$.
Then $\widehat{X}$ has ordinary singularities along the edges of $T$ (see Definition~\ref{def:ordinarysingularities}) and no further singularities. The same happens for a general surface of $\widehat{\mathcal{X}}$. Let $\mathcal{D}$ be the family of the cubic plane curves of $\mathbb{P}^3$ intersecting the edges of $T$ at one point each. We have that $\mathcal{D}$ is an irreducible variety of dimension $6$. Then what is true for the special cubic plane curve $\widehat{\delta}\in \mathcal{D}$ is also true for the general cubic plane curve $\delta\in \mathcal{D}$. Therefore there exist Enriques sextics in $\mathbb{P}^3$, with ordinary singularities along the edges of $T$ and no further singularities, that contain $\delta$. Let $X$ be such a general surface and let us take its minimal desingularization $n : X^{\nu} \to X$. It follows that $X^{\nu}$ is an Enriques surface (see \cite[p.275]{CoDo89}). Furthermore, let $E$ be the strict transform of $\delta$ on $X^{\nu}$. Then $E$ is an elliptic curve such that $E\cdot H = 3$, where $H$ is the pullback of the general hyperplane section of $X$. 
%We have that $H\sim E_1+E_2+E_3$ where 
%$E_1$, $E_2$, $E_3$ are the strict transform of three coplanar edges (we may assume 
%$E_i$ is strict transform of $l_{0i}$ for $1\le i \le 3$. 
If $\delta$ moved in a linear system on $X$, then $E$ would move in an elliptic pencil on $X^{\nu}$. Thus, $E$ should be $2$-divisible
%every elliptic pencil on an Enriques surface has two double elements 
(see \cite[Lemma 17.1]{BPV84}) and $E\cdot H=3$ would be a contradiction.
% sarebbe E \sim 2F \sim 2F' quindi E\cdot H = 2(F\cdot H) = 2(F'\cdot H) 
So $\delta$ does not move in any linear system on $X$ and a compute of parameters allows us to see that the general Enriques sextic in $\mathbb{P}^3$ contains some cubic plane curve of $\mathcal{D}$.
% e quindi posso supporre di aver presto \delta tale che la generica sestica di enriques sta in X... e quindi l'immagine è ovviemente una superficie di Enriques
These arguments prove that the image of a general $X\in \mathcal{X}$ via the rational map $\nu_{\mathcal{X}} : \mathbb{P}^3 \dashrightarrow \mathbb{P}^7$ is an Enriques surface. 
Finally one can prove that $W_F^{7}=\nu_{\mathcal{X}}(\mathbb{P}^3)\subset \mathbb{P}^{7}$ is not a cone over a general hyperplane section, as in the proof of Theorem~\ref{thm:WF13moderno}. 
Furthermore, $\deg W_F^7 = 12$ (see Code~\ref{code:fano7} of Appendix~\ref{app:code}). So 
%, up to a normalization, 
we obtain the following theorem.

\begin{theorem}\cite[\S 4]{Fa38}\label{thm:WF7 isEF}
The image of $\mathbb{P}^{3}$ via the rational map defined by $\mathcal{X}$ is an Enriques-Fano threefold $W_{F}^{7}\subset \mathbb{P}^7$ of genus $p=7$.
\end{theorem}

By Remark~\ref{rem:WF7projWF13}, the above theorem also follows by \cite[Lemma 4.4, Lemma 4.6]{CDGK20}. 

%\subsection{Singularities}

In order to describe the geometry and the singularities of $W_F^7$, we will use the techniques of the proof of Theorems~\ref{thm:WF13 isEF},~\ref{thm:WF9 isEF}.

\begin{remark}\label{rem:variabeleTC7}
Let $X$ be a general element of $\mathcal{X}$ and let us take four distinct indices $i,j,k,h\in\{0,1,2,3\}$. We have that
%\begin{itemize}
%\item[(i)] 
$TC_{v_i}X=f_j\cup f_k \cup f_h$ and,
%so $X$ has triple points at the vertices of $T$;
%\item[(ii)] 
if $p\in l_{ij}$ with $p\ne v_k$ and $p\ne v_h$, we have that $TC_{p}X$ is the union of two variable planes $\pi_{p,X}, \pi_{p,X}'\in |\mathcal{I}_{l_{ij}|\mathbb{P}^3}(1)|$ depending on the choice of $p$ and of $X$ and which can also coincide.
%\item[(iii)] 
In particular if $p=p_{ij}$, then one of the two planes of $TC_{p_{ij}}X$ is tangent to $\delta$ at $p_{ij}$ and we will denote this plane by $\pi_{ij}$. 
%\end{itemize}
\end{remark}

Let us blow-up $\mathbb{P}^3$ at the vertices of $T$ and at the six points $p_{ij}$, for $0\le i<j\le 3$. We obtain a smooth threefold $Y'$ and a birational morphism $bl' : Y '\to \mathbb{P}^{3}$ with exceptional divisors $E_i := (bl')^{-1}(v_i)$, $E_{ij} := (bl')^{-1}(p_{ij})$.
Let $\mathcal{X}'$ be the strict transform of $\mathcal{X}$ and let us denote by $H$ the pullback on $Y'$ of the hyperplane class on $\mathbb{P}^{3}$. Then an 
%general 
element of $\mathcal{X}'$ is linearly equivalent to $6H-3\sum_{i=0}^{3}E_i-2\sum_{0\le i<j\le 3}E_{ij}$.
Let $\widetilde{f}_i$ be the strict transform of the face $f_i$ and let $\widetilde{\pi}_{ij}$ be the strict transform of the plane $\pi_{ij}$ defined in Remark~\ref{rem:variabeleTC7}, for $0\le i<j\le 3$. We denote by $\gamma_{ki}:=E_k\cap \widetilde{f}_j$ the line cut out by $\widetilde{f}_i$ on $E_k$ and by $\lambda_{ij}:=E_{ij}\cap \widetilde{\pi}_{ij}$ the line cut out by $\widetilde{\pi}_{ij}$ on $E_{ij}$, for distinct indices $i,j,k\in\{0,1,2,3\}$ with $i< j$. We have that $\gamma_{ki}$ and $\lambda_{ij}$ are $(-1)$-curves respectively on $\widetilde{f}_i$ and $\widetilde{\pi}_{ij}$. 
Let $X'$ be the strict transform of a general $X\in\mathcal{X}$. By Remark~\ref{rem:variabeleTC7} 
%(i) 
we have that $X'\cap E_{k}=\bigcup_{\substack{i=0\\ i\ne k}}^3\gamma_{ki}$ for all $0\le k \le 3$.

\begin{remark}\label{rem:betap7}
We observe that $X'\cap E_{ij} = \lambda_{ij} \cup \beta_{ij,X}$, where $\beta_{ij,X}$ moves in the pencil of the lines of $E_{ij}$ through the point $E_{ij}\cap \widetilde{l}_{ij}$ and it depends on the choice of $X$, for all $0\le i<j\le 3$ (see Remark~\ref{rem:variabeleTC7}).
% (iii)
\end{remark}

Let us take the strict transforms $\widetilde{l}_{ij}$ of the six edges of $T$ and the strict transform $\widetilde{\delta}$ of the cubic plane curve $\delta$. The base locus of $\mathcal{X}'$ is given by the union of the six curves $\widetilde{l}_{ij}$ (along which a general $X'\in\mathcal{X}$' has double points), of the curve $\widetilde{\delta}$, of the twelve curves $\gamma_{ij}$ and the six curves $\lambda_{ij}$ (see Remark~\ref{rem:variabeleTC7}).
%(ii) 
Let us blow-up the strict transforms of the edges of $T$ and of the cubic plane curve $\delta$. We obtain a smooth threefold $Y''$ and a birational morphism $bl'' : Y'' \to Y'$ with exceptional divisors 
$(bl'')^{-1}(\widetilde{\delta})=: F_{\delta} \cong \mathbb{P}(\mathcal{N}_{\widetilde{\delta}|Y'})$ and
$$(bl'')^{-1}(\widetilde{l}_{ij})=: F_{ij}\cong \mathbb{P}(\mathcal{N}_{\widetilde{l}_{ij}|Y'}) \cong \mathbb{P}(\mathcal{O}_{\mathbb{P}^1}(-2)\oplus \mathcal{O}_{\mathbb{P}^1}(-2))\cong \mathbb{F}_0,$$
for $0\le i<j\le 3$. 

\begin{remark}\label{rem:degNormaldelta}
The divisor $F_{\delta}$ is a smooth elliptic ruled surface, since it is a $\mathbb{P}^1$-bundle over the elliptic curve $\widetilde{\delta}$. We also have that $\deg (\mathcal{N}_{\widetilde{\delta}|Y'})=0$. Indeed, since $\delta$ is the complete intersection of the plane $\pi$ and of a cubic surface passing through the points $p_{ij}$, then we have $\mathcal{N}_{\widetilde{\delta}|Y'}\cong \mathcal{O}_{\widetilde{\delta}}(H-\sum_{0\le i<j \le 3}E_{ij})\oplus \mathcal{O}_{\widetilde{\delta}}(3H-\sum_{0\le i<j \le 3}E_{ij})$ and $\deg (\mathcal{N}_{\widetilde{\delta}|Y'})=(3-6)+(9-6)=0$.
\end{remark}

Since $bl'':Y'' \to Y'$ has no effect on $\widetilde{f}_i$, we will use the same symbols to indicate its strict transforms on $Y''$; furthermore, let us denote by $\widetilde{E}_i$ and $\widetilde{E}_{ij}$ respectively the strict transforms of $E_{i}$ and $E_{ij}$, for $0\le i<j\le 3$.

\begin{remark}\label{rem:selfalphap7}
Let us take the curves $\alpha_{kij}:=\widetilde{E}_k\cap F_{ij}$, $\alpha_{ij}:=\widetilde{E}_{ij}\cap F_{ij}$, $\alpha_{ij}':=\widetilde{E}_{ij}\cap F_{\delta}$, for distinct indices $i,j,k\in\{0,1,2,3\}$ with $i<j$.
We have that $\alpha_{kij}$ is a $(-1)$-curve on $\widetilde{E}_k$ and a fibre on $F_{ij}$; $\alpha_{ij}$ is a $(-1)$-curve on $\widetilde{E}_{ij}$ and a fibre on $F_{ij}$;
$\alpha_{ij}'$ is a $(-1)$-curve on $\widetilde{E}_{ij}$ and a fibre on $F_{\delta}$.
\end{remark}

Let $\mathcal{X}''$ be the strict transform of $\mathcal{X}'$ and let $X''$ be an 
%general 
element of $\mathcal{X}''$. Then 
$$X''\sim 6H-3\sum_{i=0}^3 \widetilde{E}_i-2\sum_{0\le i<j\le 3} E_{ij}- 2\sum_{0\le i < j \le 3} F_{ij}-F_{\delta},$$ 
where, by abuse of notation, $H$ denotes the pullback $bl''^* H$.
By Remark~\ref{rem:variabeleTC7}, 
%(ii) 
we have that the base locus of $\mathcal{X}''$ is given by the disjoint union of the strict transforms $\widetilde{\gamma}_{ki}$ and $\widetilde{\lambda}_{ij}$ of the 18 lines $\gamma_{ki}$ and $\lambda_{ij}$, for distinct indices $i,j,k\in\{0,1,2,3\}$ and $i< j$.

\begin{remark}\label{rem:-1curvesX}
We have that 
$\widetilde{\gamma}_{ki}^2|_{\widetilde{E}_{k}}=-1$, $\widetilde{\gamma}_{ki}^2|_{\widetilde{f}_{i}}=-1$, $\widetilde{\lambda}_{ij}^2|_{\widetilde{E}_{ij}}=-1.$
Furthermore, if $X''$ is the strict transform of a general $X'\in\mathcal{X}'$, we also have that the twelve $\widetilde{\gamma}_{ki}$ are $(-1)$-curves on $X''$ for $i,k\in\{0,1,2,3\}$ and $i\ne k$ (see Remark~\ref{rem:-1curvesSigma}). Finally we want to show that the $6$ curves $\widetilde{\lambda}_{ij}$ are $(-1)$-curves on $X''$ too. We observe that $X''\cap \widetilde{E}_{ij}=\widetilde{\lambda}_{ij}\cup \widetilde{\beta}_{ij,X''}$, where $\widetilde{\beta}_{ij,X''}$ is the strict transform of the curve of Remark~\ref{rem:betap7}, which moves in a pencil and depends on $X''$.
Since $\widetilde{\lambda}_{ij}$ and $\widetilde{\beta}_{ij,X''}$ are disjoint, we have
$(\widetilde{\lambda}_{ij})^2|_{X''}+(\widetilde{\beta}_{ij,X''})^2|_{X''} = (X''\cap \widetilde{E}_{ij})^2|_{X''}=\widetilde{E}_{ij}^2\cdot X'' =\widetilde{E}_{ij}\cdot (\widetilde{E}_{ij}\cdot X'')=\widetilde{E}_{ij}\cdot \widetilde{\lambda}_{ij} +\widetilde{E}_{ij} \cdot \widetilde{\beta}_{ij,X''}.$
Hence $(\widetilde{\lambda}_{ij})^2|_{X''}= \widetilde{\lambda}_{ij}\cdot \widetilde{E}_{ij} = \widetilde{\pi}_{ij}\cdot \widetilde{E}_{ij}^2 = -1$.
\end{remark}

Finally let us consider $bl''' : Y \to Y''$ the blow-up of $Y''$ along the above $18$ curves, with exceptional divisors $\Gamma_{ki}:=bl'''^{-1}(\widetilde{\gamma}_{ki})$, $\Lambda_{ij}:=bl'''^{-1}(\widetilde{\lambda}_{ij})$, for distinct indices $i,j,k\in\{0,1,2,3\}$ with $i<j$. We denote by $\mathcal{E}_{i}$ and $\mathcal{E}_{ij}$ respectively the strict transform of $\widetilde{E}_i$ and $\widetilde{E}_{ij}$; by $\mathcal{F}_{ij}$ the strict transform of $F_{ij}$; by $\mathcal{F}_{\delta}$ the strict transform of $F_{\delta}$; by $\mathcal{H}$ the pullback of $H$.

\begin{remark}\label{rem:Gamma^3p7}
We have that
$$\Gamma_{ki}=\mathbb{P}(\mathcal{N}_{\widetilde{\gamma}_{ki}|Y''})\cong \mathbb{P}(\mathcal{O}_{\widetilde{\gamma}_{ki}}(\widetilde{E}_k)\oplus \mathcal{O}_{\widetilde{\gamma}_{ki}}(\widetilde{f}_{i})) \cong \mathbb{P}(\mathcal{O}_{\mathbb{P}^1}(-1)\oplus \mathcal{O}_{\mathbb{P}^1}(-1))\cong \mathbb{F}_0,$$
$$\Lambda_{ij}=\mathbb{P}(\mathcal{N}_{\widetilde{\lambda}_{ij}|Y''})\cong \mathbb{P}(\mathcal{O}_{\widetilde{\lambda}_{ij}}(\widetilde{E}_{ij})\oplus \mathcal{O}_{\widetilde{\lambda}_{ij}}(\widetilde{\pi}_{ij})) \cong \mathbb{P}(\mathcal{O}_{\mathbb{P}^1}(-1)\oplus \mathcal{O}_{\mathbb{P}^1}(-1))\cong \mathbb{F}_0.$$
Furthermore, we have $\Gamma_{ki}^{3}=-\deg (\mathcal{N}_{\widetilde{\gamma}_{ki}|Y''})=2$, and 
$\Lambda_{ij}^{3}=-\deg (\mathcal{N}_{\widetilde{\lambda}_{ij}|Y''})=2$
(see \cite[Chap 4, \S 6]{GH} and \cite[Lemma 2.2.14]{IsPro99}).
\end{remark}

\begin{remark}\label{rem:intersectionYp7}
Let us take distinct indices $i,j,k\in\{0,1,2,3\}$ with $i< j$. The divisor $\mathcal{F}_{ij}$ intersects $\Gamma_{ki}$, $\Gamma_{kj}$, $\Lambda_{ij}$ each along a $\mathbb{P}^1$ which is a $(-1)$-curve on $\mathcal{F}_{ij}$ and a fibre on $\Gamma_{ki}$, $\Gamma_{kj}$, $\Lambda_{ij}$. Similarly we have $\Lambda_{ij}^2\cdot\mathcal{F}_{\delta}=-1$ and $\Lambda_{ij}\cdot\mathcal{F}_{\delta}^2=0.$
Let us consider the strict transforms $\widetilde{\alpha}_{kij}$, $\widetilde{\alpha}_{ij}$, $\widetilde{\alpha}_{ij}'$ of the curves defined in Remark~\ref{rem:selfalphap7}. Then we have 
$$\widetilde{\alpha}_{kij}^2|_{\mathcal{E}_k}=\mathcal{F}_{ij}^2\cdot \mathcal{E}_k=-1, \quad \widetilde{\alpha}_{ij}^2|_{\mathcal{E}_{ij}}={\mathcal{F}_{ij}}^2\cdot \mathcal{E}_{ij}=-1, \quad 
\widetilde{\alpha}_{ij}'^2|_{\mathcal{E}_{ij}'}=\mathcal{F}_{\delta}^2\cdot\mathcal{E}_{ij}=-1,$$
$$\widetilde{\alpha}_{kij}^2|_{\mathcal{F}_{ij}}=\mathcal{E}_k^2\cdot \mathcal{F}_{ij}=-2,\quad
\widetilde{\alpha}_{ij}^2|_{\mathcal{F}_{ij}}=\mathcal{E}_{ij}^2\cdot \mathcal{F}_{ij}=-1,\quad \widetilde{\alpha}_{ij}^2|_{\mathcal{F}_{\delta}^2}=\mathcal{E}_{ij}^2\cdot \mathcal{F}_{\delta}=-1.$$
Finally we recall that a general line of $\mathbb{P}^3$ does not intersect the edges of $T$ and the curve $\delta$; instead a general plane of $\mathbb{P}^3$ intersects each edge of $T$ at one point and the curve $\delta$ at $3$ points. Hence we have $\mathcal{H}^2\cdot \mathcal{F}_{ij}=\mathcal{H}^2\cdot \mathcal{F}_{\delta}=0$, $\mathcal{F}_{ij}^2\cdot\mathcal{H}=-1$ and $\mathcal{F}_{\delta}^2\cdot\mathcal{H}=-3$.
\end{remark}

\begin{remark}\label{rem:Ek^3p7}
We recall that by construction we have
${bl'''}^*(E_k) = \mathcal{E}_k +\sum_{\substack{t=0\\ t\ne k}}^3 \Gamma_{kt}$
and ${bl'''}^*(E_{ij}) = \mathcal{E}_{ij} + \Lambda_{ij}$, 
and, by abuse of notation, we denote $\mathcal{E}_{k}\cap \Gamma_{ki}$ and $\mathcal{E}_{ij}\cap \Lambda_{ij}$ by $\widetilde{\gamma}_{ki}$ and $\widetilde{\lambda}_{ij}$, for distinct $i,j,k\in\{1,2,3\}$ with $i<j$. 
Let us denote by $\mathcal{L}_{ki}$ and $\mathcal{L}_{ij}$ respectively the strict transform on $Y$ of a general line of $E_{k}$ and $E_{ij}$.   
By using similar arguments to the ones in Remark~\ref{rem:Ek^3p9} we obtain that $\mathcal{E}_{k}^3=4$ and $\mathcal{E}_{ij}^3=2$, since we have
$$\mathcal{E}_k|_{\mathcal{E}_k} \sim - ( \mathcal{L}_k + \sum_{\substack{i=0 \\ i\ne k}}^3\widetilde{\gamma}_{ki} ) \sim -(4\mathcal{L}_k-2\sum_{0\le i<j\le 3}\widetilde{\alpha}_{ij}),$$
$$\mathcal{E}_{ij}|_{\mathcal{E}_{ij}} \sim - ( \mathcal{L}_{ij} + \widetilde{\lambda}_{ij}) \sim -(2\mathcal{L}_{ij}-\widetilde{\alpha}_{ij}-\widetilde{\alpha}_{ij}').$$
\end{remark}

\begin{remark}\label{rem:Fij^3p7}
By using similar arguments to the ones in Remark~\ref{rem:Fij^3p13} we have $\mathcal{F}_{ij}^3=-\deg (\mathcal{N}_{\widetilde{l}_{ij}|Y'})=4$, for $0\le i<j\le3$, and $\mathcal{F}_{\delta}^3=-\deg (\mathcal{N}_{\widetilde{\delta}|Y'})=0$ (see Remark~\ref{rem:degNormaldelta}).
\end{remark}

Let $\widetilde{X}$ be the strict transform on $Y$ of an 
%general 
element of $\mathcal{X}''$: then
$$\widetilde{X} \sim 6\mathcal{H}-\sum_{k=0}^3 3\mathcal{E}_k-2\sum_{0\le i<j \le 3}\mathcal{E}_{ij}-2\sum_{0\le i<j \le 3}\mathcal{F}_{ij}-\mathcal{F}_{\delta}-\sum_{ \substack{ i,k=0 \\ i\ne k } }^3  4\Gamma_{ki}-\sum_{0\le i<j \le 3}3\Lambda_{ij}.$$
Let us take the linear system $\widetilde{\mathcal{X}}:=|\mathcal{O}_{Y}(\widetilde{X})|$ on $Y$. It is base point free and it defines a birational morphism $\nu_{\widetilde{\mathcal{X}}}: Y \to W_F^{7} \subset \mathbb{P}^{7}$. 
Furthermore, we have the following diagram:

$$\begin{tikzcd}
Y \arrow[d, "bl'''"] \arrow[drrr, "\nu_{\widetilde{\mathcal{X}}}"] & & & \\
Y''  \arrow{r}{bl''} & Y' \arrow{r}{bl'} & \mathbb{P}^3 \arrow[dashrightarrow]{r}{\nu_{\mathcal{X}}} & W_F^{7} \subset \mathbb{P}^{7}.
\end{tikzcd}$$

\begin{remark}\label{rem:contractionInpoints7}
The divisors $\mathcal{E}_i$ and the strict transforms $\widetilde{f}_{i}$ on $Y$ of the faces of $T$ are contracted by $\nu_{\widetilde{\mathcal{X}}} : Y \to W_F^{7} \subset \mathbb{P}^{7}$ to points of $W_F^{7}$, for $0\le i \le 3$, since $\widetilde{X}\cdot \mathcal{E}_i=\widetilde{X}\cdot \widetilde{f}_i=0$ for a general $\widetilde{X}\in\widetilde{\mathcal{X}}$.
\end{remark}

\begin{remark}\label{rem:nu7blowdown}
The 18 exceptional divisors of $bl''': Y \to Y''$ and the six divisors $\mathcal{E}_{ij}$, for $0\le i,j\le 3$, are contracted by the morphism $\nu_{\widetilde{\mathcal{X}}} : Y \to W_{F}^{7} \subset \mathbb{P}^{7}$ to curves of $W_{F}^{7}$.
This follows by the fact that $\widetilde{X}\cdot \Gamma_{ki}$, $\widetilde{X}\cdot \Lambda_{ij}$, $\widetilde{X}\cdot \mathcal{E}_{ij} \ne 0$ and
$\widetilde{X}^2\cdot \Gamma_{ki} = \widetilde{X}^2\cdot \Lambda_{ij} = \widetilde{X}^2\cdot \mathcal{E}_{ij} = 0$ for distinct $i,j,k\in\{0,1,2,3\}$ and $i< j$ (use Remarks~\ref{rem:Gamma^3p7},~\ref{rem:intersectionYp7}).
\end{remark}

\begin{remark}\label{rem:noContractionSpigoli7}
Let us fix $0\le i<j\le 3$ and let $\widetilde{X}$ be a general element of $\widetilde{\mathcal{X}}$. Since 
$\widetilde{X}^2\cdot \mathcal{F}_{ij}=3>0$ and $\widetilde{X}^2\cdot \mathcal{F}_{\delta}=6>0$ (use Remarks~\ref{rem:intersectionYp7},~\ref{rem:Fij^3p7}), then the curves $\widetilde{X}\cap \mathcal{F}_{ij}$ and $\widetilde{X}\cap \mathcal{F}_{\delta}$ are not contracted by the rational map defined by $\widetilde{\mathcal{X}}|_{\widetilde{X}}$. 
\end{remark}

We still define $P_{i+1} : = \nu_{\widetilde{\mathcal{X}}}(\mathcal{E}_i)$ and $P_{i+1}' : = \nu_{\widetilde{\mathcal{X}}}(\widetilde{f}_{i})$, as in \S~\ref{subsec:Fano13}. They are quadruple points of $W_F^{7}$ whose tangent cone is a cone over a Veronese surface. The proof is similar to the one of Proposition~\ref{prop:quadruplePointsp13}.
We recall that $\nu_{\mathcal{X}} : \mathbb{P}^{3} \dashrightarrow  W_F^7 \subset \mathbb{P}^{7}$ is an isomorphism outside $T\cup \pi$ (see Remark~\ref{rem:veryample7}). Then $P_1$, $P_2$, $P_3$, $P_4$, $P_1'$, $P_2'$, $P_3'$ and $P_4'$ are the only singular points of $W_F^{7}$ (see Remarks~\ref{rem:contractionInpoints7},~\ref{rem:nu7blowdown},~\ref{rem:noContractionSpigoli7}).

\begin{lemma}\label{lem:6linesp7}
The six divisors $\mathcal{E}_{ij}$, with $0\le i<j\le 3$, are mapped by the morphism $\nu_{\widetilde{\mathcal{X}}} : Y \to W_F^{7} \subset \mathbb{P}^{7}$ to lines of $W_F^{7}$. In particular we have $\nu_{\widetilde{\mathcal{X}}}(\mathcal{E}_{ij}) = \left\langle P_{i+1}', P_{j+1}' \right\rangle.$
\end{lemma}
\begin{proof}
We know that the above $6$ divisors are mapped by $\nu_{\widetilde{\mathcal{X}}} : Y \to W_F^{7}\subset \mathbb{P}^{7}$ to curves (see Remark~\ref{rem:nu7blowdown}). Let us show that these curves are lines. Let $\widetilde{X}$ be a general element of $\widetilde{\mathcal{X}}$ and let us consider the divisor $\mathcal{E}_{ij}$ for a fixed pair of indices $0\le i<j\le 3$. We observe that $\widetilde{\mathcal{X}}|_{\mathcal{E}_{ij}}\cong |\mathcal{O}_{\mathcal{E}_{ij}}(\widetilde{\beta}_{ij,\widetilde{X}})|\cong \mathbb{P}^1$ (see Remark~\ref{rem:betap7}), so $\nu_{\widetilde{\mathcal{X}}}(\mathcal{E}_{ij})\subset W_F^{7}$ is a line. Since $\mathcal{E}_{ij}\cap \widetilde{f}_i \ne \emptyset$ and $\mathcal{E}_{ij}\cap \widetilde{f}_j\ne \emptyset$, then $\nu_{\widetilde{\mathcal{X}}}(\mathcal{E}_{ij})$ is the line joining the points $P_{i+1}'$ and $P_{j+1}'$.
\end{proof} 

\begin{theorem}\label{thm:associatedo7}
Each of the eight points $P_1$, $P_2$, $P_3$, $P_4$, $P_1'$, $P_2'$, $P_3'$, $P_4'$ is associated with $m=6$ of the others.
\end{theorem}
\begin{proof}
The $12$ divisors $\Gamma_{ki}$, for $i,k\in\{0,1,2,3\}$ and $i\ne k$, are mapped by $\nu_{\widetilde{\mathcal{X}}} : Y \to W_F^{7}\subset \mathbb{P}^{7}$ to lines of $W_F^{7}$ joining the points $P_1$, $P_2$, $P_3$, $P_4$, $P_1'$, $P_2'$, $P_3'$, $P_4'$ as in the proof of Theorem~\ref{thm:associatedo13}.  We also recall that the six lines $\nu_{\widetilde{\mathcal{X}}}(\mathcal{E}_{ij})$ joins the points $P_1'$, $P_2'$, $P_3'$, $P_4'$ two by two (see Lemma~\ref{lem:6linesp7}). It remains to show that
$\left\langle P_1, P_2 \right\rangle$, $\left\langle P_1, P_3 \right\rangle$, $\left\langle P_1, P_4 \right\rangle$, $\left\langle P_2, P_3 \right\rangle$, $\left\langle P_2, P_4 \right\rangle$, $\left\langle P_3, P_4 \right\rangle \subset W_F^7$ and that
$\left\langle P_1, P_1' \right\rangle$, $\left\langle P_2, P_2' \right\rangle$, $\left\langle P_3, P_3' \right\rangle$, $\left\langle P_4, P_4' \right\rangle, \not\subset W_F^7$. This follows by a computational analysis with Macaulay2 (see Code~\ref{code:fano7} of Appendix~\ref{app:code}).
\end{proof}

\section{The Enriques-Fano threefold of genus 6}\label{subsec:Fano6}

%\subsection{Construction}

Let us consider five general points $q_1$, $q_2$, $q_3$, $q_4$, $q_5$ in $\mathbb{P}^{3}$. We have the following result.

\begin{theorem}\label{thm:C1C2C3}
There are three twisted cubics $C_1$, $C_2$ and $C_3$, three quadric surfaces $Q_6$, $Q_7$ and $Q_8$ of $\mathbb{P}^3$ and three lines $r_1$, $r_2$, and $r_3$ such that $Q_6$ and $Q_7$ are smooth and
$$C_1\cap C_2\cap C_3 = \{q_1,\dots q_5\},\,\, Q_6\cap Q_7 = C_1\cup r_1,\,\, Q_6\cap Q_8 = C_2\cup r_2,\,\, Q_7\cap Q_8 = C_3\cup r_3,$$
where $r_i$ intersects $C_i$ at two
%distinct 
points $a_i'$ and $a_i''$, for $1\le i \le 3$. Furthermore, by taking three distinct indices $i,j,k\in\{1,2,3\}$ with $i<j$, we have the following three possibilities:
\begin{itemize}
\item[(i)] $Q_8$ is smooth and the three lines $r_1$, $r_2$ and $r_3$ intersect pairwise at three distinct points $b_{ij}:= r_i\cap r_j$ such that $b_{ij} = r_i\cap C_k = r_j\cap C_k$;
\item[(ii)] $Q_8$ is smooth and the three lines $r_1$, $r_2$, $r_3$ intersect at a same point $b$; moreover, up to renaming the points of $r_k\cap C_k$ and of $C_1\cap C_2\cap C_3$, we have that $r_k\cap C_i = r_k\cap C_j = a_k'' = q_k$;
\item[(iii)] $Q_8$ is a cone and the three lines $r_1$, $r_2$, $r_3$ intersect at the vertex $v$ of $Q_8$; moreover, up to renaming the points of $r_k\cap C_k$ and of $C_1\cap C_2\cap C_3$, we have that $v=q_1=a_k'' = r_k\cap C_i = r_k \cap C_j$.
\end{itemize}
\end{theorem}

\begin{proof}
The proof is divided into 6 steps, given by the Lemmas~\ref{lem:C1C2C3}, ~\ref{lem:riCHORD}, ~\ref{lem:twistedCubicInCone}, ~\ref{lem:Q8notcone}, ~\ref{lem:Q8smoothb}, ~\ref{lem:Q8coneb} below.

\begin{lemma}\label{lem:C1C2C3}
There are three twisted cubics $C_1$, $C_2$ and $C_3$ passing through the five general points $q_1,\dots ,q_5$ and there are three quadric surfaces $Q_6$, $Q_7$ and $Q_8$ and three lines $r_1$, $r_2$, and $r_3$ such that $Q_6$ and $Q_7$ are smooth and $Q_6\cap Q_7 = C_1\cup r_1$, $Q_6\cap Q_8 = C_2\cup r_2$, $Q_7\cap Q_8 = C_3\cup r_3.$ 
\end{lemma}
\begin{proof}
Let us consider the two-dimensional family $\mathcal{C}$
of the twisted cubics passing through $q_1,\dots , q_5$.
Let us take a general twisted cubic $C_1\in \mathcal{C}$ and two smooth quadric surfaces $Q_6$, $Q_7$ containing $C_1$, i.e. two general elements $Q_6,Q_7 \in |\mathcal{I}_{C_1 | \mathbb{P}^3}(2)|\cong \mathbb{P}^2$. It is known that there exists a line $r_1$ such that
$Q_6 \cap Q_7 = C_1 \cup r_1$ (see \cite[Example 1.11]{Harris}).
Since $Q_6$ is a smooth quadric surface in $\mathbb{P}^3$, then 
%$\operatorname{Pic}(Q_6)=\mathbb{Z}\oplus \mathbb{Z}$ and 
%every divisor $D$ on $Q_6$ is of type $D\sim af_1+bf_2$, where $f_1$ and $f_2$ represent the two rulings of $Q_6$ and satisfy the relations $f_1^2 = 0=f_2^2$ and $f_1\cdot f_2=1$. We recall that 
a quadric section of $Q_6\subset \mathbb{P}^3$ is linearly equivalent to $2f_1+2f_2$, where $f_1$ and $f_2$ represent the two rulings of $Q_6$ and satisfy the relations $f_1^2 = 0=f_2^2$ and $f_1\cdot f_2=1$.
In particular, since $Q_6\cap Q_7 \sim 2f_1+2f_2$, we can suppose that $r_1\sim f_1$ and $C_1 \sim f_1+2f_2.$
We also have that $|\mathcal{O}_{Q_6}(f_1+2f_2)|\cong \mathbb{P}^5$.
%: indeed, since $3f_1+4f_2$ is an ample divisor (see \cite[V, Example 1.10.1]{Hart}), then 
%$$h^i (Q_6,\mathcal{O}_{Q_6}(f_1+2f_2)) = h^i (Q_6,\mathcal{O}_{Q_6}(K_{Q_6}+(3f_1+4f_2)))=0 \, \text{ for } i=1,2$$ 
%by Kodaira Vanishing Theorem, 
%%(see \textbf{for example} \cite[III, Remark 7.5]{Hart}), 
%and by Riemann-Roch we have $$h^0(Q_6, \mathcal{O}_{Q_6}(f_1+2f_2))=6.$$
By the generality of $q_1, \dots , q_5$ we may assume that $C_1$ is the unique twisted cubic in $|\mathcal{O}_{Q_6}(f_1+2f_2)|$ through $q_1, \dots , q_5$. 
% nel senso che impone 5 condizioni al \mathbb{P}^5 e quindi ottiene un \mathbb{P}^0
Similarly let us take the unique twisted cubic $C_2$ through $q_1, \dots , q_5$ in $|\mathcal{O}_{Q_6}(2f_1+f_2)|$. Therefore,
%Furthermore we observe that $C_1$ and $C_2$ are the unique twisted cubics through $q_1, \dots , q_5$ contained in $Q_6$. Moreover 
each smooth quadric surface passing through $q_1, \dots , q_5$ contains exactly two twisted cubics passing through them: let $C_3$ be the other twisted cubic in $Q_7$ through $q_1, \dots , q_5$.
Let us define $\Lambda_{C_i} := |\mathcal{I}_{C_i | \mathbb{P}^3}(2)| \cong \mathbb{P}^2$ for $i=1,2,3$. Since $\Lambda_{C_i} \subset |\mathcal{I}_{\{q_1,\dots , q_5\} | \mathbb{P}^3}(2)|\cong \mathbb{P}^4$ for $i=1,2,3$, then $\dim \Lambda_{C_2} \cap \Lambda_{C_3} \ge 0$. So there exists a quadric surface $Q_8\in \Lambda_{C_2} \cap \Lambda_{C_3}$ such that $C_2 \subset Q_6\cap Q_8$ and $C_3 \subset Q_7 \cap Q_8$, and there are two lines $r_2$ and $r_3$ such $Q_6\cap Q_8=C_2\cup r_2$ and $Q_7 \cap Q_8 = C_3\cup r_3$.
\end{proof}

Let us fix now three twisted cubics, three lines and three quadric surfaces as in Lemma~\ref{lem:C1C2C3}. It must be $r_1\not\subset Q_8$, $r_2\not\subset Q_7$ and $r_3\not\subset Q_6$.
Let us take $i,j,k\in\{1,2,3\}$ with $k\ne i$ and $i<j$. By construction we have $r_i\cdot C_i = 2$, so we let $\{a_i', a_i''\}:= r_i\cap C_i$. 
%where $a_i'\ne a_i''$. 
Furthemore $r_i\cdot C_k = 1$, so we define $a_{ik}:=r_i\cap C_k$.

\begin{lemma}\label{lem:riCHORD}
The line $r_i$ intersects the line $r_j$ for all $1\le i< j\le 3$.
\end{lemma}
\begin{proof}
By construction we have that $r_1\cap r_2 \ne \emptyset$ and $r_1\cap r_3 \ne \emptyset$. Furthermore, it must be $r_2\cap r_3 \ne \emptyset$. Indeed, if $Q_8$ is a cone, then $r_2$ and $r_3$ intersect at the vertex; if $Q_8$ is smooth, then $C_2$ and $C_3$ are not linearly equivalent and $r_2$ and $r_3$ belong to different rulings.
\end{proof}

By Lemma~\ref{lem:riCHORD} we have two possibilities: the three lines $r_1$, $r_2$ and $r_3$ intersect pairwise at $3$ distinct points $\{r_1\cap r_2, r_1\cap r_2, r_2\cap r_3\}$ or they intersect at a same point $r_1\cap r_2 \cap r_3$.

\begin{lemma}\label{lem:twistedCubicInCone}
Let $Q\subset \mathbb{P}^3$ be a quadric cone with vertex $v$. If $C$ is a twisted cubic contained in $Q$, then $v\in C$.
\end{lemma}
\begin{proof}
Let us suppose that $v\not\in C$. Let $H$ be a general plane of $\mathbb{P}^3$ such that $v\not\in H$ and let us take the projection map $\pi_{v} : \mathbb{P}^3 \dashrightarrow H\cong \mathbb{P}^2$ from the point $v$ to the plane $H$. Since $v\not\in C$, then $\pi_{v}(C)$ is a cubic plane curve. Furthermore, $\pi_{v}(C)$ has to be contained in $\pi_{v}(Q)$, which is a conic. So we have a contradiction.
\end{proof}

\begin{lemma}\label{lem:Q8notcone}
Let us suppose that the three lines $r_1$, $r_2$ and $r_3$ intersect pairwise at $3$ distinct points and let us denote them by $b_{ij}:=r_i\cap r_j$ for $1\le i<j\le 3$. Then the quadric surface $Q_8\subset \mathbb{P}^3$ is smooth and we have $b_{ij} = r_i\cap C_k = r_j\cap C_k$ for all $1\le k \le 3$ such that $k\ne i$ and $k\ne j$.
\end{lemma}
\begin{proof}
Let us suppose that $Q_8$ is a cone with vertex $v$. Then $r_2$, $r_3$, $C_2$ and $C_3$ must pass through $v$ (see Lemma~\ref{lem:twistedCubicInCone}). In particular there exists a point in $r_2\cap C_2$, a point in $r_3\cap C_3$ and a point in $\{q_1,\dots , q_5\}$, for example $a_2''$, $a_3''$ and $q_5$, such that $v=a_2''=a_3''=b_{23}=a_{21}=a_{31}=a_{23}=a_{32}=q_{5}$. 
%(see Figure~\ref{fig:Q8cone}). 
Moreover, we have that $r_1\cap Q_8 = (r_1\cap Q_6)\cap Q_8 = r_1\cap (C_2\cup r_2) = \{a_{12},b_{12}\}$, since $r_1\subset Q_6$. Similarly $r_1\cap Q_8= \{a_{13},b_{13}\}$, since $r_1\subset Q_7$ (see Lemma~\ref{lem:C1C2C3}). Since $b_{12}\neq b_{13}$ by hypothesis, it must be $b_{12}=a_{13}$ and $b_{13}=a_{12}$. This implies $b_{12}=r_{2}\cap C_3 =v$ and $b_{13}=r_{3}\cap C_2 =v$, which is a contradiction. Hence $Q_8$ is a smooth quadric surface of $\mathbb{P}^3$.
Finally we observe that 
$r_1\cap Q_8 = \{b_{13},a_{13}\}=\{ b_{12}, a_{12}\}$, $r_2\cap Q_7 = \{b_{12},a_{21}\}=\{b_{23},a_{23}\}$ and $r_3\cap Q_6 = \{b_{13}, a_{31}\} = \{b_{23},a_{32}\}$. Since $b_{12}$, $b_{13}$, $b_{23}$ are three distinct points by hypothesis, then it must be 
$b_{13}=a_{12}=a_{32}$, $b_{12}=a_{13}=a_{23}$ and $b_{23}=a_{21}=a_{31}$.
\end{proof}

\begin{lemma}\label{lem:Q8smoothb}
Let us suppose that the three lines $r_1$, $r_2$ and $r_3$ intersect at the same point $b$. If $Q_8$ is smooth, then we obtain the assertion (ii) of Theorem~\ref{thm:C1C2C3}.
\end{lemma}
\begin{proof}
Let $i,j,k$ be three distinct indices in $\{1,2,3\}$. Since $r_{k}\subset Q_{k+i+3}$, we have that $r_{k}\cap Q_{i+j+3}= (r_k\cap Q_{k+i+3}) \cap Q_{i+j+3}= r_k\cap (C_i\cup r_i) = \{a_{ki},b\}$ (see Lemma~\ref{lem:C1C2C3}). Similarly we have that $r_{k}\cap Q_{i+j+3} = \{a_{kj},b\}$, since $r_k\subset Q_{k+j+3}$. So it must be $a_{ki}=a_{kj}$. This implies that there are three points in $\{q_1,\dots , q_5\}$, namely $q_1$, $q_2$, $q_3$, and there exists a point in $r_k\cap C_k$, for example $a_k''$, such that $a_{ki}=a_{kj}=a_k''=q_k$.
\end{proof}

\begin{lemma}\label{lem:Q8coneb}
Let us suppose that the three lines $r_1$, $r_2$ and $r_3$ intersect at the same point $b$. If $Q_8$ is a cone, then we obtain the assertion (iii) of Theorem~\ref{thm:C1C2C3}.
\end{lemma}
\begin{proof}
If $Q_8$ is a cone with vertex $v$, then $v=r_2\cap r_3 = r_1\cap r_2\cap r_3=b$. Since $C_2, C_3 \subset Q_8$, then $v \in C_2\cap C_3 = C_1\cap C_2\cap C_3=\{q_1, \dots , q_5\}$ (see Lemma~\ref{lem:twistedCubicInCone}). 
% mi era venuto un dubbio... io ho C1\cap C2 = 5 punti e C1\cap C3 = stessi 5 punti
% e so anche che C1\cap C2\cap C3 = stessi 5 punti
% chi mi dice però che C2\cap C3 = sempre in questi 5 punti?
% ovviamente C2\cap C3 contiene C1\cap C2 \cap C3 = 5 punti
% ma chi mi dice che non ce ne siano altri?
% secondo me dipende dal fatto che una cubica in PP3 e' univocamente
% determinata da 6 punti... quindi se avessi C2\cap C3 >= 6 punti
% allora C2 e C3 devono coincidere, mentre sono distinte per costruzione
Thus, we have $v=a_{12}=a_{13}=a_{21}=a_{23}=a_{31}=a_{32}$. Furthermore, there exist a point in each $r_1\cap C_1$, $r_2\cap C_2$, $r_3\cap C_3$, for example $a_1''$, $a_2''$, $a_3''$, such that $v=a_{1}''=a_{2}''=a_{3}''$. 
\end{proof} 

\end{proof}

Let us see that by choosing sufficiently general objects, we can exclude the cases (ii) and (iii) of Theorem~\ref{thm:C1C2C3}. 
Let us take the two-dimensional family $\mathcal{C}_{q_1, \dots ,q_5}$ of the twisted cubics of $\mathbb{P}^3$ passing through the fixed points $q_1,\dots , q_5$. For all $C\in \mathcal{C}_{q_1, \dots , q_5}$ we define $\Lambda_C := |\mathcal{I}_{C|\mathbb{P}^3}(2)|\cong \mathbb{P}^2$, which is a plane in $|\mathcal{I}_{\{q_1, \dots , q_5\} | \mathbb{P}^3}(2)|\cong \mathbb{P}^4$. 
%We also consider $l_h:=\{Q\in |\mathcal{I}_{\{q_1, \dots , q_5\} | \mathbb{P}^3}(2)| Q \text{ is singular at } q_h\}\cong \mathbb{P}^1$ for $1\le h\le 5$. Each of the lines $l_1, \dots ,l_5$ contains points correspondent to reducible quadric surfaces given by the union of two planes; these kind of surfaces cannot be elements of any $\Lambda_C$, since the twisted cubics are non-degenerate.
%% quindi $l_h$ non puo' essere contenuto in $\Lambda_C$ e dunque: 
%Hence the lines $l_1, \dots , l_5$ intersect every $\Lambda_C$ at most at finitely many points. 
%In particular, 
We recall that, if we fix a general $C \in \mathcal{C}_{q_1, \dots , q_5}$ and if $Q\in \Lambda_C$ is general, then $Q$ is smooth and contains exactly two twisted cubics $C,C'\in C_{q_1, \dots , q_5}$ (see proof of Lemma~\ref{lem:C1C2C3}). We can consider the map $\varphi_{C} : \Lambda_{C} \to \mathcal{C}_{q_1, \dots , q_5}$ which sends a general $Q\in \Lambda_{C}$ to the other twisted cubic $C'$ in $Q$ passing through $q_1, \dots ,q_5$. This map is well defined and it has fibres of dimension $0$: indeed, by Bezout's Theorem we have that two quadric surfaces of $\mathbb{P}^3$ intersecting along $C\cup C'$ have to coincide, since $C\cup C'$ is a curve of degree $6$. Hence $\varphi_{C}$ is a birational map. In other words, the correspondence $C' \leftrightarrow Q$ is $1:1$ between an open set of $\mathcal{C}_{q_1, \dots ,q_5}$ and an open set of $\Lambda_{C}$.
Let us \textit{fix} now a \textit{general} $C_1\in \mathcal{C}_{q_1, \dots ,q_5}$ and a \textit{general} smooth quadric surface $Q_6\in \Lambda_{C_1}$. Then $C_2:=\varphi_{C_1}(Q_6)\in \mathcal{C}_{q_1, \dots , q_5}$ is fixed too, since it is uniquely determined by $Q_6$. Let us take another \textit{general} $Q_7\in \Lambda_{C_1}$, which is another smooth quadric surface of $\mathbb{P}^3$ containing $C_1$. 
%By the generality of the points $q_1, \dots , q_5$ in $\mathbb{P}^5$ and of the quadric surfaces $Q_6, Q_7 \in Lambda_{C_1}$, 
We may assume that $Q_7$ is sufficiently general in order to have that $Q_7$ intersects $Q_6$ along the union of $C_1$ and a line $r_1$ not passing through $q_1, \dots , q_5$.  
Let us define $C_3:=\varphi_{C_1}(Q_7)\in \mathcal{C}_{q_1, \dots , q_5}$.
%which is uniquely determined by $Q_7$. 
Then $\dim_{\mathbb{P}^4} \Lambda_{C_2}\cap \Lambda_{C_3} \ge 0$ and, by Bezout's theorem, if $Q_8\in \Lambda_{C_2}\cap \Lambda_{C_3}$ then $Q_8$ is \textit{unique}. In particular $Q_8$ is uniquely determined by $C_3$ which is uniquely determined by $Q_7$.
%, i.e. $h^0(\mathcal{I}_{C_2\cup C_3|\mathbb{P}^3}(2))=1$. 
Let $r_2$, $r_3$ be the lines such that $Q_6 \cap Q_8 = C_2\cup r_2$ and $Q_7 \cap Q_8 = C_3\cup r_3$.
Since $\{q_1, \dots , q_5\}\cap r_1=\emptyset$ by construction, 
%we can say that the twisted cubics, the quadric surfaces and the line above constructed do not satisfy the case (ii) of Theorem~\ref{thm:C1C2C3}.\\
%We want to show that by the generality of $Q_6, Q_7 \in \Lambda_{C_1}$, it follows that $Q_8$ is smooth. 
%First we observe that, since we have fixed $C_1$ and $Q_6$ (and so $C_2$), the set $\{(Q_7,Q_8)\}\subset \Lambda_{C_1}\times \Lambda_{C_3} \subset \mathbb{P}^4\times \mathbb{P}^4$ has dimension $2$: indeed $Q_8$ is uniquely determined by $C_3$ which is uniquely determined by $Q_7$ which is general in $\Lambda_{C_1}$. 
%Furthermore the $1:1$ correspondences $Q_7 \leftrightarrow C_3 \leftrightarrow Q_8$ imply that $Q_8$ has to vary on a $2$-dimensional locus in $|\mathcal{I}_{\{q_1, \dots , q_5\}}|\cong \mathbb{P}^4$. In particular if $\pi_2 : \mathbb{P}^4 \times \mathbb{P}^4 \to \mathbb{P}^4$ is the projection map to the second factor, then the general fibre of $\pi_2 |_{\{Q_7,Q_8\}}$ is $0$-dimensional.
%If $Q_8$ were singular, then $Q_8$ would be a cone and its vertex should be one of the points $q_1, \dots , q_5$ (see Theorem~\ref{thm:C1C2C3} (iii)). Let us suppose that $q_1$ is the vertex of $Q_8$. 
%Hence $Q_8$ should vary on the line $l_1$, i.e. $\pi_2\left(\{(Q_7,Q_8\}\right)\subseteq l_1$ and the general fibre of $\pi_2|_{\{(Q_7,Q_8)\}}$ should be one-dimensional, which is is a contradiction.\\
%Thus 
we may suppose to fix three twisted cubics $C_1$, $C_2$, $C_3$, three lines $r_1$, $r_2$, $r_3$ and three smooth quadric surfaces $Q_6$, $Q_7$, $Q_8$ in $\mathbb{P}^3$ satisfying the property (i) of Theorem~\ref{thm:C1C2C3}.
By the generality of $Q_6, Q_7\in \Lambda_{C_1}\cong \mathbb{P}^2$, we may also assume that $r_i$ is a \textit{chord} of $C_i$ for $i=1,2$, i.e. $a_1'\ne a_1''$ and $a_2'\ne a_2''$. Let us explain this.
We recall that $C_1$ and $C_2$ are the only twisted cubics in $Q_6$ through $q_1, \dots , q_5$. In particular, if $f_1$ and $f_2$ represent the two ruling of $Q_6$, then we have that $C_1\sim_{Q_6} f_1+2f_2$ and $C_2\sim_{Q_6} 2f_1+f_2$. For any choice of $R_1\in |\mathcal{O}_{Q_6}(f_1)|$ and $R_2\in |\mathcal{O}_{Q_6}(f_2)|$ we have that $h^0(\mathcal{I}_{R_1\cup C_1 |\mathbb{P}^3}(2))=h^0(\mathcal{I}_{R_2\cup C_2 | \mathbb{P}^3}(2))=2$. 
% quindi i sistemi lineari di quadriche contenenti R_i\cup C_i sno rette
Furthermore, by the Hurwitz formula applied to $|\mathcal{O}_{C_i}(f_i)|$ for $i=1,2$, there exist only two lines $R_{i,1}, R_{i,2}\in |\mathcal{O}_{Q_6}(f_i)|$ which are tangent to $C_i$.
% in pratica prendo il sistema lineare che le R_i tagliano su C_i
% ricordiamo che R_i\cdot C_i= 2 e che le R_i variano in una retta
% quindi |O_C(R)| è un sistema lineare di dimensione 1 e di grado 2
% (è una cosiddetta g^1_2 su C:= C_i)
% quindi definisce una mappa 2:1 dalla cubica gobba C_i \to PP^1
% ora la formula di Hurwitz ci dice che
% 2g(C)-2=(C\cdot R)(2g(PP1)-2)+deg(ramificazione)
% quindi -2 = 2(-2)+deg(ramificazione) => deg(ram)=4-2=2
% quindi la immagino così: PP1 è come se paramatrizzasse la coppia di punti
% a_i', a_i'' di R_i\cap C_i al variare di R_i
% e sono tutte coppie di punti distinti, tranne due
% quindi ci sono due rette R_i,1 e R_i,2 che toccano C_i in due punti coincidenti
% ovvero sono tangenti in tale punto
Let us consider the four lines
$L_{1,1} = |\mathcal{I}_{R_{1,1}\cup C_1}(2)|$, $L_{1,2} = |\mathcal{I}_{R_{1,2}\cup C_1}(2)|$, $L_{2,1} = |\mathcal{I}_{R_{2,1}\cup C_2}(2)|$, $L_{2,2} = |\mathcal{I}_{R_{2,2}\cup C_2}(2)|.$
By the generality of $Q_7\in\Lambda_{C_1}\cong \mathbb{P}^2$ and by using the fact the $Q_8$ is uniquely determined by $Q_7$, we may assume that $Q_7\not\in L_{1,1} \cup L_{1,2}$ and $Q_8\not \in L_{2,1} \cup L_{2,2}$. So $Q_7\cap Q_6=C_1\cup r_1$ with $r_1$ transverse to $C_1$ and $Q_8\cap Q_6 = C_2 \cup r_2$ with $r_2$ transverse to $C_2$. 
%{\color{blue}
Instead we cannot exclude the case where $r_3$ is \textit{tangent} to $C_3$, i.e. $r_3\cap C_3=\{a_3',a_3''\}$ with $a_3'=a_3''$. 
%It would be interesting to study the above case,
%%} 
%but we will analyze the case in which $r_3$ is a \textit{chord} of $C_3$, since it is the situation mentioned by Fano in order to describe his Enriques-Fano threefold of genus $6$ (see \cite[\S 3]{Fa38}).
Despite the interest that studying the afore-mentioned case would catch, we will devote our analysis to the case in which $r_3$ is a \textit{chord} of $C_3$, considering the latter was the case Fano used in order to describe his Enriques-Fano threefold of genus $6$ (see \cite[\S 3]{Fa38}).

So let us fix now three twisted cubics $C_1$, $C_2$, $C_3$, three chords $r_1$, $r_2$, $r_3$ and three smooth quadric surfaces $Q_6$, $Q_7$, $Q_8$ satisfying (i) of Theorem~\ref{thm:C1C2C3}. Let $\mathcal{P}$ be the linear system of the septic surfaces of $\mathbb{P}^{3}$ double along the three twisted cubics $C_1$, $C_2$ and $C_3$ passing through $q_1$, $q_2$, $q_3$, $q_4$, $q_5$.

\begin{remark}\label{rem:unicityQ6Q7Q8}
The surface $Q_{i+j+3}$ is the unique quadric surface of $\mathbb{P}^3$ containing $C_i\cup C_j\cup r_i\cup r_j $ for all $1\le i<j \le 3$. 
Indeed, we have $h^0(\mathcal{I}_{C_i\cup C_j\cup r_i\cup r_j | \mathbb{P}^3}(2))=1$, by the smoothness of the three quadric surfaces and by following exact sequence
$$0 \to \mathcal{I}_{Q_{i+j+3}|\mathbb{P}^3}(2) \to \mathcal{I}_{C_i\cup C_j\cup r_i\cup r_j | \mathbb{P}^3}(2) \to \mathcal{I}_{C_i\cup C_j\cup r_i\cup r_j | Q_{i+j+3}} (2) \to 0.$$
% in pratica se con H chiamo la sezione iperpiana di Q, ho che il terzo termine diventa
% O_Q(2H-4H) .. dato che C_i+C_j+r_i+r_j = 4f1+4f2 = 4H.
% ma allora l'h' di O_Q(-2H) è 0 e i primi due h^0 coincidono...
% e dato che l'unica quadrica di P3 contenentent Q è Q stessa allora ho la tesi
\end{remark}

\begin{remark}\label{rem:ri contained in P}
%Let us take three twisted cubics $C_1$, $C_2$, $C_3$, three lines $r_1$, $r_2$, $r_3$ and three quadric surfaces $Q_6$, $Q_7$, $Q_8$ satisfying (i) of Theorem~\ref{thm:C1C2C3}. Let $\mathcal{P}$ be the linear system of the septic surfaces of $\mathbb{P}^{3}$ double along the three twisted cubics $C_1$, $C_2$ and $C_3$. Then 
An element $P\in \mathcal{P}$ contains the lines $r_1$, $r_2$, $r_3$. Assume the contrary. Let us fix three distinct indices $i,j,k\in\{1,2,3\}$. By Bezout's Theorem, $P\cap r_{i}$ is given by $7$ points. Furthermore, $r_{i}$ is a line through four double points of $P$, i.e. $r_{i}\cap C_j$, $r_i\cap C_k$, $a_i'$ and $a_i''$. So we obtain that $P\cap r_{i}$ contains at least $8$ points, counted with multiplicity, which is a contradiction. It must be $r_{i} \subset P$. 
\end{remark}

%The linear system defined in Remark~\ref{rem:ri contained in P} seems to be the one found by Fano in order to describe his Enriques-Fano threefold of genus $6$ (see \cite[\S 3]{Fa38}). Indeed, if we took the linear system of the septic surfaces of $\mathbb{P}^3$ double along three twisted cubics $C_1$, $C_2$, $C_3$ satisfying (ii) or (iii) of Theorem~\ref{thm:C1C2C3}, then we couldn't deduce that the three correspondent lines $r_1$, $r_2$, $r_3$ are contained in its base locus, as Fano said and as we proved for the case (i).\\
%
%For this reason, we fix now three twisted cubics $C_1$, $C_2$, $C_3$, three lines $r_1$, $r_2$, $r_3$ and three quadric surfaces $Q_6$, $Q_7$, $Q_8$ satisfying (i) of Theorem~\ref{thm:C1C2C3}, and we will study the linear system $\mathcal{P}$ defined in Remark~\ref{rem:ri contained in P}.
%% of the septic surfaces of $\mathbb{P}^{3}$ double along the three fixed twisted cubics $C_1$, $C_2$ and $C_3$ passing through $q_1, q_2, q_3, q_4, q_5$.

Let $g_{i+j+3}:=g_{i+j+3}(s_0,s_1,s_2,s_3)$ be the quadratic homogeneous polynomial defining the smooth quadric surface $Q_{i+j+3}\subset \mathbb{P}^3_{\left[s_0,\dots ,s_3\right]}$ for $1\le i<j \le 3$.

\begin{lemma}\label{lem:equationP}
The linear system $\mathcal{P}$ has equation $g_6g_7f_{8}+g_6g_8f_{7}+g_7g_8f_{6}=0$, where $f_{i+j+3}\in H^0(\mathbb{P}^3, \mathcal{I}_{C_i\cup C_j | \mathbb{P}^3}(3))$, for $1\le i<j\le 3$.
\end{lemma} 
\begin{proof}
Let $F:=F(s_0,s_1,s_2,s_3)$ be the homogeneous polynomial of degree $7$ defining a general element $P$ of $\mathcal{P}$ in $\mathbb{P}^3_{\left[s_0,\dots ,s_3\right]}$.
We recall that the intersection of an irreducible septic surface of $\mathbb{P}^3$ with a quadric surface is a curve of degree $14$. In particular, $P$ intersects each quadric surface $Q_{i+j+3}$ along the curve of degree $14$ given by the two double twisted cubics $C_i$ and $C_j$ plus the two lines $r_{i}$ and $r_j$, for $1\le i<j \le 3$.  
This implies that it must be 
$$P\cap Q_{i+j+3}=\{g_{i+k+3}g_{j+k+3}f_{i+j+3}=0,\, g_{i+j+3} = 0\}=2C_{i}+2C_{j}+r_{i}+r_j$$ 
for some $f_{i+j+3}\in H^0(\mathbb{P}^3, \mathcal{I}_{C_i\cup C_j | \mathbb{P}^3}(3))$, where $1\le k \le 3$ with $k\ne i$ and $k\ne j$. 
%So we have $F = g_6g_7f_{8}+g_6g_8f_{7}+g_7g_8f_{6}$.
Then it must be
$F = g_6 h_5+ g_7g_8f_6,$
where $h_5$ is a homogeneous polynomial of degree $5$ such that
$h_5 = g_7h_3+g_8f_7,$
where $h_3$ is a homogeneous polynomial of degree $3$ such that
$h_3 = g_8h_1+f_8,$
where $h_1$ is a homogeneous polynomial of degree $1$.
Thus, we have $F = g_6g_7g_8h_1 + g_6g_7f_{8}+g_6g_8f_{7}+g_7g_8f_{6}.$
Since $g_{i+j+3}h_1\in H^0(\mathbb{P}^3, \mathcal{I}_{C_i\cup C_j | \mathbb{P}^3}(3))$ for $1\le i<j \le 3$, we obtain that $F$ has the expression of the statement.
\end{proof}

\begin{lemma}\label{lem:cubiccontainingCiCj}
Let us take $1\le i<j\le 3$. Then $\dim H^0(\mathbb{P}^3, \mathcal{I}_{C_i\cup C_j | \mathbb{P}^3}(3))=5$ and a general element of $|\mathcal{I}_{C_i\cup C_j | \mathbb{P}^3}(3)|\cong \mathbb{P}^4$ corresponds to a smooth irreducible surface.
\end{lemma}
\begin{proof}
Let us consider the following exact sequence
$$0 \to \mathcal{I}_{Q_{i+j+3} | \mathbb{P}^3}(3) \to \mathcal{I}_{C_i\cup C_j | \mathbb{P}^3}(3) \to \mathcal{I}_{C_i\cup C_j | Q_{i+j+3}}(3) \to 0.$$
Since $h^1(\mathcal{I}_{Q_{i+j+3| \mathbb{P}^3}}(3))=h^1(\mathcal{O}_{\mathbb{P}^3}(1))=0$, $h^0(\mathcal{I}_{Q_{i+j+3} | \mathbb{P}^3}(3))=h^0(\mathcal{O}_{\mathbb{P}^3}(1))=4$ and $h^0( \mathcal{I}_{C_i\cup C_j | Q_{i+j+3}}(3)) =h^0(\mathcal{O}_{Q_{i+j+3}})=1$, then we obtain
$h^0(\mathcal{I}_{C_i\cup C_j | \mathbb{P}^3}(3)) = 5$.
Let $S_3$ be now a general element of $|\mathcal{I}_{C_i\cup C_j | \mathbb{P}^3}(3)|$. We may assume that $Q_{i+j+3}\not\subset S_3$ and so that $S_3$ is irreducible, since $C_i\cup C_j$ is not degenerate and the only quadric surface containing this curve is $Q_{i+j+3}$ (see Remark~\ref{rem:unicityQ6Q7Q8}). We want to show that $S_3$ is smooth.
First let us see that $C_i\cup C_j$ is the 
%only base locus
base scheme
of $|\mathcal{I}_{C_i\cup C_j | \mathbb{P}^3}(3)|$: we have to show that 
$h^0(\mathcal{I}_{C_i\cup C_j\cup \{x\} | \mathbb{P}^3}(3)) = h^0(\mathcal{I}_{C_i\cup C_j | \mathbb{P}^3}(3))-1=4$
%if we force the cubic surfaces of $|\mathcal{I}_{C_i\cup C_j | \mathbb{P}^3}(3)|$ to pass through a point  
for a point $x\in \mathbb{P}^3\setminus (C_i\cup C_j)$.
This is exactly what happens: indeed, $x \not \in Q_{i+j+3}$ (otherwise $x\in S_3'\cap Q_{i+j+3} = C_i \cup C_j$ for $S_3'\in|\mathcal{I}_{C_i\cup C_j\cup \{x\} | \mathbb{P}^3}(3)|$, which is a contradiction) and so we have the following exact sequence
$$0 \to \mathcal{I}_{Q_{i+j+3}\cup \{x\} | \mathbb{P}^3}(3) \to \mathcal{I}_{C_i\cup C_j\cup \{x\} | \mathbb{P}^3}(3) \to \mathcal{I}_{C_i\cup C_j | Q_{i+j+3}}(3) \to 0$$
from which $h^0(\mathcal{I}_{C_i\cup C_j\cup \{x\} | \mathbb{P}^3}(3)) = 4$, since 
$h^1(\mathcal{I}_{Q_{i+j+3}\cup \{x\}| \mathbb{P}^3}(3))=h^1(\mathcal{I}_{\{x\}|\mathbb{P}^3}(1))=0$ and $h^0(\mathcal{I}_{Q_{i+j+3}\cup \{x\} | \mathbb{P}^3}(3))=h^0(\mathcal{I}_{\{x\}| \mathbb{P}^3}(1))=3$.
% quindi quelle riducibili sono solo del tipo Q_{i+j+3}U H con H piano
% quindi formano un P^3 nel P^4 in cui scelgo S_3
% dato che scelgo S_3 generica io la prendo fuori da questo iperpiano
% e quindi posso concludere che è irriducibile
Let $p\in (C_i\cup C_j) \setminus \{q_1, \dots , q_5\}$. If $S_3$ were singular at $p$, then $S_3 \cap Q_{i+j+3}=C_i\cup C_j$ would be singular at $p$, which is a contradiction.
Let $p\in \{q_1, \dots , q_5\}$. If $H$ is a plane such that $q_h\not\in H$ for all $1\le h \le 5$, then $Q_{i+j+3}\cup H \in |\mathcal{I}_{C_i\cup C_j | \mathbb{P}^3}(3)|$. Since $Q_{i+j+3}\cup H$ is smooth at $p$, then the general element of $|\mathcal{I}_{C_i\cup C_j | \mathbb{P}^3}(3)|$ is a cubic surface smooth at $p$. Thus, $S_3$ is smooth.
\end{proof}

\begin{remark}\label{rem:Q6Q7Q8independent}
There exists a septic surface in $\mathcal{P}$ containing $Q_6$ but not $Q_7$. By Lemmas~\ref{lem:equationP},~\ref{lem:cubiccontainingCiCj} it is sufficient to take a septic surface defined by the equation
$g_6g_7f_{8}+g_6g_8f_{7}+g_7g_8g_6h=0$
with $f_{8}\in H^0(\mathbb{P}^3, \mathcal{I}_{C_2\cup C_3 | \mathbb{P}^3}(3))$, $h\in H^0(\mathbb{P}^3, \mathcal{O}_{\mathbb{P}^3}(1))$ and where $f_{7}$ is a general (irreducible) element of $H^0(\mathbb{P}^3, \mathcal{I}_{C_1\cup C_3 | \mathbb{P}^3}(3))$. 
One can also construct a septic surface in $\mathcal{P}$ containing $Q_6$ and $Q_7$ but not $Q_8$. By Lemmas~\ref{lem:equationP},~\ref{lem:cubiccontainingCiCj} it is sufficient to take a septic surface with equation
$g_6g_7f_{8}+g_6g_8g_7h'+g_7g_8g_6h=0$ where $h,h'\in H^0(\mathbb{P}^3, \mathcal{O}_{\mathbb{P}^3}(1))$ and where $f_{8}$ is a general (irreducible) element of $H^0(\mathbb{P}^3, \mathcal{I}_{C_2\cup C_3 | \mathbb{P}^3}(3))$.
\end{remark}

A priori we have that $\dim \mathcal{P}\le 14$, since the equation of $\mathcal{P}$ depends by $15$ parameters which can be linearly dependent (see Lemmas~\ref{lem:equationP},~\ref{lem:cubiccontainingCiCj}). However we have the following result.

\begin{proposition}\label{prop:dimP=6}
The linear system $\mathcal{P}$ defined as above has $\dim \mathcal{P} = 6$.
\end{proposition} 
\begin{proof}
Let us consider the sublinear system of the septic surfaces of $\mathcal{P}$ containing $Q_{i+j+3}$ for $1\le i<j\le 3$. The movable part of this linear system is isomorphic to the linear system $\mathcal{T}$ of the quintic surfaces of $\mathbb{P}^3$ containing the two twisted cubics $C_i,C_j\subset Q_{i+j+3}$, containing the line $r_k$, and with double points along the twisted cubic $C_k$, where $1\le k\le 3$ and $k\ne i$ and $k\ne j$. We want to show that $\operatorname{codim}\left( \{P\in \mathcal{P}| P\supset Q_{i+j+3}\},\mathcal{P}\right) =1$. In order to do it, let $V\subset H^0(\mathcal{O}_{\mathbb{P}^{3}}(7))$ and $K\subset H^0(\mathcal{O}_{\mathbb{P}^{3}}(5))$ be the subspaces such that $\mathcal{P}=\mathbb{P}(V)$ and $\mathcal{T}=\mathbb{P}(K)$. From 
$$0\to \mathcal{O}_{\mathbb{P}^3}(5) \to \mathcal{O}_{\mathbb{P}^3}(7) \to \mathcal{O}_{Q_{i+j+3}}(7) \to 0,$$
we obtain
$$0\to H^0(\mathcal{O}_{\mathbb{P}^3}(5)) \to H^0(\mathcal{O}_{\mathbb{P}^3}(7)) \to H^0(\mathcal{O}_{Q_{i+j+3}}(7)) \to 0$$
$$\quad \quad \cup \quad \quad  \quad \quad \cup \quad \quad \quad \quad \quad \cup \quad \quad \quad \quad$$
$$0\to \quad K \quad \to \quad  V \quad \to \quad  V|_{Q_{i+j+3}} \quad \to 0.$$
%Thus 
We have to show that $\operatorname{codim}(K,V)=1$, which is equivalent to find $\dim V|_{Q_{i+j+3}}=1$. This follows by the fact that $\dim \mathcal{P}|_{Q_{i+j+3}}=0$, since $\mathcal{P}|_{Q_{i+j+3}}$ only has fixed part $2C_i+2C_j+r_i+r_j$. 
Then we have that $\operatorname{codim}\left( \{P\in \mathcal{P}| P\supset Q_{i+j+3}\},\mathcal{P}\right) =1$ and, since containing the three quadrics $Q_6$, $Q_7$ and $Q_8$ imposes independent conditions (see Remark~\ref{rem:Q6Q7Q8independent}), we also obtain $\operatorname{codim}(\{P\in \mathcal{P}| P\supset Q_6, Q_7, Q_8\},\mathcal{P})=3$.
%Furthermore 
Each element of $\{P\in \mathcal{P}| P\supset Q_6,Q_7,Q_8\}$ is of the form $Q_6\cup Q_7\cup Q_7\cup \pi$, where $\pi$ is a general plane of $\mathbb{P}^3$. Thus, we have $\dim \{P\in \mathcal{P}| P\supset Q_6,Q_7,Q_8\}= \dim |\mathcal{O}_{\mathbb{P}^3}(1)|=3$ and finally $\dim \mathcal{P}=3+3=6$.
\end{proof}

\begin{remark}\label{rem:TCcubic}
Let us fix $1\le i<j \le 3$ and let us consider the quadric surface $Q_{i+j+3}\subset \mathbb{P}^3$. Since $Q_{i+j+3}$ is smooth, then the tangent space to $Q_{i+j+3}$ at the point $p\in Q_{i+j+3}$ is a plane of $\mathbb{P}^3$, which is spanned by the two lines of $Q_{i+j+3}$ intersecting at $p$. 
%Since this plane is uniquely determined, we denote it by $f_{hij}$. 
Let us take the point $p=q_h$ for some $1\le h \le 5$. Since the twisted cubics $C_i$ and $C_j$ are contained in $Q_{i+j+3}$ and they pass through $q_h$, then the tangent plane to $Q_{i+j+3}$ at $q_h$ has to contain the tangent lines to $C_i$ and $C_j$ at $q_h$. We recall that $C_i \cdot C_j = (f_1+2f_2)\cdot (2f_1+f_2) = 5$, where $f_1$ and $f_2$ represent the two rulings of $Q_{i+j+3}$. Since $q_1, \dots ,q_5$ are distinct by construction, the intersection of $C_i$ and $C_j$ at each $q_h$ is transverse. Then we have $T_{q_h}C_{i}\ne T_{q_h}C_j$ and $t_{hij}:=T_{q_h}Q_{i+j+3}=\left\langle T_{q_h}C_{i}, T_{q_h}C_j \right\rangle$. 
In particular we have that $T_{q_h}Q_{6}=\left\langle T_{q_h}C_{1}, T_{q_h}C_2 \right\rangle$ and $T_{q_h}Q_{7}=\left\langle T_{q_h}C_{1}, T_{q_h}C_3 \right\rangle$. By the generality of $q_1, \dots , q_5$ and by the generality of $Q_6$ and $Q_7$, we may assume $T_{q_h}Q_{6}\cap T_{q_h}Q_{7} = T_{q_h}C_{1}$. Thus, $T_{q_h}C_{1}$, $T_{q_h}C_{2}$ and $T_{q_h}C_{3}$ are linearly independent.
\end{remark}

\begin{proposition}\label{prop:variabeleTC6}
Let $P$ be a general element of $\mathcal{P}$ and let us take $1\le h \le 5$ and three distinct indices $i,j,k\in\{1,2,3\}$ with $i<j$. Then we have that
\begin{itemize}
\item[(i)] $TC_{q_h}P=\bigcup_{1\le i<j \le 3}T_{q_h}Q_{i+j+3}=\bigcup_{1\le i<j \le 3}t_{hij}$ and $P$ has triple points at the five points $q_1, \dots , q_5$;
\item[(ii)] if $p\in C_{k}$ with $p\not\in\{q_1, \dots , q_5\}$, $p\not\in r_k\cap C_k$ and $p\ne b_{ij}$, then $TC_{p}P$ is the union of two variable planes $\pi_{p,P}$ and $\pi_{p,P}'$ containing $T_pC_k$ and depending on the choice of the point $p$ and of the surface $P$;
\item[(iii)] if $p\in r_k\cap C_{k}$, then $TC_{p}P=\pi_{p}\cup \pi_{p,P}$, where the plane $\pi_{p}:=T_pQ_{i+k+3}=T_pQ_{j+k+3}$ contains $T_pC_k$ and $r_k$, and where $\pi_{p, P}$ is a plane containing $T_pC_k$ and depending on the choice of $p\in\{a_k',a_k''\}$ and of $P$; 
\item[(iv)] $TC_{b_{ij}}P = \pi_{ij,i} \cup \pi_{ij,j}$, where the plane $\pi_{ij,i}:=T_pQ_{i+k+3}$ contains $r_{i}$ and $T_{b_{ij}}C_k$, and where the plane $\pi_{ij, j}:=T_pQ_{j+k+3}$ contains $r_{j}$ and $T_{b_{ij}}C_k$;
\item[(v)] if $p\in r_k$ with $p\not\in r_k\cap C_k$ and $p\ne r_k\cap r_i$, then $TC_{p}P$ is a variable plane depending on the choice of $p$ and $P$.
\end{itemize}
\end{proposition}
\begin{proof}
We may assume that $P$ has equation
$g_6g_7f_{8}+g_6g_8f_{7}+g_7g_8f_{6}=0$ for smooth irreducible $f_{i+j+3}\in H^0(\mathbb{P}^3, \mathcal{I}_{C_i\cup C_j | \mathbb{P}^3}(3))$ for $1\le i<j\le 3$ (see Lemmas~\ref{lem:equationP},~\ref{lem:cubiccontainingCiCj}). 
Let $p$ be a point of $P$ and let us consider an open affine set $U_p\cong \mathbb{A}^3 \subset \mathbb{P}^3$ containing $p$. By abuse of notation, let us denote by $F:=g_6g_7f_{8}+g_6g_8f_{7}+g_7g_8f_{6}$ the polynomial of degree $7$ defining $P\cap U_p$. In order to compute the tangent cone to $P$ at the point $p$, we have to take the minimal degree homogeneous part of the Taylor series of $F$ at $p$. In the following, if $h$ is a polynomial, then $h_d(p)$ will denote the homogeneous part of degree $d$ of the Taylor series of $h$ at $p$. By using this notation, if $p$ is a point of the quadric $Q_{i+j+3}$, we have that $T_pQ_{i+j+3} = \{g_{i+j+3,1}(p)=0\}$, for $1\le i<j\le 3$. Let us study $TC_p P$ case by case.
\begin{itemize}
\item[(i)] Let us take $p\in\{q_1, \dots , q_5\}$. Then $TC_{p}P$ has equation 
$g_{6,1}(p)g_{7,1}(p)f_{8,1}(p)+g_{6,1}(p)g_{8,1}(p)f_{7,1}(p)+g_{7,1}(p)g_{8,1}(p)f_{6,1}(p)=0$, where $\{f_{8,1}(p)=0\}=\left\langle T_pC_2, T_pC_3 \right\rangle$, $\{f_{7,1}(p)=0\}=\left\langle T_pC_1, T_pC_3 \right\rangle$, $\{f_{6,1}(p)=0\}=\left\langle T_pC_1, T_pC_2 \right\rangle$ (see Remark~\ref{rem:TCcubic}). So we obtain $TC_pP = \bigcup_{1\le i<j\le 3} T_p Q_{i+j+3}$.
\item[(ii)] Let us take $p\in C_k$ such that $p\not\in \{q_1, \dots , q_5\}$, $p\not\in r_k\cap C_k$ and $p\ne b_{ij}$ for distinct indices $i,j,k\in\{1,2,3\}$ with $i<j$. Let us suppose $k=1$. Then $TC_{p}P$ has equation $c_1g_{6,1}(p)g_{7,1}(p)+c_2g_{6,1}(p)f_{7,1}(p)+c_3g_{7,1}(p)f_{6,1}(p)=0$, where $c_1$, $c_2$, $c_3$ are constants (depending on the choice of $p$ and $P$). Since $\{g_{6,1}(p)g_{7,1}(p)=0\}$, $\{g_{6,1}(p)f_{7,1}(p)=0\}$ and $\{g_{7,1}(p)f_{6,1}(p)=0\}$ are three reducible quadric surfaces given by two planes containing the line $T_pC_{1}$, then $TC_pP$ is singular along $T_pC_1$ and so it is the union of two planes containing $T_pC_{1}$. Similarly for $k=2,3$.
\item[(iii)] Let us take $p\in r_k\cap C_k$ for $1\le k\le 3$. Let us suppose $k=1$. Then $TC_{p}P$ has equation $c_1g_{6,1}(p)g_{7,1}(p)+c_2g_{6,1}(p)f_{7,1}(p)+c_3g_{7,1}(p)f_{6,1}(p)=0$, where $c_1$, $c_2$, $c_3$ are constants (depending on the choice of $p$ and $P$) and where $\{g_{6,1}(p)=0\}=T_pQ_6\supset T_pC_1\cup r_1$, $\{g_{7,1}(p)=0\}=T_pQ_7\supset T_pC_1\cup r_1$, $\{f_{68,1}(p)=0\}\supset T_pC_1$ and $\{f_{78,1}(p)=0\}\supset T_pC_1$. In this case we also have $T_pQ_6 = T_pQ_7$, otherwise it would be 
$$1\le \dim T_p(Q_6\cap Q_7) \le \dim (T_pQ_6 \cap T_pQ_7) = 1$$ 
and $Q_6\cap Q_7 = r_1\cup C_1$ would be smooth at $p\in r_1\cap C_1$, which is a contradiction. Thus, $T_pP$ is the union of the plane $T_pQ_6(=TpQ_7)$, which contains $T_pC_1$ and $r_1$, and a plane containing $T_pC_1$. Similarly for $k=2,3$.
\item[(iv)] Let us take $p\in r_i\cap r_j\cap C_k$ for three distinct indices $i,j,k\in\{1,2,3\}$ and $i<j$. Let us suppose $i=1$, $j=2$, $k=3$. Then $TC_{p}P$ has equation $g_{7,1}(p)g_{8,1}(p)=0$, where $\{g_{7,1}(p)=0\}=T_pQ_7\supset T_pC_3\cup r_1$ and 

$\{g_{8,1}(p)=0\}\supset T_pC_3\cup r_2$. Thus, $TC_pP$ is the union of $T_pQ_7\cup T_pQ_8$. Similarly by taking $(i,j,k)\in \{(1,3,2),(2,3,1)\}$.
\item[(v)] Let us take $p\in r_k$ with $p\not\in r_k\cap C_k$ and $p\ne r_k\cap r_i$ for $1\le i,k\le 3$ and $i\ne k$. Let us suppose $k=1$. Then $TC_{p}P$ has equation $c_1g_{6,1}(p)+c_2g_{7,1}(p)=0$ where $c_1$ and $c_2$ are constants depending on the choice of $p$ and $P$. Similarly for $k=2,3$. 
\end{itemize}
\end{proof}

\begin{lemma}\label{lem:birationalitanu6}
The rational map $\nu_{\mathcal{P}} : \mathbb{P}^{3} \dashrightarrow \mathbb{P}^{6}$ defined by $\mathcal{P}$ is birational onto the image.
\end{lemma}
\begin{proof}
It is sufficient to prove that the map defined by $\mathcal{P}$ on a general $P \in \mathcal{P}$ is birational onto the image. 
This actually happens because $\mathcal{P}|_{P}$ contains a sublinear system that defines a birational map. Indeed, $\mathcal{P}$ contains a sublinear system $\overline{\mathcal{P}} \subset \mathcal{P}$ 
%which has a movable part, given by the planes of $\mathbb{P}^3$, and 
whose fixed part is given by $Q_6\cup Q_7 \cup Q_8$ and such that $\overline{\mathcal{P}}|_{P}$ coincides with the linear system on $P$ cut out by the planes of $\mathbb{P}^3$.
\end{proof} 

\begin{remark}\label{rem:veryample6}
The proof of Lemma~\ref{lem:birationalitanu6} tells us that the linear system $\mathcal{P}$ is very ample outside the three quadric surfaces $Q_6$, $Q_7$, $Q_8$. So $\nu_{\mathcal{P}} : \mathbb{P}^{3} \dashrightarrow  \nu_{\mathcal{P}}(\mathbb{P}^3)\subset \mathbb{P}^{6}$ is an isomorphism outside $Q_6\cup Q_7\cup Q_8$.
\end{remark}

\begin{theorem}\cite[\S 3]{Fa38}\label{thm:WF6 isEF}
The image of $\mathbb{P}^{3}$ via the rational map defined by $\mathcal{P}$ is an Enriques-Fano threefold $W_{F}^{6}$ of genus $p=6$.
\end{theorem}

%\subsubsection{Proof of Theorem~\ref{thm:WF6 isEF}}\label{subsec:dimW6}

\begin{proof}

We will prove the theorem by using the same techniques of the proof of Theorems~\ref{thm:WF13 isEF},~\ref{thm:WF9 isEF}. In particular the proof is divided into several steps, given by the Remarks~\ref{rem:betap6},~\ref{rem:selfalphap6}, the Proposition~\ref{prop:tildePSmoothePa0}, the Remarks~\ref{rem:-1curvesP},$\dots$,~\ref{rem:uniqueQQQp6} and the Theorem~\ref{thm:S6isEnriques} below.

First we blow-up $\mathbb{P}^3$ at the five points $q_1$, $q_2$, $q_3$, $q_4$, $q_5$, at the six points $a_1'$, $a_2'$, $a_3'$, $a_1''$, $a_2''$, $a_3''$ and at the three points $b_{12}$, $b_{13}$, $b_{23}$. We obtain a smooth threefold $Y'$ and a birational morphism $bl' : Y '\to \mathbb{P}^{3}$ with exceptional divisors $E_h := (bl')^{-1}(q_h)$, $E_{ij} = (bl')^{-1}(b_{ij})$, $E_i' := (bl')^{-1}(a_i')$, $E_i'' := (bl')^{-1}(a_i'')$, 
for $1\le h \le 5$ and $1\le i<j\le 3$.
Let $\mathcal{P}'$ be the strict transform of $\mathcal{P}$ and let us denote by $H$ the pullback on $Y'$ of the hyperplane class on $\mathbb{P}^{3}$. Then an 
%general 
element of $\mathcal{P}'$ is linearly equivalent to $7H-3\sum_{h=0}^{5}E_i-2\sum_{i=1}^3(E_{i}'+E_i'')-2\sum_{1\le i<j\le 3}E_{ij}$.
Let $\widetilde{t}_{hij}$, $\widetilde{\pi}_{a_i'}$, $\widetilde{\pi}_{a_i''}$, $\widetilde{\pi}_{ij,i}$, $\widetilde{\pi}_{ij,j}$ be the strict transforms of the planes defined in Remark~\ref{rem:TCcubic} and Proposition~\ref{prop:variabeleTC6}, for $1\le h \le 5$ and $1\le i<j\le 3$. Let us consider the following $27$ lines on $Y'$:
$$\gamma_{hij}:=E_h\cap \widetilde{t}_{hij},\,\, \lambda_{ij,i}:=E_{ij}\cap \widetilde{\pi}_{ij,i},\,\, \lambda_{ij,j}:=E_{ij}\cap \widetilde{\pi}_{ij,j},\,\, \lambda_{i}'= E_i' \cap \pi_{a_i'},\,\, \lambda_i'' = E_i'' \cap \pi_{a_i''}.$$ 
%$$\gamma_{hij}:= \left\langle E_h\cap\widetilde{C}_i, E_h\cap\widetilde{C}_j  \right\rangle, \,\,
%\lambda_{ij,i} := \left\langle E_{ij}\cap \widetilde{r}_i , E_{ij}\cap \widetilde{C}_k  \right\rangle, \,\, \lambda_{ij,j} := \left\langle E_{ij}\cap \widetilde{r}_j , E_{ij}\cap \widetilde{C}_k  \right\rangle$$
%$$\lambda_i' := \left\langle E_i'\cap \widetilde{C}_i , E_i'\cap \widetilde{r}_i  \right\rangle , \,\,\lambda_i'' := \left\langle E_i''\cap \widetilde{C}_i , E_i''\cap \widetilde{r}_i  \right\rangle ,$$ 
They are respectively $(-1)$-curves on $\widetilde{t}_{hij}$, $\widetilde{\pi}_{ij,i}$, $\widetilde{\pi}_{ij,j}$, $\widetilde{\pi}_{a_i'}$ and $\widetilde{\pi}_{a_i''}$. Let $P'$ be the strict transform of a general $P\in\mathcal{P}$. By Proposition~\ref{prop:variabeleTC6} (i) and (iv) we have that $P'\cap E_{h}=\bigcup_{\le i<j \le 3}\gamma_{hij}$ and $P'\cap E_{ij}=\lambda_{ij,i}\cup \lambda_{ij,j}$, for all $1\le h \le 5$ and $1\le i< j\le 3$.

\begin{remark}\label{rem:betap6}
Let us fix $1\le i \le 3$. We have that $P'\cap E_{i}' = \lambda_{i}' \cup \beta_{i,P}'$ and $P'\cap E_i'' = \lambda_i'' \cup \beta_{i,P}''$, where the curve $\beta_{i,P}'$ moves in the pencil of the lines of $E_{i}'$ through the point $E_{i}'\cap \widetilde{C}_{i}$ and the curve $\beta_{i,P}''$ moves in the pencil of the lines of $E_{i}''$ through the point $E_{i}''\cap \widetilde{C}_{i}$, and both lines depend on the choice of $P$ (see Proposition~\ref{prop:variabeleTC6} (iii)).
\end{remark}

Let us take the strict transforms $\widetilde{C}_{i}$ and $\widetilde{r}_i$ of the three twisted cubics and of their chords, for $1\le i \le 3$. The base locus of $\mathcal{P}'$ is given by the union of the three curves $\widetilde{C}_{i}$ (along which a general $P'\in\mathcal{P}'$ has double points), of the three curves $\widetilde{r}_i$, and of the $27$ curves $\gamma_{hij}$, $\lambda_{ij,i}$, $\lambda_{ij,j}$, $\lambda_i'$, $\lambda_{i}''$ defined above (see Proposition~\ref{prop:variabeleTC6}). 
%(ii).
Let us blow-up $Y'$ along the strict transforms of the three twisted cubics and of their chords. We obtain a smooth threefold $Y''$ and a birational morphism $bl'' : Y'' \to Y'$ with exceptional divisors 
$(bl'')^{-1}(\widetilde{C}_i):=F_i\cong \mathbb{P}(\mathcal{N}_{\widetilde{C}_i | Y'} )\cong \mathbb{P}( \mathcal{O}_{\mathbb{P}^1}(-3)\oplus \mathcal{O}_{\mathbb{P}^1}(-3) )\cong \mathbb{F}_0$ and
$(bl'')^{-1}(\widetilde{r}_i):=R_i \cong \mathbb{P}(\mathcal{N}_{\widetilde{r}_{i}|Y'}) \cong \mathbb{P}(\mathcal{O}_{\mathbb{P}^1}(-3)\oplus \mathcal{O}_{\mathbb{P}^1}(-3))\cong \mathbb{F}_0$
for $1\le i \le 3$, since $\mathcal{N}_{C_i|\mathbb{P}^3}\cong \mathcal{O}_{\mathbb{P}^1}(5)\oplus \mathcal{O}_{\mathbb{P}^1}(5)$ (see \cite[Proposition 6]{EV81}). Let us denote by $\widetilde{E}_h$, $\widetilde{E}_i'$, $\widetilde{E}_i''$ and $\widetilde{E}_{ij}$ respectively the strict transforms of $E_{h}$, $E_{i}'$, $E_{i}''$, $E_{ij}$, for $1\le h \le 5$ and $0\le i<j\le 3$.

\begin{remark}\label{rem:selfalphap6}
Let us consider the curves $\alpha_{hi}:=\widetilde{E}_h\cap F_{i}$, $\alpha_{i}':=\widetilde{E}_{i}'\cap F_{i}$, $\alpha_{i}'':=\widetilde{E}_{i}''\cap F_{i}$, $\alpha_{ijk}:=\widetilde{E}_{ij}\cap F_{k}$, $\rho_{i}':=\widetilde{E}_{i}'\cap R_{i}$, $\rho_{i}'':=\widetilde{E}_{i}''\cap R_{i}$, $\rho_{ij,i}:=\widetilde{E}_{ij}\cap R_{i}$, $\rho_{ij,j}:=\widetilde{E}_{ij}\cap R_{j}$ for $1\le h \le 5$ and distinct indices $i,j,k\in\{1,2,3\}$ with $i<j$. Each of these curves is a fibre on the exceptional divisor of $bl'':Y''\to Y'$ which contains it, and is a $(-1)$-curve on the strict transform of the exceptional divisor of $bl: Y' \to \mathbb{P}^3$ containing it.
\end{remark}

Let $\mathcal{P}''$ be the strict transform of $\mathcal{P}'$. If $P''$ is an 
%general 
element of $\mathcal{P}''$, then 
$$P''\sim 7H-3\sum_{h=1}^5 \widetilde{E}_h-2\sum_{i=1}^3 (\widetilde{E}_{i}'+\widetilde{E}_i'')- 2\sum_{1\le i<j\le 3}\widetilde{E}_{ij}- 2\sum_{i=1}^3 F_{i}-\sum_{i=1}^3R_{i},$$ 
where, by abuse of notation, $H$ denotes the pullback $bl''^* H$.

\begin{proposition}\label{prop:tildePSmoothePa0}
A general element $P'' \in \mathcal{P}''$ is a smooth surface with zero arithmetic genus $p_a(P'')=0$.
\end{proposition}
\begin{proof}
The smoothness of $P''$ is shown in \cite[p.620-621]{GH}, since $P''$ is the blow-up of a surface $P\in \mathcal{P}$ with ordinary singularities along its singular curves (see Definition~\ref{def:ordinarysingularities} and Proposition~\ref{prop:variabeleTC6}). We have to compute $p_a(P'')=\chi ( \mathcal{O}_{Y''}(K_{Y''}+P''))$ (see proof of Proposition~\ref{prop:tildeKSmoothePa0}). Since
$$K_{Y''}\sim -4H+2\sum_{h=1}^5 \widetilde{E}_h+2\sum_{i=1}^3 (\widetilde{E}_{i}'+\widetilde{E}_i'')+ 2\sum_{1\le i<j\le 3}\widetilde{E}_{ij}+ \sum_{i=1}^3 F_{i}+\sum_{i=1}^3R_{i}$$
(see \cite[p.187]{GH}), 
then $K_{Y''}+P'' \sim 3H -\sum_{h=1}^5 \widetilde{E}_h - \sum_{i=1}^3 F_{i}$.
By denoting the fibre class of $F_{i}$ by $f_{i}$ for $i=1,2,3$, we have the following two exact sequences:
$$0 \to \mathcal{O}_{Y''}(3H-\sum_{h=1}^5\widetilde{E}_h) \to \mathcal{O}_{Y''}(3H) \to \oplus_{h=1}^5\mathcal{O}_{\widetilde{E}_h} \to 0,$$
$$0 \to \mathcal{O}_{Y''}(K_{Y''}+P'') \to \mathcal{O}_{Y''}(3H-\sum_{h=1}^5\widetilde{E}_h) \to \oplus_{i=1}^3\mathcal{O}_{F_{i}}(4f_i)\to 0.$$
So we obtain
$\chi (\mathcal{O}_{Y''}(K_{Y''}+P'')) = \binom{6}{3}-5-3\cdot 5=0.$
\end{proof}

By Proposition~\ref{prop:variabeleTC6},
%(ii), 
we have that the base locus of $\mathcal{P}''$ is given by the disjoint union of the strict transforms $\widetilde{\gamma}_{hij}$, $\widetilde{\lambda}_{ij,i}$, $\widetilde{\lambda}_{ij,j}$, $\widetilde{\lambda}_{i}'$ and $\widetilde{\lambda}_{i}''$ of the 27 lines $\gamma_{hij}$, $\lambda_{ij,x}$, $\lambda_i'$, $\lambda_{i}''$ for $1\le h \le 5$ and $1\le i<j\le 3$. 

\begin{remark}\label{rem:-1curvesP}
We observe that 
$\widetilde{\gamma}_{hij}^2|_{\widetilde{E}_{h}}=\widetilde{\lambda}_{ij,i}^2|_{\widetilde{E}_{ij}}=\widetilde{\lambda}_{ij,j}^2|_{\widetilde{E}_{ij}}=\widetilde{\lambda}_{i}'^2|_{\widetilde{E}_{i}'}=\widetilde{\lambda}_{i}''^2|_{\widetilde{E}_{i}''}=-1$.
Furthermore, by using similar arguments to the ones in Remark~\ref{rem:-1curvesSigma} and Remark~\ref{rem:-1curvesX}, we have that the $27$ curves $\widetilde{\gamma}_{hij}$, $\widetilde{\lambda}_{ij,i}$, $\widetilde{\lambda}_{ij,j}$, $\widetilde{\lambda}_{i}'$ and $\widetilde{\lambda}_{i}''$ are $(-1)$-curves on the strict transform $P''$ of a general $P'\in\mathcal{P}'$. Moreover $P''$ contains other $(-1)$-curves that depend on $P''$ itself: they are the strict transforms $\widetilde{\beta}_{i,P}'$ and $\widetilde{\beta}_{i,P}''$ of the curves defined in Remark~\ref{rem:betap6}.
\end{remark}

Finally let us consider $bl''' : Y \to Y''$ the blow-up of $Y''$ along the above $27$ curves, with exceptional divisors $\Gamma_{hij}:=bl'''^{-1}(\widetilde{\gamma}_{hij})$, $\Lambda_{ij,i}:=bl'''^{-1}(\widetilde{\lambda}_{ij,i})$, $\Lambda_{ij,j}:=bl'''^{-1}(\widetilde{\lambda}_{ij,j})$, $\Lambda_{i}':=bl'''^{-1}(\widetilde{\lambda}_{i}')$, $\Lambda_{i}'':=bl'''^{-1}(\widetilde{\lambda}_{i}'')$ for $1\le h \le 5$ and $1\le i<j \le 3$. We denote by $\mathcal{E}_{h}$, $\mathcal{E}_{i}'$, $\mathcal{E}_i''$ and $\mathcal{E}_{ij}$ respectively the strict transform of $\widetilde{E}_h$, $\widetilde{E}_i'$, $\widetilde{E}_i''$ and $\widetilde{E}_{ij}$; by $\mathcal{F}_{i}$ the strict transform of $F_{i}$; by $\mathcal{R}_{i}$ the strict transform of $R_{i}$; by $\mathcal{H}$ the pullback of $H$.

\begin{remark}\label{rem:Gamma^3p6}
We have that
$$\Gamma_{hij}=\mathbb{P}(\mathcal{N}_{\widetilde{\gamma}_{hij}|Y''})\cong \mathbb{P}(\mathcal{O}_{\widetilde{\gamma}_{hij}}(\widetilde{E}_h)\oplus \mathcal{O}_{\widetilde{\gamma}_{hij}}(\widetilde{t}_{hij})) \cong \mathbb{P}(\mathcal{O}_{\mathbb{P}^1}(-1)\oplus \mathcal{O}_{\mathbb{P}^1}(-1))\cong \mathbb{F}_0,$$
$$\Lambda_{ij,i}=\mathbb{P}(\mathcal{N}_{\widetilde{\lambda}_{ij,i}|Y''})\cong \mathbb{P}(\mathcal{O}_{\widetilde{\lambda}_{ij,i}}(\widetilde{E}_{ij})\oplus \mathcal{O}_{\widetilde{\lambda}_{ij,i}}(\widetilde{\pi}_{ij,i})) \cong \mathbb{P}(\mathcal{O}_{\mathbb{P}^1}(-1)\oplus \mathcal{O}_{\mathbb{P}^1}(-1))\cong \mathbb{F}_0,$$
$$\Lambda_{ij,j}=\mathbb{P}(\mathcal{N}_{\widetilde{\lambda}_{ij,j}|Y''})\cong \mathbb{P}(\mathcal{O}_{\widetilde{\lambda}_{ij,j}}(\widetilde{E}_{ij})\oplus \mathcal{O}_{\widetilde{\lambda}_{ij,j}}(\widetilde{\pi}_{ij,j})) \cong \mathbb{P}(\mathcal{O}_{\mathbb{P}^1}(-1)\oplus \mathcal{O}_{\mathbb{P}^1}(-1))\cong \mathbb{F}_0,$$
$$\Lambda_{i}'=\mathbb{P}(\mathcal{N}_{\widetilde{\lambda}_{i}'|Y''})\cong \mathbb{P}(\mathcal{O}_{\widetilde{\lambda}_{i}'}(\widetilde{E}_{i}')\oplus \mathcal{O}_{\widetilde{\lambda}_{i}'}(\widetilde{\pi}_{a_i'})) \cong \mathbb{P}(\mathcal{O}_{\mathbb{P}^1}(-1)\oplus \mathcal{O}_{\mathbb{P}^1}(-1))\cong \mathbb{F}_0,$$
$$\Lambda_{i}''=\mathbb{P}(\mathcal{N}_{\widetilde{\lambda}_{i}''|Y''})\cong \mathbb{P}(\mathcal{O}_{\widetilde{\lambda}_{i}''}(\widetilde{E}_{i}'')\oplus \mathcal{O}_{\widetilde{\lambda}_{i}''}(\widetilde{\pi}_{a_i''})) \cong \mathbb{P}(\mathcal{O}_{\mathbb{P}^1}(-1)\oplus \mathcal{O}_{\mathbb{P}^1}(-1))\cong \mathbb{F}_0,$$
Furthermore, we have $\Gamma_{hij}^{3}=-\deg (\mathcal{N}_{\widetilde{\gamma}_{hij}|Y''})=2$, 
$\Lambda_{ij,i}^{3}=-\deg (\mathcal{N}_{\widetilde{\lambda}_{ij,i}|Y''})=2$, $\Lambda_{ij,j}^{3}=-\deg (\mathcal{N}_{\widetilde{\lambda}_{ij,j}|Y''})=2$, $\Lambda_{i}'^{3}=-\deg (\mathcal{N}_{\widetilde{\lambda}_{i}'|Y''})=2$, $\Lambda_{i}''^{3}=-\deg (\mathcal{N}_{\widetilde{\lambda}_{i}''|Y''})=2$ (see \cite[Chap 4, \S 6]{GH} and \cite[Lemma 2.2.14]{IsPro99}).
\end{remark}

\begin{remark}\label{rem:intersectionYp6}
Let us take $1\le h \le 5$ and distinct indices $i,j,k\in\{1,2,3\}$ with $i< j$. The divisor $\mathcal{F}_{k}$ intersects $\Gamma_{hst}$, $\Lambda_{k}'$, $\Lambda_{k}''$, $\Lambda_{ij,i}$, $\Lambda_{ij,j}$ each along a $\mathbb{P}^1$ which is a $(-1)$-curve on $\mathcal{F}_{i}$ and a fibre on $\Gamma_{hst}$, $\Lambda_{k}'$, $\Lambda_{k}''$, $\Lambda_{ij,i}$, $\Lambda_{ij,j}$, where $1\le s<t \le 3$ and $k\in\{s,t\}$. Similarly we have $\Lambda_{i}'^2\cdot\mathcal{R}_{i}=\Lambda_{i}''^2\cdot\mathcal{R}_{i}=\Lambda_{ij,i}^2\cdot\mathcal{R}_{i}=\Lambda_{ij,j}^2\cdot\mathcal{R}_{j}=-1$ and $\Lambda_{i}'\cdot\mathcal{R}_{i}^2=\Lambda_{i}''\cdot\mathcal{R}_{i}^2=\Lambda_{ij,i}\cdot\mathcal{R}_{i}^2=\Lambda_{ij,j}^2\cdot\mathcal{R}_{j}=0$.
Let us consider the strict transforms $\widetilde{\alpha}_{hi}$, $\widetilde{\alpha}_{ijk}$, $\widetilde{\alpha}_{i}'$, $\widetilde{\alpha}_{i}''$, $\widetilde{\rho}_{ij,i}$, $\widetilde{\rho}_{ij,j}$, $\widetilde{\rho}_{i}'$, $\widetilde{\rho}_{i}''$ of the curves defined in Remark~\ref{rem:selfalphap6}. Then we have 
$$\widetilde{\alpha}_{hi}^2|_{\mathcal{E}_h}=\mathcal{F}_{i}^2\cdot \mathcal{E}_h=-1, \quad \widetilde{\alpha}_{hi}^2|_{\mathcal{F}_{i}}=\mathcal{E}_h^2\cdot \mathcal{F}_{i}=-2,$$ 
$$\widetilde{\alpha}_{ijk}^2|_{\mathcal{E}_{ij}}={\mathcal{F}_{k}}^2\cdot \mathcal{E}_{ij}=-1, \quad
\widetilde{\alpha}_{ijk}^2|_{\mathcal{F}_{k}}=\mathcal{E}_{ij}^2\cdot \mathcal{F}_{k}=-2,$$
$$\widetilde{\alpha}_{i}'^2|_{\mathcal{E}_{i}'}=\mathcal{F}_{i}^2\cdot\mathcal{E}_{i}'=-1,\quad \widetilde{\alpha}_{i}'^2|_{\mathcal{F}_{i}^2}=\mathcal{E}_{i}'^2\cdot \mathcal{F}_{i}=-1,$$
$$\widetilde{\alpha}_{i}''^2|_{\mathcal{E}_{i}''}=\mathcal{F}_{i}^2\cdot\mathcal{E}_{i}''=-1, \quad \widetilde{\alpha}_{i}''^2|_{\mathcal{F}_{i}^2}=\mathcal{E}_{i}''^2\cdot \mathcal{F}_{i}=-1,$$
$$\widetilde{\rho}_{ij,i}^2|_{\mathcal{E}_{ij}}={\mathcal{R}_{i}}^2\cdot \mathcal{E}_{ij}=-1, \quad \widetilde{\rho}_{ij,j}^2|_{\mathcal{E}_{ij}}={\mathcal{R}_{j}}^2\cdot \mathcal{E}_{ij}=-1,$$
$$\widetilde{\rho}_{ij,i}^2|_{\mathcal{R}_{i}}=\mathcal{E}_{ij}^2\cdot \mathcal{R}_{i}=-1, \quad \widetilde{\rho}_{ij,j}^2|_{\mathcal{R}_{j}}=\mathcal{E}_{ij}^2\cdot \mathcal{R}_{j}=-1,$$
$$\widetilde{\rho}_{i}'^2|_{\mathcal{E}_{i}'}=\mathcal{R}_{i}^2\cdot\mathcal{E}_{i}'=-1, \quad \widetilde{\rho}_{i}'^2|_{\mathcal{R}_{i}}=\mathcal{E}_{i}'^2\cdot \mathcal{R}_{i}=-1,$$
$$\widetilde{\rho}_{i}''^2|_{\mathcal{E}_{i}''}=\mathcal{R}_{i}^2\cdot\mathcal{E}_{i}''=-1, \quad \widetilde{\rho}_{i}''^2|_{\mathcal{R}_{i}}=\mathcal{E}_{i}''^2\cdot \mathcal{R}_{i}=-1.$$
Finally we recall that a general line of $\mathbb{P}^3$ does not intersect the three twisted cubics $C_1$, $C_2$, $C_3$ and their chords; instead a general plane of $\mathbb{P}^3$ intersects each twisted cubic at three points and each chord at one point. Hence we have $\mathcal{H}^2\cdot \mathcal{F}_{i}=\mathcal{H}^2\cdot \mathcal{R}_{i}=0$, $\mathcal{F}_{i}^2\cdot\mathcal{H}=-3$ and $\mathcal{R}_{i}^2\cdot\mathcal{H}=-1$.
\end{remark}

\begin{remark}\label{rem:Ek^3p6}
By construction we have
$${bl'''}^*(\widetilde{E}_h) = \mathcal{E}_h +\sum_{1\le i<j\le 3} \Gamma_{hij},
{bl'''}^*(\widetilde{E}_{ij}) = \mathcal{E}_{ij} + \Lambda_{ij,i}+\Lambda_{ij,j},$$ 
$${bl'''}^*(\widetilde{E}_{i}') = \mathcal{E}_{i}' + \Lambda_{i}', 
{bl'''}^*(\widetilde{E}_{i}'') = \mathcal{E}_{i}'' + \Lambda_{i}''.$$ 
By abuse of notation, we denote $\mathcal{E}_{h}\cap \Gamma_{hij}$, $\mathcal{E}_{ij}\cap \Lambda_{ij,x}$, $\mathcal{E}_{i}'\cap \Lambda_{i}'$ and $\mathcal{E}_{i}''\cap \Lambda_{i}''$ respectively  by $\widetilde{\gamma}_{hij}$, $\widetilde{\lambda}_{ij,i}$, $\widetilde{\lambda}_{ij,j}$, $\widetilde{\lambda}_i'$, $\widetilde{\lambda}_{i}''$ for $1\le h \le 5$ and $1\le i<j\le 3$. Let $\mathcal{L}_{h}$, $\mathcal{L}_{ij}$, $\mathcal{L}_i'$, $\mathcal{L}_{i}''$ be respectively the strict transform on $Y$ of a general line of $E_{h}$, $E_{ij}$, $E_{i}'$ and $E_{i}''$.   
By using similar arguments to the ones in Remark~\ref{rem:Ek^3p9} we obtain that $\mathcal{E}_{h}^3=4$, $\mathcal{E}_{ij}^3=3$ and $\mathcal{E}_i'^3 = \mathcal{E}_i''^3 = 2$, since we have
$\mathcal{E}_h|_{\mathcal{E}_h} \sim - ( \mathcal{L}_h + \sum_{1\le i<j \le 3}\widetilde{\gamma}_{hij} ) \sim -(4\mathcal{L}_h-2\sum_{i=1}^3\widetilde{\alpha}_{hi}),$
$\mathcal{E}_{ij}|_{\mathcal{E}_{ij}} \sim - ( \mathcal{L}_{ij} + \widetilde{\lambda}_{ij,i}+ \widetilde{\lambda}_{ij,j}) \sim -(3\mathcal{L}_{ij}-2\widetilde{\alpha}_{ijk}-\widetilde{\rho}_{ij,i}-\widetilde{\rho}_{ij,j}),$
$\mathcal{E}_{i}'|_{\mathcal{E}_{i}'} \sim - ( \mathcal{L}_{i}' + \widetilde{\lambda}_{i}') \sim -(2\mathcal{L}_{i}'-\widetilde{\alpha}_{i}'-\widetilde{\rho}_{i}')$ and 
$\mathcal{E}_{i}''|_{\mathcal{E}_{i}''} \sim - ( \mathcal{L}_{i}'' + \widetilde{\lambda}_{i}'') \sim -(2\mathcal{L}_{i}''-\widetilde{\alpha}_{i}''-\widetilde{\rho}_{i}'').$
\end{remark}

\begin{remark}\label{rem:Fij^3p6}
By using similar arguments to the ones in Remark~\ref{rem:Fij^3p13} we have $\mathcal{F}_{i}^3=-\deg (\mathcal{N}_{\widetilde{C}_{i}|Y'})=6$ and $\mathcal{R}_{i}^3=-\deg (\mathcal{N}_{\widetilde{r}_i|Y'})=6$ for $1\le i \le 3$.
\end{remark}

Let $\widetilde{P}$ be the strict transform on $Y$ of an 
%general 
element of $\mathcal{P}''$: then
$$\mathcal{P}\sim 7\mathcal{H}-3\sum_{h=1}^5 \mathcal{E}_h-2\sum_{i=1}^3 (\mathcal{E}_{i}'+\mathcal{E}_i'')- 2\sum_{1\le i<j\le 3}\mathcal{E}_{ij}- 2\sum_{i=1}^3 \mathcal{F}_{i}-\sum_{i=1}^3\mathcal{R}_{i}+$$ 
$$-4\sum_{\substack{h=1 \\ 1\le i<j\le 3}}^5 \Gamma_{hij}-3\sum_{i=1}^3 (\Lambda_{i}'+\Lambda_i'')- 3\sum_{1\le i<j\le 3}(\Lambda_{ij,i}+\Lambda_{ij,j}).$$ 
Let us take the linear system $\widetilde{\mathcal{P}}:=|\mathcal{O}_{Y}(\widetilde{P})|$ on $Y$. It is base point free and it defines a morphism $\nu_{\widetilde{\mathcal{P}}}: Y \to \mathbb{P}^{6}$ birational onto the image $W_F^{6}:= \nu_{\widetilde{\mathcal{P}}} (Y)$, which is a threefold of degree $\deg W_F^{6} = 10$. It follows by Lemma~\ref{lem:birationalitanu6} and by the fact that $\widetilde{P}^3=10$ (use Remarks~\ref{rem:Gamma^3p6},~\ref{rem:intersectionYp6},~\ref{rem:Ek^3p6},~\ref{rem:Fij^3p6}). Then we have the following diagram:

$$\begin{tikzcd}
Y \arrow[d, "bl'''"] \arrow[drrr, "\nu_{\widetilde{\mathcal{P}}}"] & & & \\
Y''  \arrow{r}{bl''} & Y' \arrow{r}{bl'} & \mathbb{P}^3 \arrow[dashrightarrow]{r}{\nu_{\mathcal{P}}} & W_F^{6} \subset \mathbb{P}^{6}.
\end{tikzcd}$$

It remains to show that the general hyperplane section of the threefold $W_F^6$ is an Enriques surface. 

\begin{remark}\label{rem:P|E=0}
Let $\widetilde{Q}_6$, $\widetilde{Q}_7$ and $\widetilde{Q}_8$ be the strict transforms on $Y$ of the quadric surfaces $Q_6$, $Q_7$, $Q_8$. By construction we have $\widetilde{P}\cdot \mathcal{E}_h = \widetilde{P}\cdot \mathcal{E}_{ij} = \widetilde{P}\cdot \widetilde{Q}_{i+j+3} = 0$ for a general $\widetilde{P}\in\widetilde{\mathcal{P}}$ and for all $1\le h \le 5$ and $1\le i<j \le 3$. 
\end{remark} 

\begin{remark}\label{rem:nu6blowdown}
The 27 exceptional divisors of $bl''': Y \to Y''$, the six divisors $\mathcal{E}_{i}$ and $\mathcal{E}_i'$, and the three divisors $\mathcal{R}_{i}$ are contracted by the morphism $\nu_{\widetilde{\mathcal{P}}} : Y \to W_{F}^{6} \subset \mathbb{P}^{6}$ to curves of $W_{F}^{6}$. This follows by the fact that $\widetilde{P}\cdot \Gamma_{hij}$, $\widetilde{P}\cdot \Lambda_{ij,i}$, $\widetilde{P}\cdot \Lambda_{ij,j}$, $\widetilde{P}\cdot \Lambda_{i}'$, $\widetilde{P}\cdot \Lambda_{i}''$, $\widetilde{P}\cdot \mathcal{E}_{i}'$, $\widetilde{P}\cdot \mathcal{E}_{i}''$, $\widetilde{P}\cdot \mathcal{R}_{i}\ne 0$ and $\widetilde{P}^2\cdot \Gamma_{hij} = \widetilde{P}^2\cdot \Lambda_{ij,i} = \widetilde{P}^2\cdot \Lambda_{ij,j} = \widetilde{P}^2\cdot \Lambda_{i}' = \widetilde{P}^2\cdot \Lambda_{i}'' = \widetilde{P}^2\cdot \mathcal{E}_{i}'= \widetilde{P}^2\cdot \mathcal{E}_{i}''= \widetilde{P}^2\cdot \mathcal{R}_{i}= 0$ for $1\le h \le 5$ and $1\le i<j\le 3$ (use Remarks~\ref{rem:Gamma^3p6},~\ref{rem:intersectionYp6}).
\end{remark}

\begin{remark}\label{rem:noContractionSpigoli6}
Let us fix $0\le i\le 3$ and let $\widetilde{P}$ be a general element of $\widetilde{\mathcal{P}}$. Since 
$\widetilde{P}^2\cdot \mathcal{F}_{i}=10>0$ (use Remarks~\ref{rem:intersectionYp6},~\ref{rem:Fij^3p6}), then the curve $\widetilde{P}\cap \mathcal{F}_{i}$ is not contracted by the rational map defined by $\widetilde{\mathcal{P}}|_{\widetilde{P}}$. 
\end{remark}

\begin{remark}\label{rem:-1curvenuove6}
Let us fix $1\le h \le 5$ and $1\le i<j\le 3$.
Let us consider a general element $\widetilde{P}\in \widetilde{\mathcal{P}}$ and let us take $S:=\nu_{\widetilde{\mathcal{P}}}(\widetilde{P})$ and $P'':=bl'''(\widetilde{P})\in\mathcal{P}''$. Since $bl''' : Y \to Y''$ has no effect on $P''$, then $\widetilde{P}\cap \Gamma_{hij}$, $\widetilde{P}\cap \Lambda_{ij,i}$, $\widetilde{P}\cap \Lambda_{ij,j}$, $\widetilde{P}\cap \Lambda_{i}'$, $\widetilde{P}\cap \Lambda_{i}''$, $\widetilde{P}\cap \mathcal{E}_{i}'\cong \widetilde{\beta}_{i,P}'$ and $\widetilde{P}\cap \mathcal{E}_{i}''\cong \widetilde{\beta}_{i,P}''$ are still $(-1)$-curves on $\widetilde{P}$ (see Remark~\ref{rem:-1curvesP}). 
We also have that $(\widetilde{P}\cap \mathcal{R}_{i})|_{\widetilde{P}}^2 = \mathcal{R}_{i}^2\cdot \widetilde{P} = -5$ (use Remarks~\ref{rem:intersectionYp6},~\ref{rem:Fij^3p6}). Furthermore, $\widetilde{P}\cap\mathcal{R}_{i}$ intersects the four curves $\widetilde{P}\cap \Lambda_{i}'$, $\widetilde{P}\cap \Lambda_{i}''$, $\widetilde{P}\cap \Lambda_{ij,i}$, $\widetilde{P}\cap \Lambda_{st,i}$ at one point each, for $1\le s<t \le 3$ and $i\in\{s,t\}$
%: indeed we have that $\mathcal{R}_{i}\cdot_{\widetilde{P}} \Lambda_{i}' = \mathcal{R}_{i}\cdot_{\widetilde{P}} \Lambda_{i}'' = \mathcal{R}_{i}\cdot_{\widetilde{P}} \Lambda_{ij,i} = \mathcal{R}_{i}\cdot_{\widetilde{P}} \Lambda_{st,i}= 1$ 
(use Remark~\ref{rem:intersectionYp6}).
Thus, we can see the map $\nu_{\widetilde{\mathcal{P}}} |_{\widetilde{P}} : \widetilde{P} \to S$ as the blow-up of $S$ at the $21$ points $\nu_{\widetilde{\mathcal{P}}}(\widetilde{P}\cap \Gamma_{hij})$, $\nu_{\widetilde{\mathcal{P}}}(\widetilde{P}\cap \mathcal{E}_i')$, $\nu_{\widetilde{\mathcal{P}}}(\widetilde{P}\cap \mathcal{E}_i'')$, at the three points $\nu_{\widetilde{\mathcal{P}}}(\widetilde{P}\cap \mathcal{R}_{i})$ and at the four points $\nu_{\widetilde{\mathcal{P}}}(\widetilde{P}\cap \Lambda_{i}')$, $\nu_{\widetilde{\mathcal{P}}}(\widetilde{P}\cap \Lambda_{i}'')$, $\nu_{\widetilde{\mathcal{P}}}(\widetilde{P}\cap \Lambda_{ij,i})$, $\nu_{\widetilde{\mathcal{P}}}(\widetilde{P}\cap \Lambda_{st,i})$ which are infinitely near to each $\nu_{\widetilde{\mathcal{P}}}(\widetilde{P}\cap \mathcal{R}_{i})$ (see Remarks~\ref{rem:veryample6},~\ref{rem:P|E=0},~\ref{rem:nu6blowdown},~\ref{rem:noContractionSpigoli6}). Then $S$ is a smooth surface.
\end{remark}

\begin{remark}\label{rem:uniqueQQQp6}
The surface $Q_6\cup Q_7 \cup Q_8$ is the only sextic surface of $\mathbb{P}^3$ which is singular along the three twisted cubic $C_1$, $C_2$, $C_3$. Let us consider the strict transforms $\widetilde{Q}_6$, $\widetilde{Q}_7$ and $\widetilde{Q}_8$ on $Y$ of these quadric surfaces. Then we have
$$\widetilde{Q}_6+\widetilde{Q}_7+\widetilde{Q}_8\sim 6\mathcal{H}-\sum_{h=1}^5 3\mathcal{E}_h-\sum_{i=1}^3 2(\mathcal{E}_{i}'+\mathcal{E}_i'')-\sum_{1\le i<j \le 3} 3\mathcal{E}_{ij}- \sum_{i=1}^3 2\mathcal{F}_{i}-\sum_{i=1}^3 2\mathcal{R}_{i}+$$ 
$$-\sum_{\substack{h=1 \\ 1\le i<j\le 3}}^5 4\Gamma_{hij}-\sum_{i=1}^3 4(\Lambda_{i}'+\Lambda_i'')- \sum_{1\le i<j\le 3} 4(\Lambda_{ij,i}+\Lambda_{ij,j}).$$ 
If $\widetilde{P}$ is a general element of $\widetilde{\mathcal{P}}$, then  
\begin{center}
$0\sim (\widetilde{Q}_6+\widetilde{Q}_7+\widetilde{Q}_8)|_{\widetilde{P}} \sim \Big( 6\mathcal{H}-\sum_{i=1}^3 2(\mathcal{E}_{i}'+\mathcal{E}_i'')- \sum_{i=1}^3 2\mathcal{F}_{i}-\sum_{i=1}^3 2\mathcal{R}_{i}+$

$-4\sum_{\substack{h=1,\dots , 5 \\ 1\le i<j\le 3}} \Gamma_{hij}-4\sum_{i=1}^3 (\Lambda_{i}'+\Lambda_i'')- 4\sum_{1\le i<j\le 3}(\Lambda_{ij,i}+\Lambda_{ij,j}) \Big)|_{\widetilde{P}}.$
\end{center}
\end{remark}

\begin{theorem}\label{thm:S6isEnriques}
Let $S$ be a general hyperplane section of the threefold $W_F^{6}\subset \mathbb{P}^{6}$. Then $S$ is an Enriques surface.
\end{theorem} 
\begin{proof}
We recall that $S$ is the image of a general element $\widetilde{P}\in \widetilde{\mathcal{P}}$, via the birational morphism $\nu_{\widetilde{\mathcal{P}}} : Y \to W_{F}^{6} \subset \mathbb{P}^{6}$. Furthermore, $S$ is smooth (see Remark~\ref{rem:-1curvenuove6}).
By Proposition~\ref{prop:tildePSmoothePa0} and by using the arguments of Theorem~\ref{thm:S6isEnriques}, we have that $q(\widetilde{P})=p_g(\widetilde{P})=0$.
It remains to prove that $2K_{S} \sim 0$. Since by \cite[p.187]{GH} we have that
$$K_{Y} =  {bl'''}^*(K_{Y''})+\sum_{\substack{h=1 \\ 1\le i<j\le 3}}^5 \Gamma_{hij}+\sum_{i=1}^3 (\Lambda_{i}'+\Lambda_i'')+\sum_{1\le i<j\le 3}(\Lambda_{ij,i}+\Lambda_{ij,j}) \sim $$
$$\sim-4\mathcal{H}+\sum_{h=1}^5 2\mathcal{E}_h+\sum_{i=1}^3 2(\mathcal{E}_{i}'+\mathcal{E}_i'')+\sum_{1\le i<j \le 3} 2\mathcal{E}_{ij}+ \sum_{i=1}^3 \mathcal{F}_{i}+\sum_{i=1}^3 \mathcal{R}_{i}+$$ 
$$+3\sum_{\substack{h=1 \\ 1\le i<j\le 3}}^5 \Gamma_{hij}+3\sum_{i=1}^3 (\Lambda_{i}'+\Lambda_i'')+ 3\sum_{1\le i<j\le 3}(\Lambda_{ij,i}+\Lambda_{ij,j}),$$
then we obtain that
$2K_{\widetilde{P}} = 2(K_{Y}+\widetilde{P})|_{\widetilde{P}}\sim ( 6\mathcal{H}- \sum_{i=1}^3 2\mathcal{F}_{i}-2\sum_{\substack{h=1,\dots , 5 \\ 1\le i<j\le 3}} \Gamma_{hij} )|_{\widetilde{P}}$.
Furthermore, by Remark~\ref{rem:uniqueQQQp6}, we have that $2K_{\widetilde{P}}$ is linearly equivalent to
$$\Big(\sum_{i=1}^3 2(\mathcal{E}_{i}'+\mathcal{E}_i'')+\sum_{i=1}^3 2\mathcal{R}_{i}+\sum_{\substack{h=1 \\ 1\le i<j\le 3}}^5 2\Gamma_{hij} + \sum_{i=1}^3 4(\Lambda_{i}'+\Lambda_i'')+ \sum_{1\le i<j\le 3} 4(\Lambda_{ij,i}+\Lambda_{ij,j}) \Big)|_{\widetilde{P}}.$$
Finally, by Remark\ref{rem:-1curvenuove6}, we have
$2K_{S}\sim (\nu_{\widetilde{\mathcal{P}}})_{*}(2K_{\widetilde{P}})\sim 0$.
\end{proof}

One can prove that $W_F^{6}\subset \mathbb{P}^{6}$ is not a cone over a general hyperplane section, as in the proof of Theorem~\ref{thm:WF13moderno}. So $W_F^6\subset \mathbb{P}^6$ is an Enriques-Fano threefold of genus $p=\frac{S^3}{2}+1=\frac{\widetilde{P}^3}{2}+1=6$.
\end{proof}

It would be interesting to verify with modern techniques if the general hyperplane section of $W_{F}^6\subset \mathbb{P}^{6}$ actually is a Reye congruence, as stated by Fano in \cite[\S 3]{Fa38} (see also \cite[Proposition 3]{Co83}).

%\subsection{Singularities}

We recall that the divisors $\mathcal{E}_1$, $\mathcal{E}_2$, $\mathcal{E}_3$, $\mathcal{E}_4$, $\mathcal{E}_5$, $\widetilde{Q}_6$, $\widetilde{Q}_7$, $\widetilde{Q}_8$ are contracted by $\nu_{\widetilde{\mathcal{P}}} : Y \to W_F^{6} \subset \mathbb{P}^{6}$ to points of $W_F^{6}$ (see Remark~\ref{rem:P|E=0}).
Let us define $P_h : = \nu_{\widetilde{\mathcal{P}}}(\mathcal{E}_h)$ for $1\le h \le 5$ and $P_6 : = \nu_{\widetilde{\mathcal{P}}}(\widetilde{Q}_6)$, $P_7 : = \nu_{\widetilde{\mathcal{P}}}(\widetilde{Q}_7)$,  $P_8 : = \nu_{\widetilde{\mathcal{P}}}(\widetilde{Q}_8)$.

%$$\widetilde{Q}_6\sim 2\mathcal{H}-\sum_{h=1}^5 \mathcal{E}_h-\sum_{i=1,2} (\mathcal{E}_{i}'+\mathcal{E}_i'')-\sum_{1\le i<j \le 3} \mathcal{E}_{ij}- \sum_{i=1,2} \mathcal{F}_{i}-\sum_{i=1,2} \mathcal{R}_{i}+$$ 
%$$-\sum_{\substack{h=1 \\ 1\le i<j\le 3 \\ (i,j)\ne (1,2)}}^5 \Gamma_{hij}-\sum_{h=1}^5 2\Gamma_{h12}-\sum_{i=1,2} 2(\Lambda_{i}'+\Lambda_i'')- \sum_{1\le i<j\le 3}(\Lambda_{ij,i}+\Lambda_{ij,j}),$$ 
%
%$$\widetilde{Q}_7\sim 2\mathcal{H}-\sum_{h=1}^5 \mathcal{E}_h-\sum_{i=1,3} (\mathcal{E}_{i}'+\mathcal{E}_i'')-\sum_{1\le i<j \le 3} \mathcal{E}_{ij}- \sum_{i=1,3} \mathcal{F}_{i}-\sum_{i=1,3} \mathcal{R}_{i}+$$ 
%$$-\sum_{\substack{h=1 \\ 1\le i<j\le 3}}^5 \Gamma_{hij}-\sum_{h=1}^5 2\Gamma_{h13}-\sum_{i=1,3} 2(\Lambda_{i}'+\Lambda_i'')- \sum_{1\le i<j\le 3}(\Lambda_{ij,i}+\Lambda_{ij,j}),$$
%
%$$\widetilde{Q}_8\sim 2\mathcal{H}-\sum_{h=1}^5 \mathcal{E}_h-\sum_{i=2,3} (\mathcal{E}_{i}'+\mathcal{E}_i'')-\sum_{1\le i<j \le 3} \mathcal{E}_{ij}- \sum_{i=2,3} \mathcal{F}_{i}-\sum_{i=2,3} \mathcal{R}_{i}+$$ 
%$$-\sum_{\substack{h=1 \\ 1\le i<j\le 3}}^5 \Gamma_{hij}-\sum_{h=1}^5 2\Gamma_{h23}-\sum_{i=2,3} 2(\Lambda_{i}'+\Lambda_i'')- \sum_{1\le i<j\le 3}(\Lambda_{ij,i}+\Lambda_{ij,j}),$$

\begin{remark}\label{rem:3to3p6}
By Remark~\ref{rem:P|E=0} we have that $\nu_{\widetilde{\mathcal{P}}}(\mathcal{E}_{ij})$ is a point of $W_F^{6}$ for all $1\le i<j \le 3$. In particular we have $\nu_{\widetilde{\mathcal{P}}}(\mathcal{E}_{ij})= \nu_{\widetilde{\mathcal{P}}}(\widetilde{Q}_{i+j+3})$, since $\widetilde{Q}_{i+j+3}\cap \mathcal{E}_{ij}\ne \emptyset$. Indeed, one can verify that $\widetilde{Q}_{i+j+3}\cap \mathcal{E}_{ij}$ is the strict transform of the line of $E_{ij}$ joining the points $E_{ij}\cap \widetilde{r}_i$ and $E_{ij}\cap \widetilde{r}_j$.
% praticamente r_i\cap r_j è un punto semplice per la quadrica Q_{i+j+3}
% quindi il tangente è un piano... che interseca Q_{i+j+3} proprio nelle due retta r_i ed r_j
% quindi Q_{i+j+3} interseca E_{ij} nel proiettificato di questo piano
% che è la retta individuata dalle trasf. strette di r_i ed r_j su E_{ij} 
\end{remark}

\begin{proposition}\label{prop:quadruplePointsp6}
The eight points $P_1, \dots , P_8$, defined as above, are quadruple points of $W_F^{6}$ whose tangent cone is a cone over a Veronese surface.
\end{proposition}
\begin{proof}
The analysis of the points $P_1$, $P_2$, $P_3$, $P_4$ and $P_5$ follows by Remark~\ref{rem:Ek^3p6}, as in the proof of Proposition~\ref{prop:quadruplePointsp13}. Let us fix now three distinct indices $i,j,k\in\{1,2,3\}$ with $i<j$. 
%The hyperplane sections of $W_F^{6}\subset \mathbb{P}^{6}$ passing through $P_{i+j+3}$ correspond to the septic surfaces of $\mathcal{P}$ containing the quadric surface $Q_{i+j+3}$. Let $\mathcal{P}_{ij}$ be the sublinear system of $\mathcal{P}$ defined by these septics. The movable part of $\mathcal{P}_{ij}$ is given by the quintic surfaces of $\mathbb{P}^3$ containing the two twisted cubics $C_i,C_j\subset Q_{i+j+3}$, containing the line $r_k$, and with double points along the twisted cubic $C_k$. Such a surface cuts on $Q_{i+j+3}$ a curve of degree $10$ given by $C_i\cup C_j$ and by a variable quartic curve, which is a quadric section of $Q_{i+j+3}$, with node at $r_i\cap r_j$.
The hyperplane sections of $W_F^{6}\subset \mathbb{P}^{6}$ passing through $P_{i+j+3}$ correspond to the elements of $\widetilde{\mathcal{P}}$ containing $\widetilde{Q}_{i+j+3}\cup \mathcal{E}_{ij}$ (see Remark~\ref{rem:3to3p6}). Let us take the sublinear system $\widetilde{\mathcal{P}}_{ij}:=\widetilde{\mathcal{P}}-\widetilde{Q}_{i+j+3}-\mathcal{E}_{ij}$ of $\widetilde{\mathcal{P}}$ defined by these elements. 
Let us study $\widetilde{\mathcal{P}}_{ij}|_{\widetilde{Q}_{i+j+3}} = |\mathcal{O}_{\widetilde{Q}_{i+j+3}}(-\widetilde{Q}_{i+j+3}-\mathcal{E}_{ij})|$. If we consider the case $(i,j)=(1,2)$, we have 
$$\widetilde{Q}_6\sim_{Y} 2\mathcal{H}-\sum_{h=1}^5 \mathcal{E}_h-\sum_{t=1,2} (\mathcal{E}_{t}'+\mathcal{E}_t'')-\sum_{1\le r<s \le 3} \mathcal{E}_{rs}- \sum_{t=1,2}\mathcal{F}_{t}-\sum_{t=1,2}\mathcal{R}_{t}+$$ 
$$-\sum_{h=1}^5 (2\Gamma_{h12}+\Gamma_{h13}+\Gamma_{h23})-\sum_{t=1,2} 2(\Lambda_{t}'+\Lambda_t'')- \sum_{1\le r<s\le 3}(\Lambda_{rs,r}+\Lambda_{rs,s}),$$ 
and so
$$\widetilde{Q}_6|_{\widetilde{Q}_6}\sim_{\widetilde{Q}_6} \Big( 2\mathcal{H}-\sum_{1\le r<s \le 3} \mathcal{E}_{rs}-\sum_{t=1,2} (\mathcal{F}_t+\mathcal{R}_t)-\sum_{h=1}^5 2\Gamma_{h12}-\sum_{t=1,2} 2(\Lambda_{t}'+\Lambda_t'') \Big)|_{\widetilde{Q}_6}.$$ 
Let $\mathcal{C}_6$ be the pullback on $\widetilde{Q}_6$ of the linear equivalence class of the hyperplane sections of $Q_6$. By abuse of notation, let us denote by $\widetilde{\gamma}_{h12}$, $\widetilde{\lambda}_{t}'$, $\widetilde{\lambda}_{t}''$ the $(-1)$-curves on $\widetilde{Q}_6$ given by $\Gamma_{h12}|_{\widetilde{Q}_6}$, $\Lambda_{t}'|_{\widetilde{Q}_6}$, $\Lambda_{t}''|_{\widetilde{Q}_6}$ for $1\le h\le 5$ and $t=1,2$. Let us also consider the $(-1)$-curves on $\widetilde{Q}_6$ defined by $\epsilon_{rs}:=\mathcal{E}_{rs}'|_{\widetilde{Q}_6}$ for $1\le r<s \le 3$.
Then
\begin{center}
$\widetilde{Q}_6|_{\widetilde{Q}_6}\sim_{\widetilde{Q}_6} 2\mathcal{C}_6-\sum_{1\le r<s \le 3} \epsilon_{rs}-(2\mathcal{C}_6-\sum_{h=1}^5 \widetilde{\gamma}_{h12}-\widetilde{\lambda}_{1}'-\widetilde{\lambda}_1''-\epsilon_{23}-\epsilon_{12})+$

$-(2\mathcal{C}_6-\sum_{h=1}^5 \widetilde{\gamma}_{h12}-\widetilde{\lambda}_{2}'-\widetilde{\lambda}_2''-\epsilon_{13}-\epsilon_{12} )-\sum_{h=1}^5 2\widetilde{\gamma}_{h12} -\sum_{t=1,2} (\widetilde{\gamma}_{t}'+\widetilde{\gamma}_t'') = -2\mathcal{C}_6+ \epsilon_{12}= -2\mathcal{C}_{i+j+3}+ \epsilon_{ij}.$ \end{center}
Similarly for $(i,j)\in\{(1,3),(2,3)\}$.
Thus, we have 
$$\widetilde{\mathcal{P}}_{ij}|_{\widetilde{Q}_{i+j+3}} = |\mathcal{O}_{\widetilde{Q}_{i+j+3}}(2\mathcal{C}_{i+j+3}-2\epsilon_{ij})|,$$ 
which is the linear system of the quadric sections of $Q_{i+j+3}$ with node at the point $r_i\cap r_j$.
It is known that $Q_{i+j+3}$ is image of the rational map $\psi : \mathbb{P}^2 \dashrightarrow \mathbb{P}^{3}$ defined by the linear system of the conics passing through two fixed points $x_1$ and $x_2$. 
%So a general quadric section of $Q_{i+j+3}$ corresponds to a quartic plane curve with node at the points $x_1$ and $x_2$. 
The quadric sections of $Q_{i+j+3}$ with node at $r_i\cap r_j$ correspond to the quartic plane curves with node at the points $x_1$ and $x_2$ and at a third fixed point $x_3:=\psi^{-1}(r_i\cap r_j)$.
%Let us denote by $\widetilde{\mathcal{P}}_{ij}$ the strict transform on $Y$ of $\mathcal{P}_{ij}$. 
By applying a quadratic transformation, we obtain that $\widetilde{\mathcal{P}}_{ij}|_{\widetilde{Q}_{i+j+3}} \cong |\mathcal{O}_{\mathbb{P}^2}(2)|$, 
whose image is a Veronese surface $V_{ij}$.
% una quadrica in P3 è immagine di P2 tramite il sistema lineare di coniche per due punti...
% quindi effettivamente le sezioni iperpiane di una quadrica coincidono con coniche piane...
Furthermore, we have that $\widetilde{\mathcal{P}}_{ij}|_{\mathcal{E}_{ij}} = |\mathcal{O}_{\mathcal{E}_{ij}}(-\widetilde{Q}_{i+j+3}-\mathcal{E}_{ij})|=|\mathcal{O}_{\mathcal{E}_{ij}}(2\mathcal{L}_{ij}-2\widetilde{\alpha}_{ijk})|\cong \mathbb{P}^2$ 
%where $1\le k\le 3$ with $k\ne i$ and $k\ne j$ 
(see Remark~\ref{rem:Ek^3p6}). Since $\widetilde{\mathcal{P}}_{ij}|_{\mathcal{E}_{ij}}$ is isomorphic to the linear system of the conics of $E_{ij}$ with node at the point $E_{ij}\cap \widetilde{C}_{k}$, then its image is a conic $C_{ij}$. Since $V_{ij}\cup C_{ij} = \mathbb{P}(TC_{P_{i+j+3}}W_F^6)$, then it must be $C_{ij} \subset V_{ij} = \mathbb{P}(TC_{P_{i+j+3}}W_F^6$). Therefore $\widetilde{Q}_{i+j+3}$ is contracted by $\nu_{\widetilde{\mathcal{P}}}$ to the point $P_{i+j+3}$, which is a quadruple point whose tangent cone tangent is a cone over a Veronese surface, and the divisor $\mathcal{E}_{ij}$ is contracted in a conic contained in the Veronese surface given by the exceptional divisor of the minimal resolution of $P_{i+j+3}$. 
\end{proof}

We recall that $\nu_{\mathcal{P}} : \mathbb{P}^{3} \dashrightarrow  W_F^6 \subset \mathbb{P}^{6}$ is an isomorphism outside $Q_6\cup Q_7\cup Q_8$ (see Remark~\ref{rem:veryample6}). Then $P_1$, $P_2$, $P_3$, $P_4$, $P_5$, $P_6$, $P_7$ and $P_8$, are the only singular points of $W_F^{6}$ (see Remarks~\ref{rem:P|E=0},~\ref{rem:nu6blowdown},~\ref{rem:noContractionSpigoli6}). Furthermore, $\nu_{\widetilde{\mathcal{P}}} : Y \to W_F^6$ is a desingularization of $W_F^6$ but it is not the minimal one: indeed, the proof of Proposition~\ref{prop:quadruplePointsp6} says us that $\nu_{\widetilde{\mathcal{P}}} : Y \to W_F^6$ is the blow-up of the minimal desingularization of $W_F^6$ along curves (conics) contained in the minimal resolutions of $P_6$, $P_7$ and $P_8$. Finally we have the following result.

\begin{theorem}\label{thm:associatedo6}
The eight points $P_1$, $P_2$, $P_3$, $P_4$, $P_5$, $P_6$, $P_7$, $P_8$ are all associated with each other.
\end{theorem}
\begin{proof}
Let us fix $1\le h<t\le 5$ and $1\le i<j \le 3$. 
Let us consider the line $l_{ht} \subset \mathbb{P}^3$ joining the points $q_h$ and $q_t$. Let $\widetilde{l}_{ht}$ be its strict transform on $Y$. We obtain that $\nu_{\widetilde{\mathcal{P}}}(\widetilde{l}_{ht}) = \left\langle P_h, P_t \right\rangle \subset W_F^{6}$, since $\widetilde{l}_{ht} \cap \mathcal{E}_h \ne \emptyset$, $\widetilde{l}_{ht} \cap \mathcal{E}_t \ne \emptyset$ and $\deg (\nu_{\widetilde{\mathcal{P}}}(\widetilde{l}_{ht}))=\widetilde{P}\cdot (\mathcal{H}-\mathcal{E}_h-\mathcal{E}_t-\sum_{1\le i<j \le 3}\Gamma_{hij}-\sum_{1\le i<j \le 3}\Gamma_{tij})^2=1$. So $P_h$ is associated with $P_t$.
We recall now that $\Gamma_{hij}$, $\Lambda_{ij,i}$, $\Lambda_{ij,j}$, $\Lambda_{i}'$, $\Lambda_{i}''$ and $\mathcal{R}_i$ are mapped by $\nu_{\widetilde{\mathcal{P}}} : Y \to W_F^{6}\subset \mathbb{P}^{6}$ to curves of $W_F^{6}$ (see Remark~\ref{rem:nu6blowdown}). In particular $\Gamma_{hij}$, $\Lambda_{ij,i}$, $\Lambda_{ij,j}$, $\Lambda_{i}'$, $\Lambda_{i}''$ are mapped to lines of $W_F^{6}$ (use similar arguments of proof of Theorem~\ref{thm:associatedo13}). We also have that $\nu_{\widetilde{\mathcal{P}}}(\mathcal{R}_{i})= \nu_{\widetilde{\mathcal{P}}}(\Lambda_{i}')=\nu_{\widetilde{\mathcal{P}}}(\Lambda_{i}'')=\nu_{\widetilde{\mathcal{P}}}( \Lambda_{ij,i}) = \nu_{\widetilde{\mathcal{P}}}(\Lambda_{st,i})$ for $1\le s<t \le 3$ and $i\in\{s,t\}$ (see Remark~\ref{rem:-1curvenuove6}).
Since $\Gamma_{hij}\cap \mathcal{E}_{h} \ne \emptyset$ and $\Gamma_{hij}\cap \widetilde{Q}_{i+j+3} \ne \emptyset$, then $\left\langle P_h, P_{i+j+3} \right\rangle= \nu_{\widetilde{\mathcal{P}}}(\Gamma_{hij})$. So each $P_h$ is associated with each $P_{i+j+3}$.
Finally $P_6$, $P_7$ and $P_8$ are mutually associated, since $\nu_{\widetilde{\mathcal{P}}}(\mathcal{R}_1)=\left\langle P_6, P_7 \right\rangle$, $\nu_{\widetilde{\mathcal{P}}}(\mathcal{R}_2)=\left\langle P_6, P_8 \right\rangle$, $\nu_{\widetilde{\mathcal{P}}}(\mathcal{R}_3)=\left\langle P_7, P_8 \right\rangle$ (see Remark~\ref{rem:3to3p6}).
\end{proof}

\appendix
\section{Appendix: Macaulay2 codes}\label{app:code}

\subsection{Computational analysis of the Enriques-Fano threefold of genus 13}\label{code:fano13}
\begin{scriptsize}
\begin{verbatim}
Macaulay2, version 1.11
with packages: ConwayPolynomials, Elimination, IntegralClosure, InverseSystems,
               LLLBases, PrimaryDecomposition, ReesAlgebra, TangentCone
i1 : needsPackage "Cremona";
i2 : PP3 = ZZ/10000019[s_0..s_3];
i3 : PP13 = ZZ/10000019[w_0..w_13];
i4 : sextiesSigma = rationalMap map(PP3,PP13, matrix{{s_0*s_1^3*s_2*s_3, s_0^2*s_1^2*s_2^2,
      s_0^2*s_1^2*s_2*s_3, s_0^2*s_1^2*s_3^2, s_0^3*s_1*s_2*s_3, s_0*s_1^2*s_2^2*s_3, 
      s_0*s_1^2*s_2*s_3^2,  s_0^2*s_1*s_2^2*s_3, s_0^2*s_1*s_2*s_3^2, s_1^2*s_2^2*s_3^2,
      s_0*s_1*s_2^3*s_3, s_0*s_1*s_2^2*s_3^2, s_0*s_1*s_2*s_3^3, s_0^2*s_2^2*s_3^2}});
i5 : WF13 = image sextiesSigma;
i6 : (dim WF13 -1, degree WF13) == (3, 24)
i7 : -- singular points of WF13 
     P1 = ideal{w_0,w_1,w_2,w_3,w_5,w_6,w_7,w_8,w_9,w_10,w_11,w_12,w_13};
i8 : sub(WF13, {(gens PP13)_4=>1});
i9 : ConeP1 = tangentCone oo
i10 : degree oo == 4
i11 : P2 = ideal{w_1,w_2,w_3,w_4,w_5,w_6,w_7,w_8,w_9,w_10,w_11,w_12,w_13};
i12 : sub(WF13, {(gens PP13)_0=>1});
i13 : ConeP2 = tangentCone oo
i14 : degree oo == 4
i15 : P3 = ideal{w_0,w_1,w_2,w_3,w_4,w_5,w_6,w_7,w_8,w_9,w_11,w_12,w_13};
i16 : sub(WF13, {(gens PP13)_10=>1});
i17 : ConeP3 = tangentCone oo
i18 : degree oo == 4
i19 : P4 = ideal{w_0,w_1,w_2,w_3,w_4,w_5,w_6,w_7,w_8,w_9,w_10,w_11,w_13};
i20 : sub(WF13, {(gens PP13)_12=>1});
i21 : ConeP4 = tangentCone oo
i22 : degree oo == 4
i23 : P1' = ideal{w_0,w_1,w_2,w_3,w_5,w_4,w_6,w_7,w_8,w_10,w_11,w_12,w_13};
i24 : sub(WF13, {(gens PP13)_4=>1});
i25 : ConeP1' = tangentCone oo
i26 : degree oo == 4
i27 : P2' = ideal{w_0,w_1,w_2,w_3,w_4,w_5,w_6,w_7,w_8,w_9,w_10,w_11,w_12};
i28 : sub(WF13, {(gens PP13)_13=>1});
i29 : ConeP2' = tangentCone oo
i30 : degree oo == 4
i31 : P3' = ideal{w_0,w_1,w_2,w_4,w_5,w_6,w_7,w_8,w_9,w_10,w_11,w_12,w_13};
i32 : sub(WF13, {(gens PP13)_3=>1});
i33 : ConeP3' = tangentCone oo
i34 : degree oo == 4
i35 : P4' = ideal{w_0,w_2,w_3,w_4,w_5,w_6,w_7,w_8,w_9,w_10,w_11,w_12,w_13};
i36 : sub(WF13, {(gens PP13)_1=>1});
i37 : ConeP4' = tangentCone oo
i38 : degree oo == 4
i39 : -- let us see if the lines lij joining the points Pi and Pj
      -- are contained in the threefold Wf13
      l12 = ideal{(toMap(saturate(P1*P2),1,1)).matrix};
i40 : (l12 + WF13 == l12) == false
i41 : l13 = ideal{(toMap(saturate(P1*P3),1,1)).matrix};
i42 : (l13 + WF13 == l13) == false
i43 : l14 = ideal{(toMap(saturate(P1*P4),1,1)).matrix};
i44 : (l14 + WF13 == l14) == false
i45 : l11' = ideal{(toMap(saturate(P1*P1'),1,1)).matrix};
i46 : (l11' + WF13 == l11') == false
i47 : l12' = ideal{(toMap(saturate(P1*P2'),1,1)).matrix};
i48 : (l12' + WF13 == l12') == true
i49 : l13' = ideal{(toMap(saturate(P1*P3'),1,1)).matrix};
i50 : (l13' + WF13 == l13') == true
i51 : l14' = ideal{(toMap(saturate(P1*P4'),1,1)).matrix};
i52: (l14' + WF13 == l14') == true
i53 : -- etc..         
\end{verbatim}
\end{scriptsize}

\subsection{Computational analysis of the Enriques-Fano threefold of genus 9}\label{code:fano9}

\begin{scriptsize}
\begin{verbatim}
Macaulay2, version 1.11
with packages: ConwayPolynomials, Elimination, IntegralClosure, InverseSystems,
               LLLBases, PrimaryDecomposition, ReesAlgebra, TangentCone
i1 : -- let us construct the Enriques-Fano threefold described by Bayle
     -- in the section 6.1.4 of his paper of 1994 
     -- we will denote it by WB9 and we will see that it can be obtained 
     -- as the image of the map defined by the linear system of the septic
     -- surfaces of PP3 double along the edges of two trihedra T and T'
     -- so it is of the form of the Enriques-Fano threefold of genus 9
     -- discovered by Fano, and so it makes sense to study
     -- the way in which their singularities are associated
     needsPackage "Cremona";
i2 : PP5 = ZZ/10000019[x_0, x_1, x_2, y_3, y_4, y_5];
i3 : s1 = x_0^2-3*x_1^2+2*x_2^2;
i4 : s2 = 3*x_0^2-8*x_1^2+5*x_2^2;
i5 : r1 = 3*y_3^2-8*y_4^2+5*y_5^2;
i6 : r2 = y_3^2-3*y_4^2+2*y_5^2;
i7 : X = ideal{s1+r1, s2+r2};
i8 : (dim oo -1, degree oo) == (3,4)
i9 : PP11 = ZZ/10000019[Z_0..Z_11];
i10 : phi = rationalMap map(PP5, PP11, matrix(PP5,{{x_0^2, x_1^2, x_2^2, x_0*x_1, 
      x_0*x_2, x_1*x_2, y_3^2, y_4^2, y_5^2, y_3*y_4, y_3*y_5, y_4*y_5}}));
i11 : phi(X)
i12 : (dim oo -1, degree oo) == (3,16)
i13 : H9 = ideal{ooo_0, ooo_1}
i14 : phi(X) + H9 == phi(X)
i15 : PP9 = ZZ/10000019[w_0..w_9];
i16 : inclusion = rationalMap map(PP9,PP11, matrix(PP9,{{w_0+21*w_4-55*w_5+34*w_6, 
       w_0+8*w_4-21*w_5+13*w_6, w_0, w_1, w_2, w_3, w_4, w_5, w_6, w_7, w_8, w_9 }}));
i17 : image oo == H9
i18 : WB9 = inclusion^* (phi(X));
i19 : (dim oo -1, degree oo) == (3, 16)
i20 : rationalMap map(PP11,PP9, sub(matrix inverseMap(inclusion||H9), PP11))
i21 : pigreca = phi* oo
i23 : fixedPlanex = associatedPrimes (X+ideal{x_0,x_1,x_2});
i24 : fixedPlaney = associatedPrimes (X+ideal{y_3,y_4,y_5});
i25 : -- let us take the singular points of WB9
      P1 = inclusion^* phi(fixedPlaney#0);
i26 : P4 = inclusion^* phi(fixedPlaney#1);
i27 : P2 = inclusion^* phi(fixedPlaney#2);
i28 : P3 = inclusion^* phi(fixedPlaney#3);
i29 : P1' = inclusion^* phi(fixedPlanex#0);
i30 : P4' = inclusion^* phi(fixedPlanex#1);
i31 : P2' = inclusion^* phi(fixedPlanex#2);
i32 : P3' = inclusion^* phi(fixedPlanex#3);
i33 : -- let us see if the lines lij joining the points Pi and Pj
      -- are contained in the threefold WB9
      l12 = ideal{(toMap(saturate(P1*P2),1,1)).matrix};
i34 : (l12 + WB9 == l12) == false
i35 : l13 = ideal{(toMap(saturate(P1*P3),1,1)).matrix};
i36 : (l13 + WB9 == l13) == false
i37 : l14 = ideal{(toMap(saturate(P1*P4),1,1)).matrix};
i38 : (l14 + WB9 == l14) == false
i39 : l11' = ideal{(toMap(saturate(P1*P1'),1,1)).matrix};
i40 : (l11' + WB9 == l11') == true
i41 : l12' = ideal{(toMap(saturate(P1*P2'),1,1)).matrix};
i42 : (l12' + WB9 == l12') == true
i43 : l13' = ideal{(toMap(saturate(P1*P3'),1,1)).matrix};
i44 : (l13' + WB9 == l13') == true
i45 : l14' = ideal{(toMap(saturate(P1*P4'),1,1)).matrix};
i46 : (l14' + WB9 == l14') == true
i47 : -- etc..
      proj1 = rationalMap toMap(P2,1,1);
i48 : proj2 = rationalMap toMap(proj1(P3),1,1);
i49 : proj3 = rationalMap toMap(proj2(proj1(P4)),1,1);
i50 : proj4 = rationalMap toMap(proj3(proj2(proj1(P2'))),1,1);
i51 : proj5 = rationalMap toMap(proj4(proj3(proj2(proj1(P3')))),1,1);
i52 : proj6 = rationalMap toMap(proj5(proj4(proj3(proj2(proj1(P4'))))),1,1);
i53 : proj = proj1*proj2*proj3*proj4*proj5*proj6;
i54 : proj(WB9)
i55 : PP3 = ring oo;
i56 : isBirational( proj|WB9 )
i57 : septics = rationalMap map( PP3, PP9, matrix(inverseMap( proj|WB9 )));
i58 : time image oo == WB9
i59 : comp = associatedPrimes(ideal septics)
i60 : l3' = comp#0;
i61 : l2' = comp#1;
i62 : r21 = comp#2;
i63 : r11 = comp#3;
i64 : r31 = comp#4;
i65 : l1' = comp#5;
i66 : r23 = comp#6;
i67 : r13 = comp#7;
i68 : r33 = comp#8;
i69 : r22 = comp#9;
i70 : r12 = comp#10;
i71 : r32 = comp#11;
i72 : l1 = comp#12;
i73 : l2 = comp#13;
i74 : l3 = comp#14;
i75 : -- trihedron T' :
      f1' = ideal{(gens PP3)_3};
i76 : f2' = ideal{(gens PP3)_1+(gens PP3)_3};
i77 : f3' = ideal{(gens PP3)_2+(gens PP3)_3};
i78 : f1'+f2' == l3'
i79 : f1'+f3' == l2'
i80 : f2'+f3' == l1'
i81 : v'= saturate(f1'+f2'+f3')
i82 : -- trihedron T :
      f1 = ideal{(gens PP3)_0-55*(gens PP3)_1+34*(gens PP3)_2};
i83 : f2 = ideal{(gens PP3)_0  - 21*(gens PP3)_1  +13*(gens PP3)_2};
i84 : f3 = ideal{(gens PP3)_0};
i85 : f1+f2 == l3
i86 : f1+f3 == l2
i87 : f2+f3 == l1
i88 : v = saturate(l1+l2+l3)
i89 : r11 == f1+f1'
i90 : r12 == f1+f2'
i91 : r13 == f1+f3'
i92 : r21 == f2+f1'
i93 : r22 == f2+f2'
i94 : r23 == f2+f3'
i95 : r31 == f3+f1'
i96 : r32 == f3+f2'
i97 : r33 == f3+f3'
i98 : -- general septic surface of the linear system :       
      K = septics^* ideal{random(1,PP9)};
i99 : (dim oo -1, degree oo) == (2,7)
i100 : -- K has double point along l1,l2,l3,l1',l2',l3' :
       (minors(1,jacobian(K))+l1 == l1) == true
i101 : (minors(1,jacobian(K))+l2 == l2) == true
i102 : (minors(1,jacobian(K))+l3 == l3) == true
i103 : (minors(1,jacobian(K))+l1' == l1') == true
i104 : (minors(1,jacobian(K))+l2' == l2') == true
i105 : (minors(1,jacobian(K))+l3' == l3') == true
i106 : -- K has triple point at v and v' :
       (minors(1,jacobian(jacobian(K)))+minors(1,jacobian(K))+v == v) == true
i107 : (minors(1,jacobian(jacobian(K)))+minors(1,jacobian(K))+v' == v') == true
i108 : -- remark
       septics(f1) == P2
i109 : septics(f1') == P2'
i110 : septics(f2) == P3
i111 : septics(f2') == P3'
i112 : septics(f3) == P4
i113 : septics(f3') == P4'
\end{verbatim}
\end{scriptsize}

\subsection{Computational analysis of the Enriques-Fano threefold of genus 7}\label{code:fano7}
%We will use Remark~\ref{rem:WF7projWF13}, Theorem~\ref{thm:B13=F13} and Remark~\ref{rem: configuration bayle13}.
\begin{scriptsize}
\begin{verbatim}
Macaulay2, version 1.11
with packages: ConwayPolynomials, Elimination, IntegralClosure, InverseSystems,
               LLLBases, PrimaryDecomposition, ReesAlgebra, TangentCone
i1 : needsPackage "Cremona";
i2 : PP3 = ZZ/10000019[s_0..s_3];
i3 : -- edges of the trivial tetrahedon :
     l12 = ideal{(gens PP3)_1, (gens PP3)_2};
i4 : l13 = ideal{(gens PP3)_1, (gens PP3)_3};
i5 : l23 = ideal{(gens PP3)_2, (gens PP3)_3};
i6 : l01 = ideal{(gens PP3)_0, (gens PP3)_1};
i7 : l02 = ideal{(gens PP3)_0, (gens PP3)_2};
i8 : l03 = ideal{(gens PP3)_0, (gens PP3)_3};
i9 : PP13 = ZZ/10000019[w_0..w_13];
i10 : sextiesSigma = rationalMap map(PP3,PP13, matrix{{s_0*s_1^3*s_2*s_3, s_0^2*s_1^2*s_2^2,
      s_0^2*s_1^2*s_2*s_3, s_0^2*s_1^2*s_3^2, s_0^3*s_1*s_2*s_3, s_0*s_1^2*s_2^2*s_3, 
      s_0*s_1^2*s_2*s_3^2,  s_0^2*s_1*s_2^2*s_3, s_0^2*s_1*s_2*s_3^2, s_1^2*s_2^2*s_3^2,
      s_0*s_1*s_2^3*s_3, s_0*s_1*s_2^2*s_3^2, s_0*s_1*s_2*s_3^3, s_0^2*s_2^2*s_3^2}});
i11 : WF13 = image sextiesSigma;
i12 : (dim WF13 -1, degree WF13) == (3, 24)
i13 : -- singular points of WF13 
      q1 = ideal{w_0,w_1,w_2,w_3,w_5,w_6,w_7,w_8,w_9,w_10,w_11,w_12,w_13};
i14 : q2 = ideal{w_1,w_2,w_3,w_4,w_5,w_6,w_7,w_8,w_9,w_10,w_11,w_12,w_13};
i15 : q3 = ideal{w_0,w_1,w_2,w_3,w_4,w_5,w_6,w_7,w_8,w_9,w_11,w_12,w_13};
i16 : q4 = ideal{w_0,w_1,w_2,w_3,w_4,w_5,w_6,w_7,w_8,w_9,w_10,w_11,w_13};
i17 : q1' = ideal{w_0,w_1,w_2,w_3,w_5,w_4,w_6,w_7,w_8,w_10,w_11,w_12,w_13};
i18 : q2' = ideal{w_0,w_1,w_2,w_3,w_4,w_5,w_6,w_7,w_8,w_9,w_10,w_11,w_12};
i19 : q3' = ideal{w_0,w_1,w_2,w_4,w_5,w_6,w_7,w_8,w_9,w_10,w_11,w_12,w_13};
i20 : q4' = ideal{w_0,w_2,w_3,w_4,w_5,w_6,w_7,w_8,w_9,w_10,w_11,w_12,w_13};
i21 : -- let us take a general plane of PP3
      -- and the intersection points with the edges lij
      -- plane = ideal{random(1,PP3)}
      -- example:
      plane = ideal{s_0+s_1+s_2+s_3};
i22 : PP2 = ZZ/10000019[x_0,x_1,x_2];
i23 : inclusion = rationalMap map(PP2,PP3, matrix{{-(gens PP2)_0-(gens PP2)_1-(gens PP2)_2, 
      (gens PP2)_0, (gens PP2)_1, (gens PP2)_2}})
i24 : image oo == plane
i25 : p01 = saturate(plane+l01) -- [0:0:-1:1]
i26 : p02 = saturate(plane+l02) -- [0:-1:0:1]
i27 : p03 = saturate(plane+l03) -- [0:-1:1:0]
i28 : p12 = saturate(plane+l12) -- [-1:0:0:1]
i29 : p13 = saturate(plane+l13) -- [-1:0:1:0]
i30 : p23 = saturate(plane+l23) -- [-1:1:0:0]
i31 : a01=inclusion^*p01; -- [0:-1:1]
i32 : a02=inclusion^*p02; -- [-1:0:1]
i33 : a03=inclusion^*p03; -- [-1:1:0]
i34 : a12=inclusion^*p12; -- [0:0:1]
i35 : a13=inclusion^*p13; -- [0:1:0]
i36 : a23=inclusion^*p23; -- [1:0:0]
i37 : -- in the above plane, let us take 
      -- a general cubic curve through the six points pij
      -- rationalMap toMap(saturate(a01*a02*a03*a12*a13*a23),3,1);
      -- cubicThrough6points = oo^*(ideal{random(1,ring(image oo))})
      -- example:
      cubicThrough6points = ideal{(gens PP2)_1^2*(gens PP2)_2+(gens PP2)_1*(gens PP2)_2^2+
      (gens PP2)_0^2*(gens PP2)_2+(gens PP2)_0*(gens PP2)_2^2+(gens PP2)_0^2*(gens PP2)_1+
      (gens PP2)_0*(gens PP2)_1^2+(gens PP2)_0*(gens PP2)_1*(gens PP2)_2}
--              2        2    2               2        2      2
-- o37 = ideal(x x  + x x  + x x  + x x x  + x x  + x x  + x x )
--              0 1    0 1    0 2    0 1 2    1 2    0 2    1 2
i38 : (dim oo -1, degree oo, genus oo)==(1,3,1)
i39 : delta = inclusion(cubicThrough6points)
--                                  2        2    2               2        2      2
-- o39 = ideal (s  + s  + s  + s , s s  + s s  + s s  + s s s  + s s  + s s  + s s )
--               0    1    2    3   1 2    1 2    1 3    1 2 3    2 3    1 3    2 3
i40 : (dim oo -1, degree oo, genus oo)==(1,3,1)
i41 : (delta+l01==p01) == true
i42 : (delta+l02==p02) == true
i43 : (delta+l03==p03) == true
i44 : (delta+l12==p12) == true
i45 : (delta+l13==p13) == true
i46 : (delta+l23==p23) == true
i47 : nudelta=sextiesSigma(delta)
i48 : (dim oo -1, degree oo, genus oo)==(1,6,1)
i49 : (nudelta+q1 == q1) == false
i50 : (nudelta+q2 == q2) == false
i51 : (nudelta+q3 == q3) == false
i52 : (nudelta+q4 == q4) == false
i53 : (nudelta+q1' == q1') == false
i54 : (nudelta+q2' == q2') == false
i55 : (nudelta+q3' == q3') == false
i56 : (nudelta+q4' == q4') == false
i57 : spannudelta=ideal{nudelta_0,nudelta_1,nudelta_2,nudelta_3,nudelta_4,
      nudelta_5,nudelta_6,nudelta_7}
i58 : (dim oo -1, degree oo)==(5,1)
i59 : -- let us construct the EF 3-fold WF7
      -- as projection og WF13 from spannudelta
      proj = rationalMap toMap(spannudelta,1,1)
i60 : WF7 = proj(WF13)
i61 : (dim oo -1, degree oo)==(3,12)
i62 : -- let us see the configuration of
      -- the singular points of WF7:
      P1 = proj(q1);   -- [ 0: 0: 0: 0: 1:-1: 1: 1]
i63 : P2 = proj(q2);   -- [ 0: 0: 0: 0: 0: 0: 0: 1]
i64 : P3 = proj(q3);   -- [ 1: 0: 0:-2: 0: 0: 4: 2]
i65 : P4 = proj(q4);   -- [ 1: 0:-2: 0: 4: 0: 0: 2]
i66 : P1' = proj(q1'); -- [ 1: 0: 0: 0: 0: 0: 0: 0]
i67 : P2' = proj(q2'); -- [ 1: 2: 0:-2: 2:-2: 4: 4]
i68 : P3' = proj(q3'); -- [ 0: 0: 0: 0: 1: 0: 0: 0]
i69 : P4' = proj(q4'); -- [ 0: 0: 0: 0: 0: 0: 1: 0]
i70 : line12 = ideal{(toMap(saturate(P1*P2),1,1)).matrix};
i71 : (line12 + WF7 == line12) == true
i72 : line13 = ideal{(toMap(saturate(P1*P3),1,1)).matrix};
i73 : (line13 + WF7 == line13) == true
i74 : line14 = ideal{(toMap(saturate(P1*P4),1,1)).matrix};
i75 : (line14 + WF7 == line14) == true
i76 : line23 = ideal{(toMap(saturate(P2*P3),1,1)).matrix};
i77 : (line23 + WF7 == line23) == true
i78 : line24 = ideal{(toMap(saturate(P2*P4),1,1)).matrix};
i79 : (line24 + WF7 == line24) == true
i80 : line34 = ideal{(toMap(saturate(P3*P4),1,1)).matrix};
i81 : (line34 + WF7 == line34) == true
i82 : line11' = ideal{(toMap(saturate(P1*P1'),1,1)).matrix};
i83 : (line11' + WF7 == line11') == false
i84 : line22' = ideal{(toMap(saturate(P2*P2'),1,1)).matrix};
i85 : (line22' + WF7 == line22') == false
i86 : line33' = ideal{(toMap(saturate(P3*P3'),1,1)).matrix};
i87 : (line33' + WF7 == line33') == false
i88 : line44' = ideal{(toMap(saturate(P4*P4'),1,1)).matrix};
i89 : (line44' + WF7 == line44') == false
i90 : -- let us take the rational map
      -- defined by the linear system 
      -- of the sextics of PP3 double along the edges lij
      -- and containing the curve delta      
      sextiesX = sextiesSigma*proj
i91 : WF7 == image oo
i92 : -- base locus of sextiesX
      baseX = associatedPrimes(ideal sextiesX)
i93 : baseX#0 == delta
i94 : baseX#1 == l13
i95 : baseX#2 == l12
i96 : baseX#3 == l01
i97 : baseX#4 == l23
i98 : baseX#5 == l02
i99 : baseX#6 == l03
i100 : PP7 = ring WF7
i101 : -- let us take a general element X of the
       -- linear system of the sextics double
       -- along the edges of a trivial tetrahedron
       -- and containing the cubic curve delta:
       -- X = sextiesX^*(ideal{random(1,PP7)})
       -- for example let us take:       
       matrix sextiesX
i102 : X = ideal{2*oo_(0,0)+oo_(0,4)+oo_(0,6)-oo_(0,7)}
i103 : (dim oo -1, degree oo) == (2,6)
i104 : -- remark: its image via sextiesX is contains no
       -- singular points of WF7
       S = sextiesX(X);
i105 : (S+P1 == P1) == false
i106 : (S+P2 == P2) == false
i107 : (S+P3 == P3) == false
i108 : (S+P4 == P4) == false
i109 : (S+P1' == P1') == false
i110 : (S+P2' == P2') == false
i111 : (S+P3' == P3') == false
i112 : (S+P4' == P4') == false
i113 : -- as for a general Enriques sextic,
       -- the tangent cone to X at a vertex of the tetrahedron
       -- is the union of the three faces containing that vertex
       PP3' = ZZ/10000019[x_0..x_3];
i114 : Conev0 = tangentCone(sub(X, {(gens PP3)_0 => 1 }))
i115 : degree oo == 3
i116 : Conev1 = tangentCone(sub(X, {(gens PP3)_1 => 1 }))
i117 : degree oo == 3
i118 : Conev2 = tangentCone(sub(X, {(gens PP3)_2 => 1 }))
i119 : degree oo == 3
i120 : Conev3 = tangentCone(sub(X, {(gens PP3)_3 => 1 }))
i121 : degree oo == 3
i122 : -- the tangent cone to X at a point pij
       -- it the union of two planes containing lij:
       -- let us take a change of coordinates
       -- in order to see p01 as the point [0:0:0:1]
       sub(X, {(gens PP3)_0 => (gens PP3')_0, (gens PP3)_1 => (gens PP3')_1, 
       (gens PP3)_2 => (gens PP3')_2-(gens PP3')_3, (gens PP3)_3 => (gens PP3')_3 });
i123 : sub(oo, {(gens PP3')_3  => 1})
i124 : tangentCone oo
i125 : Conep01 = sub(oo, {(gens PP3')_0 => (gens PP3)_0, (gens PP3')_1 => (gens PP3)_1, 
       (gens PP3')_2 => (gens PP3)_2+(gens PP3)_3, (gens PP3')_3 => (gens PP3)_3 })
i126 : degree oo == 2
i127 : -- let us take a change of coordinates
       -- in order to see p02 as the point [0:0:0:1]
       sub(X, {(gens PP3)_0 => (gens PP3')_0, (gens PP3)_1 => (gens PP3')_1-(gens PP3')_3, 
       (gens PP3)_2 => (gens PP3')_2, (gens PP3)_3 => (gens PP3')_3 });
i128 : sub(oo, {(gens PP3')_3  => 1})
i129 : tangentCone oo
i130 : Conep02 = sub(oo, {(gens PP3')_0=>(gens PP3)_0, (gens PP3')_1=>(gens PP3)_1+(gens PP3)_3, 
       (gens PP3')_2=>(gens PP3)_2, (gens PP3')_3=>(gens PP3)_3 })
i131 : degree oo == 2
i132 : -- let us take a change of coordinates
       -- in order to see p03 as the point [0:0:1:0]
       sub(X, {(gens PP3)_0 => (gens PP3')_0, (gens PP3)_1 => (gens PP3')_1-(gens PP3')_2, 
       (gens PP3)_2 => (gens PP3')_2, (gens PP3)_3 => (gens PP3')_3 });
i133 : sub(oo, {(gens PP3')_2  => 1})
i134 : tangentCone oo
i135 : Conep03 = sub(oo, {(gens PP3')_0=>(gens PP3)_0, (gens PP3')_1=>(gens PP3)_1+(gens PP3)_2, 
       (gens PP3')_2=>(gens PP3)_2, (gens PP3')_3=>(gens PP3)_3 })
i136 : degree oo == 2
i137 : -- let us take a change of coordinates
       -- in order to see p12 as the point [0:0:0:1]
       sub(X, {(gens PP3)_0 => (gens PP3')_0-(gens PP3')_3, (gens PP3)_1 => (gens PP3')_1, 
       (gens PP3)_2 => (gens PP3')_2, (gens PP3)_3 => (gens PP3')_3 });
i138 : sub(oo, {(gens PP3')_3  => 1})
i139 : tangentCone oo
i140 : Conep12 = sub(oo, {(gens PP3')_0=>(gens PP3)_0+(gens PP3)_3, 
       (gens PP3')_1=>(gens PP3)_1, (gens PP3')_2=>(gens PP3)_2, (gens PP3')_3=>(gens PP3)_3 })
i141 : degree oo == 2
i142 : -- let us take a change of coordinates
       -- in order to see p13 as the point [0:0:1:0]
       sub(X, {(gens PP3)_0 => (gens PP3')_0-(gens PP3')_2, (gens PP3)_1 => (gens PP3')_1, 
       (gens PP3)_2 => (gens PP3')_2, (gens PP3)_3 => (gens PP3')_3 });
i143 : sub(oo, {(gens PP3')_2  => 1})
i144 : tangentCone oo
i145 : Conep13 = sub(oo, {(gens PP3')_0=>(gens PP3)_0+(gens PP3)_2, 
       (gens PP3')_1=>(gens PP3)_1, (gens PP3')_2=>(gens PP3)_2, (gens PP3')_3=>(gens PP3)_3 })
i146 : degree oo == 2
i147 : -- let us take a change of coordinates
       -- in order to see p23 as the point [0:1:0:0]
       sub(X, {(gens PP3)_0 => (gens PP3')_0-(gens PP3')_1, (gens PP3)_1 => (gens PP3')_1, 
       (gens PP3)_2 => (gens PP3')_2, (gens PP3)_3 => (gens PP3')_3 });
i148 : sub(oo, {(gens PP3')_1  => 1})
i149 : tangentCone oo
i150 : Conep23 = sub(oo, {(gens PP3')_0=>(gens PP3)_0+(gens PP3)_1, 
       (gens PP3')_1=>(gens PP3)_1, (gens PP3')_2=>(gens PP3)_2, (gens PP3')_3=>(gens PP3)_3 })
i151 : degree oo == 2
i152 : -- let us see that the tangent cone to X
       -- at a point of lij is a couple of planes
       -- containing lij
       
       -- let us take a point [x:0:0:y] of l12
       -- since we have already studied 
       -- the points v0=[1:0:0:0] and v3=[1:0:0:0]
       -- we can assume x and y not equal to zero
       -- so let us consider the point [a:0:0:1] 
       -- with a not equal to zero
       -- let us take a change of coordinates
       -- in order to see [a:0:0:1] as the point [0:0:0:1]
       A = ZZ/10000019[a];
i153 : R = A[s_0..s_3];
i154 : R' = A[r_0..r_3];
i155 : newX = sub(X,R)
i156 : sub(newX, {(gens R)_0 => (gens R')_0+(gens A)_0*(gens R')_3, (gens R)_1 => (gens R')_1, 
       (gens R)_2 => (gens R')_2, (gens R)_3 => (gens R')_3 })
i157 : sub(oo,(gens R')_3=>1)
i158 : sub(tangentCone oo, {(gens R')_0=>(gens R)_0-(gens A)_0*(gens R)_3, 
       (gens R')_1=>(gens R)_1, (gens R')_2 =>(gens R)_2, (gens R')_3=>(gens R)_3})
i159 : -- we obtain
       --          2 2       3               2 2
       --   ideal(a s  + (- a  - 2a)s s  + 2a s )
       --            1               1 2       2
       -- so the tangent cone to X at the point [a:0:0:1]
       -- is the union of two planes containing l12
       
       -- similarly for the points of l13,l23,l01,l02,l03
       
       -- let us study the singular locus of X
       -- in order to verify if X has other kinds of singularity
       JX = jacobian(X)
i160 : singX = minors(1,JX)+X
i161 : compSingX = associatedPrimes singX
i162 : compSingX#0 == l01
i163 : compSingX#1 == l02
i164 : compSingX#2 == l12
i165 : compSingX#3 == l23
i166 : compSingX#4 == l03
i167 : compSingX#5 == l13
i168 : -- compSingX#6 is the point [0:0:2:1]
       -- compSingX#7 is the point [0:0:1:2]       
       x' = compSingX#8 -- is a point of l01
i169 : x'' = compSingX#9 -- is the point of l01
i170 : -- remark (in order to understand x' and x''):
       sub(X, QQ[s_0..s_3])
i171 : minors(1,jacobian(oo))+oo
i172 : compSingX' = associatedPrimes oo
i173 : (associatedPrimes(sub(compSingX'#6,PP3)))#0 == x' 
i174 : (associatedPrimes(sub(compSingX'#6,PP3)))#1 == x''
       -- where compSingX'#6 is:       
       --                   2             2
       --  ideal (s , s , 2s  + 3s s  + 2s )
       --          1   0    2     2 3     3
i175 : 
       -- furthermore, if r = sqrt(2) we have that
       -- compSingX#10 are the points [r:0:0:1] and [-r:0:0:1]
       -- compSingX#11 are the points [r:0:1:0] and [-r:0:1:0]
       -- compSingX#12 are the points [0:r:0:1] and [0:-r:0:1]
       -- compSingX#13 is the point [1:1:0:0]        
       compSingX#14 == p23 
i176 : -- compSingX#15 are the points [0:r:1:0] and [0:r:1:0]              
       -- hence X has just the singularities described above       
\end{verbatim}
\end{scriptsize}

\end{document}